\documentclass[11pt, letterpaper]{amsart}
\usepackage{amsmath,amssymb,hyperref}

\usepackage{tikz}
\usetikzlibrary{cd,patterns}

\usepackage[bb=ams,scr=euler]{mathalpha}

\usepackage{thmtools,enumitem}
\declaretheoremstyle[headfont=\normalsize\normalfont\bfseries,notefont=\mdseries,
notebraces={(}{)},bodyfont=\normalfont\itshape,postheadspace=0.5em]{italstyle}

\declaretheorem[style=italstyle,name=Theorem]{theorem}
\declaretheorem[style=italstyle,name=Lemma,numberwithin=section]{lemma}

\makeatletter
\newcommand{\abs}[1]{|#1|}
\newcommand{\bd}{\partial}

\newcommand{\colim}{\mathop{\mathrm{colim}}}
\newcommand{\C}{\mathbb{C}}
\renewcommand{\d}{\mathrm{d}}
\newcommand{\id}{\mathrm{id}}
\newcommand{\intprod}{\mathbin{{\tikz{\draw(-0.1,0)--(0.1,0)--(0.1,0.2)}\hspace{0.5mm}}}}
\newcommand{\ip}[1]{\left\langle#1\right\rangle}

\newcommand{\pd}[2]{\frac{\partial #1}{\partial #2}}
\newcommand{\R}{\mathbb{R}}
\newcommand{\set}[1]{\left\{#1\right\}}

\renewcommand\section{\@startsection{section}{1}{0pt}{-3.5ex \@plus -1ex \@minus -.2ex}{2.3ex \@plus.2ex}{\centering\itshape}}

\def\@secnumfont{\normalfont\itshape}
\def\subsection{\@startsection{subsection}{2}%
  \z@{.7\linespacing\@plus\linespacing}{-.5em}%
  {\normalfont\itshape}}
\def\subsubsection{\@startsection{subsubsection}{3}%
  \z@{.5\linespacing\@plus.7\linespacing}{-.5em}%
  {\normalfont\itshape}}

\newcommand{\Z}{\mathbb{Z}}

\makeatother

\title{Remarks on eternal classes in symplectic cohomology}
\author{Dylan Cant}
\email{dylan@dylancant.ca}
\address{Institut de mathématique d'Orsay, Université Paris-Saclay, Bâtiment 307, rue Michel Magat, F-91405 Orsay Cedex, France}

\begin{document}
\begin{abstract}
  This paper studies special classes in the symplectic cohomology of a semipositive and convex-at-infinity symplectic manifold $W$. The classes under consideration lie in the image of every continuation map (for this reason, we call them eternal classes as they are never born and never die). Non-eternal classes in symplectic cohomology can be used to define spectral invariants for contact isotopies of the ideal boundary $Y$ of $W$. It is shown that the spectral invariants of non-eternal classes behave sub-additively with respect to the pair-of-pants product. This is used to define a spectral pseudo-metric on the universal cover of the group of contactomorphisms. We also give criteria for existence and non-existence of eternal classes. First, a compact monotone Lagrangian with odd Euler characteristic and minimal Maslov number at least $2$ implies the existence of non-zero eternal classes (e.g., $T^{*}\mathrm{RP}^{2n}$ has non-zero eternal classes). Second, no non-zero eternal classes exist if every compact set in $W$ is smoothly displaceable (e.g., $T^{*}T^{n}$ has no non-zero eternal classes).
\end{abstract}
\maketitle

\section{Introduction}
\label{sec:introduction}

\subsection{Eternal classes}
\label{sec:eternal-classes}

This paper is concerned with \emph{eternal classes} in symplectic cohomology. To state the definition, we first recall the framework we are working in: on a semipositive and convex-at-infinity symplectic manifold $W$ one can consider the \emph{contact-at-infinity Hamiltonian isotopies} $\psi_{t}$. Each such isotopy has a \emph{Floer cohomology} $\mathrm{HF}(\psi_{t})$, which is defined as the fixed point Floer cohomology of $\psi_{1}$. Throughout we work over the field $\Z/2$, and the ambient space $W$ is assumed to be connected.

Continuation maps from $\mathrm{HF}(\psi_{0,t})\to \mathrm{HF}(\psi_{1,t})$ are defined using \emph{continuation data}, namely squares $\set{\psi_{s,t}:(s,t)\in [0,1]^{2}}$ so that:
\begin{enumerate}
\item $\psi_{s,t}$ is an extension of $\psi_{0,t}$ and $\psi_{1,t}$,
\item the ideal restriction of $s\mapsto \psi_{s,1}$ is a non-negative path in the contactomorphism group of the ideal boundary of $W$, and,
\item $\psi_{s,0}=\id$ for all $s$.
\end{enumerate}

Continuation data can be composed in a manner similar to the concatenation of paths in such a way that $\mathrm{HF}(-)$ becomes a functor from the category whose objects are contact-at-infinity Hamiltonian systems and whose morphisms are homotopy classes of continuation data. The existence of this functorial structure is well-known in Floer theory (with varying conventions throughout the literature), and the precise formulation we consider here can be found in \cite{cant-arXiv-2023,cant-hedicke-kilgore} which we review in \S\ref{sec:revi-floer-cohom}. An important invariant extracted from this functor is its colimit, which is known as the \emph{symplectic cohomology}:
\begin{equation*}
  \mathrm{SH}(W)=\colim\mathrm{HF}.
\end{equation*}
Another important invariant is the limit, $\lim\mathrm{HF}$, whose elements are the natural transformations from $\Z/2$ to $\mathrm{HF}$.

This leads us to the main object considered in this paper. An \emph{eternal class} is an element $\mathfrak{e}\in \mathrm{SH}$ which lies in the image of the natural morphism:
\begin{equation*}
  \lim \mathrm{HF}\to \mathrm{SH};
\end{equation*}
the image of this natural morphism is denoted by $\mathrm{SH}_{e}\subset \mathrm{SH}$. The first goal of the paper is to convince the reader of the significance of eternal classes, and to prove various results about eternal classes. We begin by observing two facts which are immediate from the definition.

\subsubsection{Fully infinite bars in the Reeb flow barcode}
\label{sec:fully-infinite-bars}

Fix a Reeb vector field $R^{\alpha}$ on the ideal boundary and let $R^{\alpha}_{s}$ be its time-$s$ flow, extended arbitrarily to the filling. The persistence module $V_{s}^{\alpha}=\mathrm{HF}(R^{\alpha}_{s})$, with the aforementioned continuation maps $V_{s}^{\alpha}\to V_{s+\delta}^{\alpha}$, has a barcode decomposition. The colimit has a basis parametrized by the fully infinite bars of the form $\R$ and the half-infinite bars of the form $[a,\infty)$. It follows from the definition that $\mathrm{SH}_{e}$ is the subspace spanned by the fully infinite bars; see Lemma \ref{lemma:barcode-dcomp}.

Consequently, eternal classes in the symplectic cohomology are elements which do \emph{not} lead to a spectral invariant. We return to the discussion of spectral invariants in \S\ref{sec:spectr-invar-non-intro}.

\begin{figure}[h]
  \centering
  \begin{tikzpicture}
    \draw (-4,1)--+(4,0) (2,-1)--+(3,0);
    \draw (-5,0)--+(10,0)node[right]{$\mathfrak{e}_{2}$} (-5,0.5)--+(10,0)node[right]{$\mathfrak{e}_{1}$} (-5,-0.5)--+(10,0) node[right]{$\mathfrak{e}_{3}$};

  \end{tikzpicture}
  \caption{Eternal classes are linear combinations of the basis elements corresponding to fully infinite bars (e.g., in the figure $\mathfrak{e}_{1}+\mathfrak{e}_{2}+\mathfrak{e}_{3}$ is an eternal class).}
  \label{fig:pmod-barcode}
\end{figure}
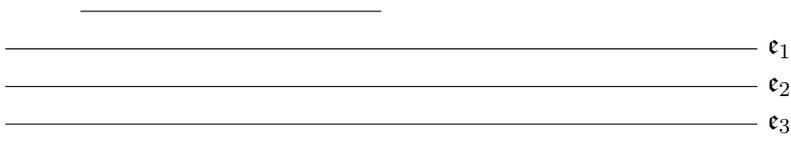

\subsubsection{Vanishing projection to Rabinowitz-Floer cohomology}
\label{sec:vanish-proj-RFH}

Because of the isomorphism in \cite{cieliebak_frauenfelder_oancea}, many modern approaches to Rabinowitz-Floer cohomology $\mathrm{RFH}(W)$ define it as a cone of the natural morphism: $$\lim \mathrm{HF}\to \colim\mathrm{HF}=\mathrm{SH}(W);$$
see, e.g., \cite{venkatesh18,venkatesh-quantitative-nature,bae-kang-kim,cieliebak-oancea-18,dahinden20}.

Therefore eternal classes in $\mathrm{SH}(W)$ project to zero in $\mathrm{RFH}(W)$ in the associated long-exact sequence for a cone; indeed, this property characterizes the eternal classes.

In other words, eternal classes realize Floer theoretic classes in $\mathrm{SH}(W)$ which are not captured by $\mathrm{RFH}(W)$. This phenomenon can be seen by comparing the work of \cite{ritter_negative_line_bundles,ritter_circle_actions}, where certain negative line bundles $W$ are shown to have non-zero symplectic cohomology, with \cite{albers-kang-vanishing-RFH} which shows negative line bundles have vanishing Rabinowitz-Floer cohomology.

\subsection{Statement of results}
\label{sec:statement-results}

We will now list some of the facts about eternal classes which we prove in this paper. A large part of the work is dedicated to exploring how eternal classes interact with the product structures on symplectic cohomology.

\subsubsection{When is the unit element eternal?}
\label{sec:when-unit-element}

In \S\ref{sec:unit-elem-mathrmsh} we recall the construction of the unit element $1\in \mathrm{SH}$. This is a distinguished element which can be characterized as the unit for the so-called pair-of-pants product which is recalled in \S\ref{sec:pair-pants-product}.

A contact-at-infinity Hamiltonian isotopy $\psi_{t}$ is said to lie in the \emph{negative cone} provided the ideal restriction of $\psi_{t}$ is a negative path.

Our first result characterizes when the unit is an eternal class:
\begin{theorem}\label{theorem:when-unit-eternal}
  Let $\psi_{t}$ lie in the negative cone. The unit lies in the image of the structure map $\mathfrak{c}:\mathrm{HF}(\psi_{t})\to\mathrm{SH}$ if and only if $1\in \mathrm{SH}_{e}$.
\end{theorem}
The rough idea of the proof is that if $\mathfrak{c}(x)=1$, then $\mathfrak{c}(x^{k})=1$ where $x^{k}\in \mathrm{HF}(\psi_{t}^{k})$ is the image of $x\otimes \dots\otimes x$ under a pair-of-pants product map:
\begin{equation*}
  \mathrm{HF}(\psi_{t})\otimes \dots\otimes \mathrm{HF}(\psi_{t})\to \mathrm{HF}(\psi_{t}^{k}).
\end{equation*}
Morally speaking, if $\psi_{t}$ is in the negative cone, then $\psi_{t}^{k}$ becomes more and more negative as $k\to\infty$. Taking a limit $k\to\infty$ will show that $1\in \mathrm{SH}_{e}$. The rigourous proof is given in \S\ref{sec:crit-unit-elem}.

One trivial example when the hypothesis holds is if $1=0\in \mathrm{SH}$. Non-trivial examples are the case of the total space of $\mathscr{O}(-1)\to \mathrm{CP}^{n}$ as proven in \cite{ritter_negative_line_bundles,ritter_circle_actions}.

One useful application of the theorem we will have occasion to use is that, if $1\not\in \mathrm{SH}_{e}$, then the unit does not lie in the image of $\mathrm{HF}(R^{\alpha}_{st})\to \mathrm{SH}$ for any negative number $s<0$. This has implications for the spectral invariants we define in \S\ref{sec:spectr-invar-non-intro}.

\subsection{The eternal elements forms an ideal}
\label{sec:claim-2}

Our second main result is:
\begin{theorem}\label{theorem:ideal-pop}
  The subspace $\mathrm{SH}_{e}\subset \mathrm{SH}$ generated by the eternal elements is an ideal with respect to the pair-of-pants product.
\end{theorem}
Consequently, $\mathrm{SH}=\mathrm{SH}_{e}$ if and only if $1\in \mathrm{SH}_{e}$. The proof is given in \S\ref{sec:subsp-etern-class}.

\subsection{On the spectral invariant of non-eternal classes}
\label{sec:spectr-invar-non-intro}

The next results concern the spectral invariants one can extract from a non-eternal element. These spectral invariants were introduced in \cite{djordjevic-uljarevic-zhang-arXiv-2023,cant-arXiv-2023} in the context of Floer theory in convex-at-infinity manifolds. Similar spectral invariants appear in \cite{albers-shelukhin-zapolsky} for ideal boundaries of negative line bundles, and in  \cite{albers-merry-JSG-2018} for spectrally finite classes in $\mathrm{RFH}$.

Let $R^{\alpha}_{s}$ be a Reeb flow, let $\varphi_{t}$ be a contact isotopy, and consider the persistence module:
\begin{equation*}
  V^{\alpha}_{s}(\varphi_{t})=\mathrm{HF}(\varphi_{t}^{-1}\circ R^{\alpha}_{s}).
\end{equation*}
The isomorphism class of the persistence module is independent of the extension of $\varphi_{t}$ and $R^{\alpha}_{s}$ to the filling $W$; see \cite{djordjevic-uljarevic-zhang-arXiv-2023,cant-arXiv-2023}.
The $V_{s_{0}}^{\alpha}\to V^{\alpha}_{s_{1}}$ structure maps in the persistence module are induced by continuation maps associated to the continuation data obtained by increasing the speed:
\begin{equation*}
  \varphi_{s,t}=\varphi_{t}^{-1}\circ R_{(1-s)s_{0}t+ss_{1}t}.
\end{equation*}
As usual with a persistence module, we can associate a \emph{spectral invariant}:
\begin{equation*}
  c_{\alpha}(\zeta;\varphi_{t}):=\inf\set{s:\zeta\in \mathrm{im}(V_{s}^{\alpha}(\varphi_{t})\to \mathrm{SH}(W))},
\end{equation*}
for any $\zeta\in \mathrm{SH}(W)$. The spectral invariants are supported on a special set of numbers and depend only on the image of $\varphi_{t}$ in the universal cover.
\begin{theorem}\label{theorem:spectral}
  Suppose that $\zeta\not\in\mathrm{SH}_{e}$, then $s=c_{\alpha}(\zeta;\varphi_{t})$ is finite and satisfies:
  \begin{equation*}
    \varphi_{1}^{-1}R_{s}^{\alpha}\text{ has a discriminant point};
  \end{equation*}
  in other words, the spectral invariant is the length of a translated point. Additionally, if $\varphi_{t}$ and $\phi_{t}$ represent the same element in the universal cover of the contactomorphism group, then $c_{\alpha}(\zeta;\varphi_{t})=c_{\alpha}(\zeta;\phi_{t}).$
\end{theorem}
\begin{proof}
  Both properties are established in \cite{cant-arXiv-2023}, and so we only recall the ideas behind their proofs. This first part follows from \cite{uljarevic-zhang-JFPTA-2022}; see also \cite{cant-arXiv-2023,djordjevic-uljarevic-zhang-arXiv-2023}. The idea is that continuation maps $V_{s}^{\alpha}(\varphi_{t})\to V_{s+\delta}^{\alpha}(\varphi_{t})$ are isomorphisms by \cite{uljarevic-zhang-JFPTA-2022} if there are no discriminant points which develop when interpolating from speed $s$ to speed $s+\delta$. The second part uses the fact that there are \emph{invertible} continuation maps $\varphi_{t}^{-1}\circ R_{st}^{\alpha}\to \phi_{t}^{-1}\circ R_{st}^{\alpha}$ which are natural with respect to speed $s$.
\end{proof}

\subsubsection{Subadditivity}
\label{sec:subadditivity}

We prove these spectral invariant are sub-additive with respect to the pair-of-pants product.
\begin{theorem}\label{theorem:sub-additive}
  For all contact isotopies $\varphi_{t},\phi_{t}$, and classes $\zeta_{0},\zeta_{1}$ we have:
  \begin{equation*}
    c_{\alpha}(\zeta_{0}\zeta_{1};\varphi_{t}\circ \phi_{t})\le c_{\alpha}(\zeta_{0};\varphi_{t})+c_{\alpha}(\zeta_{1};\phi_{t}).
  \end{equation*}
\end{theorem}
Note that spectral invariants are shown to be approximately sub-additive in \cite{djordjevic-uljarevic-zhang-arXiv-2023}, so our result is not too surprising. The proof is given in \S\ref{sec:subadd-spectr-invar}.

\subsubsection{Spectral oscillation energy}
\label{sec:spectral-oscillation}

Assuming $1\not\in \mathrm{SH}_{e}$, Theorem \ref{theorem:when-unit-eternal} implies that $c_{\alpha}(1;\id)=0$, and then Theorem \ref{theorem:sub-additive} implies the following quantity:
\begin{equation}\label{eq:spectral-oscillation-energy}
  \gamma_{\alpha}(\varphi_{t}):=c_{\alpha}(1;\varphi_{t})+c_{\alpha}(1;\varphi_{t}^{-1})
\end{equation}
is a pseudo-norm (i.e., is non-negative and sub-additive). We call this quantity the \emph{spectral oscillation energy} of $\varphi_{t}$. This designation is motivated by our next result:

\begin{theorem}\label{theorem:spectral-oscillation}
  Suppose that $1\not\in \mathrm{SH}_{e}$. Let $\varphi_{t}$ be a contact isotopy of the ideal boundary $Y$, let $\alpha$ be a choice of contact form. Then:
  \begin{equation}\label{eq:spectral-oscillation}
    0\le \gamma_{\alpha}(\varphi_{t})\le 2\inf_{s\in \R} \mathrm{dist}_{\alpha}(R^{\alpha}_{st},\varphi_{t}),
  \end{equation}
  where $\mathrm{dist}_{\alpha}$ is the Shelukhin-Hofer distance from \cite{shelukhin-JSG-2017} computed in the universal cover. In particular, $\gamma_{\alpha}(R^{\alpha}_{st})=0$ holds for all Reeb flows $R^{\alpha}_{st}$.
\end{theorem}
The quantity appearing on the right-hand side of \eqref{eq:spectral-oscillation} is the \emph{Shelukhin-Hofer oscillation energy}; see \cite{shelukhin-JSG-2017,allais-arlove,cant-hedicke-arXiv-2024}. The proof is given in \S\ref{sec:comp-with-order}.

\subsubsection{Contact isotopies with large oscillation}
\label{sec:cont-isot-sttn}

We follow \cite{nakamura-arXiv-2023,allais-arlove} which relates the oscillation of certain contact isotopies in $ST^{*}T^{n}$ to the \emph{shape invariant} studied in \cite{sikorav-duke-1989,eliashberg-JAMS-1991,eliashberg-polterovich-GAFA-2000,rosen-zhang-dedicata-2021,cant-intjmath-2023}. The precise statement is the following:
\begin{theorem}\label{theorem:shape-invar}
  Let $H:T^{*}T^{n}\to \R$ be a Hamiltonian function of the form $H(p)$, where:
  \begin{equation*}
    p:T^{*}T^{n}\to \R^{n}
  \end{equation*}
  is the projection to the cotangent fiber, and which satisfies $H(rp)=rH(p)$ for $r>0$. Let $\varphi_{t}$ be the contact isotopy obtained as the ideal restriction of the time-1 flow of $H$. Then:
  \begin{equation}\label{eq:equality}
    c_{\alpha}(1;\varphi_{t})=\max_{\abs{p}=1}H(p),
  \end{equation}
  where $\alpha=\lambda/\abs{p}$ is the contact form corresponding to the flat metric on $T^{n}$. In particular,
  \begin{equation*}
    \gamma_{\alpha}(\varphi_{t})=\max_{\abs{p}=1} H(p)-\min_{\abs{p}=1} H(p),
  \end{equation*}
  which can be made arbitrarily large.
\end{theorem}
The proof is given in \S\ref{sec:example-cont-isot}. See also \cite[Proposition 19]{shelukhin-JSG-2017} for another way to construct isotopies with large Shelukhin-Hofer oscillation.

\subsubsection{Systolic inequality and spectral invariants}
\label{sec:syst-ineq-spectr}

It is interesting to note that there is a simple argument relating the systole of a Reeb flow $R^{\alpha}$ and the spectral invariant of any positive loop $\phi_{t}$.
\begin{theorem}\label{theorem:pos-iter-sys}
  Let $\phi_{k,t}$ be a sequence of positive loops of contactomorphisms on the ideal boundary of a convex-at-infinity manifold $W$, and suppose $W$ satisfies $1\not\in\mathrm{SH}_{e}$. For any contact form $\alpha$, it holds that:
  \begin{equation*}
    c_{\alpha}(1;\phi_{1,t}\dots\phi_{k,t})\ge k\mathrm{sys}(R^{\alpha}),
  \end{equation*}
  where $\mathrm{sys}(R^{\alpha})$ is the minimal positive period of an orbit of $R^{\alpha}$. In particular, \cite[Proposition 1.6]{cant-arXiv-2023} implies $\mathrm{dist}_{\alpha}(\phi_{1,t}\dots\phi_{k,t},\id)\ge k\mathrm{sys}(R^{\alpha})$.
\end{theorem}

Similar results appear in \cite{albers-fuchs-merry-selecta-2017,sandon-unpublished-2016}. It is a consequence of monotonicity and spectrality of spectral invariants, both of which were established in \cite{cant-arXiv-2023}, and the subadditivity established in Theorem \ref{theorem:sub-additive}. The proof is short enough to be included in the introduction.
\begin{proof}
  Let $\eta_{t}^{-1}=\phi_{i,t}$ so that $\eta_{t}$ is a negative loop. Theorem \ref{theorem:spectral} implies that $c_{\alpha}(1;\eta_{t})$ is the length of an $\alpha$-translated point for $\eta_{t}$. Since $\eta_{1}=\id$, $c_{\alpha}(1;\eta_{t})$ is the period of an $\alpha$-Reeb orbit.

  It follows from \cite[Proposition 1.7]{cant-arXiv-2023} that, since $\eta_{t}$ is a negative path, $$c_{\alpha}(1;\eta_{t})<c_{\alpha}(1;\id)=0,$$ where we have used Theorem \ref{theorem:when-unit-eternal} on the right. Thus $c_{\alpha}(1;\eta_{t})\le -\mathrm{sys}(R^{\alpha})$. Apply Theorem \ref{theorem:sub-additive}:
  \begin{equation*}
    c_{\alpha}(1;\phi_{1,t}\dots\phi_{i,t})+c_{\alpha}(1;\eta_{t})\ge c_{\alpha}(1;\phi_{1,t}\dots\phi_{i-1,t}).
  \end{equation*}
  The desired result follows by induction on $i$.
\end{proof}

\subsubsection{Loops with large spectral oscillation}
\label{sec:loops-ideal-boundary}

It is also noteworthy that there are loops of contactomorphisms on $W\times T^{*}S^{1}$ with large spectral oscillation, assuming $W$ is a Liouville manifold with $\mathrm{SH}(W)\ne 0$. The construction is similar to Theorem \ref{theorem:shape-invar} in that it uses the Hamiltonian $H=p$ where $p$ is the vertical coordinate on $T^{*}S^{1}$.

\emph{Remark}. It follows from the K\"unneth formula for symplectic cohomology \cite{oancea-JSG-2008} and Theorem \ref{theorem:smooth-displaceability} below that $\mathrm{SH}_{e}(W\times T^{*}S^{1})\ne \mathrm{SH}(W\times T^{*}S^{1}),$ and hence the unit is not eternal in $W\times T^{*}S^{1}$ by Theorem \ref{theorem:ideal-pop}.

\begin{theorem}\label{theorem:systole-osc}
  Let $\alpha$ be any contact form on the ideal boundary of $W\times T^{*}S^{1}$, and let $\phi_{t}$ be the loop generated by the Hamiltonian $H=p$. Then for $k\ne 0$
  \begin{equation*}
    c_{\alpha}(1;\phi_{t}^{k})\ge \mathrm{sys}_{k}(R^{\alpha})>0,
  \end{equation*}
  where $\mathrm{sys}_{k}(R^{\alpha})$ is the smallest positive period of a Reeb orbit in the free homotopy class of the $k$th iterate of $\set{w}\times S^{1}$. In particular, for $k\ne 0$,
  \begin{equation*}
    \gamma_{\alpha}(\phi_{t}^{k})\ge \mathrm{sys}_{k}(R^{\alpha})+\mathrm{sys}_{-k}(R^{\alpha})>0,
  \end{equation*}
  and the right hand side tends to infinity as $k\to\infty$.
\end{theorem}
The proof is given in \S\ref{sec:proof-of-systole-osc}.

\subsubsection{Displaceability and the spectral invariant of the unit}
\label{sec:displ-spectr-invar}

The well-known displacement-energy inequality in symplectic geometry (see \cite[\S5.5]{hofer-zehnder-book-1994}) asserts that spectral capacity bounds the displacement energy from below. Similarly to results \cite{borman-zapolsky-GT-2015}, we have a contact analogue:
\begin{theorem}\label{theorem:unit-displace}
  Let $U\subset Y$ be an open set which is displaced by $\psi_{1}$, where $\psi_{t}$ is an isotopy with $c_{\alpha}(1;\psi_{t})\le 0$, and where $Y$ is the ideal boundary of $W$ satisfying $1\not\in \mathrm{SH}_{e}$. Then:
  \begin{equation*}
    c_{\alpha}(U)=\sup\set{c_{\alpha}(1;\varphi_{t}):\varphi_{t}\text{ is supported in }U}\le c_{\alpha}(1;\psi_{t}^{-1}).
  \end{equation*}
  If $U$ is non-empty then the left hand side is strictly positive.
\end{theorem}
The proof is given in \S\ref{sec:displ-spectr-invar-1}. One consequence of the result is a non-degeneracy type statement:
\begin{equation*}
  \psi_{1}\ne \id\implies \max\set{c_{\alpha}(1;\psi_{t}),c_{\alpha}(1;\psi_{t}^{-1})}>0.
\end{equation*}
Indeed:
\begin{equation*}
  \nu_{\alpha}(\psi_{t})=\max\set{c_{\alpha}(1;\psi_{t}),c_{\alpha}(1;\psi_{t}^{-1})},
\end{equation*}
is a pseudo-norm on the universal cover of the contactomorphism group, and the same argument used in Theorem \ref{theorem:spectral-oscillation} proves that $\nu_{\alpha}$ bounds the Shelukhin-Hofer norm from below; see \S\ref{sec:comp-with-order}.

\emph{Remark.} If $\psi_{1}$ displaces $U$, but $c_{\alpha}(1;\psi_{t})>0$, and there exists a positive loop $\phi_{t}$, then $c_{\alpha}(1;\phi_{t}^{-k}\psi_{t})\le 0$ for $k$ large enough. Clearly $\phi_{1}^{-k}\psi_{1}=\psi_{1}$ still displaces $U$. In this case, we have:
\begin{equation*}
  c_{\alpha}(U)\le c_{\alpha}(1;\psi_{1}^{-1}\phi_{1}^{k})\le c_{\alpha}(1;\psi_{-1})+c_{\alpha}(1;\phi_{1}^{k}).
\end{equation*}
In particular, the spectral capacity $c_{\alpha}(U)$ of any displaceable open set $U$ is bounded assuming $Y$ admits a positive loop of contactomorphisms (e.g., $W=T^{*}S^{n}$). Arguing in a similar vein, we can prove:
\begin{theorem}\label{theorem:disjoin}
  Suppose that $\psi_{t}$ is a contact isotopy of the ideal boundary $Y$ such that $\psi_{1}$ displaces $U$ from $R^{\alpha}_{s}(U)$ for every $s\in \R$ (i.e., $\psi_{1}$ displaces $U$ from its Reeb flow trace). Then:
  \begin{equation*}
    c_{\alpha}(U)\le \gamma_{\alpha}(\psi_{t}).
  \end{equation*}
  In particular, if $\psi_{1}$ displaces a non-empty open set from its Reeb flow trace, then its spectral oscillation is strictly positive.
\end{theorem}
\begin{proof}[Proof of Theorem \ref{theorem:disjoin}]
  We use Theorem \ref{theorem:unit-displace}. Define $\sigma=-c_{\alpha}(1;\psi_{t})$. Then a straightforward manipulation yields $c_{\alpha}(R_{\sigma t}^{\alpha}\psi_{t})=0$. Thus $R_{\sigma t}^{\alpha}\psi_{t}$ satisfies the hypotheses of Theorem \ref{theorem:unit-displace}. In particular:
  \begin{equation*}
    c_{\alpha}(U)\le c_{\alpha}(1;\psi_{t}^{-1}\circ (R^{\alpha}_{\sigma t})^{-1})\le c_{\alpha}(1;\psi_{t}^{-1})-\sigma=\gamma_{\alpha}(\psi_{t}),
  \end{equation*}
  as desired.
\end{proof}

Combining Theorem \ref{theorem:unit-displace} with Theorem \ref{theorem:pos-iter-sys} yields the following:
\begin{theorem}
  If $U\subset Y$ is an open set, where $Y$ is the ideal boundary of a convex-at-infinity manifold $W$ with $1\not\in\mathrm{SH}_{e}$, and there is a non-constant non-negative loop of contactomorphisms $\phi_{t}$ supported in $U$, then:
  \begin{equation*}
    c_{\alpha}(U)=\infty
  \end{equation*}
  and so $U$ cannot be displaced by a contact isotopy.
\end{theorem}
\begin{proof}
  By the ergodic trick of \cite{eliashberg-polterovich-GAFA-2000}, the iterate $\phi_{t}^{\ell}$ is represented in the universal cover of the contactomorphism group by a strictly positive loop for $\ell$ sufficiently large. Thus, by Theorem \ref{theorem:pos-iter-sys} we have:
  \begin{equation*}
    c_{\alpha}(1;\phi_{t}^{k\ell})\ge k\mathrm{sys}_{\alpha}(R^{\alpha}).
  \end{equation*}
  This proves $c_{\alpha}(U)=\infty$. Moreover, $Y$ clearly admits a positive loop (one can take the aforementioned loop in the same homotopy class as $\phi_{t}^{\ell}$), and hence one can use the remark following Theorem \ref{theorem:unit-displace} to conclude that $U$ is not displaceable.
\end{proof}

\subsubsection{Conjugation invariant measurements}
\label{sec:conj-invar-meas}

As communicated to the author by Baptiste Serraille, when the Reeb flow is $1$-periodic, the spectral invariants should have conjugation invariant integer parts, similarly to the results of \cite{sandon-ann-inst-four-2011}. In this direction, we will show:
\begin{theorem}\label{theorem:conjugation}
  Suppose that $\zeta\in \mathrm{SH}$, and suppose that $\phi_{t}$ is a positive loop. Then, for any other contact isotopy $\varphi_{t}$, the quantity:
  \begin{equation*}
    \ell_{\phi_{t}}(\zeta;\varphi_{t})=\inf\set{k\in \Z:\mathrm{HF}(\varphi_{t}^{-1}\circ \phi_{t}^{k})\to \mathrm{SH}\text{ hits }\zeta},
  \end{equation*}
  is conjugation invariant, i.e., $\ell(g\varphi_{t}g^{-1})=\ell(\varphi_{t})$ for all $g\in \mathrm{Cont}_{0}(Y)$. In particular, if $\phi_{t}=R^{\alpha}_{t}$ and $R^{\alpha}_{t}$ is $1$-periodic, then $\ell_{\phi_{t}}(\zeta;\varphi_{t})=\lceil c_{\alpha}(\zeta;\varphi_{t})\rceil$ is conjugation invariant, in the above sense.
\end{theorem}

\subsubsection{Comparison with order measurements}
\label{sec:comp-with-order}

In \cite{nakamura-arXiv-2023,allais-arlove} the following measurements are defined assuming the universal cover of the group of contactomorphisms is orderable:
\begin{equation*}
  \begin{aligned}
    c_{-}^{\alpha}(\varphi_{t})&:=\sup\set{s:R_{st}\le \varphi_{t}}\\
    c_{+}^{\alpha}(\varphi_{t})&:=\inf\set{s:\varphi_{t}\le R_{st}}
  \end{aligned}
\end{equation*}
The order relation is simple to state in the framework of our paper; the relation $\varphi_{0,t}\le \varphi_{1,t}$ holds if and only if there exists a continuation data $\varphi_{s,t}$ relating them (the extension to the filling is irrelevant in this discussion).

We have the following comparison with the spectral invariant of the unit:
\begin{theorem}\label{theorem:order-comp}
  If $1\not\in \mathrm{SH}_{e}$, then:
  \begin{equation*}
    -\infty<c_{-}^{\alpha}(\varphi_{t})\le c_{\alpha}(1;\varphi_{t})\le c_{+}^{\alpha}(\varphi_{t})<\infty
  \end{equation*}
  for every contact isotopy $\varphi_{t}$ of the ideal boundary.
\end{theorem}
The proof is straightforward and is given in \S\ref{sec:comp-with-order-1}.

One advantage of $c_{\alpha}(1;\varphi_{t})$ compared to $c_{\pm}^{\alpha}$ is that the spectral invariant is the length of a translated point of $\varphi_{t}$ (Theorem \ref{theorem:spectral}). It is conjectured that the order measurements also satisfy this, but this conjecture remains open at the time of writing.

We also note that an order based oscillation and pseudo-norm is defined in \cite[\S4]{allais-arlove} via:
\begin{equation*}
  \begin{aligned}
    \gamma_{+}^{\alpha}(\varphi_{t})&=c_{+}^{\alpha}(\varphi_{t})+c_{+}^{\alpha}(\varphi_{t}^{-1}),\\
    \nu_{+}^{\alpha}(\varphi_{t})&=\max\set{c_{+}^{\alpha}(\varphi_{t}),c_{+}^{\alpha}(\varphi_{t}^{-1}},
  \end{aligned}
\end{equation*}
similarly to our $\gamma_{\alpha}(\varphi_{t})$ and $\nu_{\alpha}(\varphi_{t})$. It is shown in \cite[\S4]{allais-arlove} that $\gamma_{\alpha}^{+}$ bounds the Shelukhin-Hofer oscillation from below, while $\nu_{\alpha}^{+}(\varphi_{t})$ bounds the Shelukhin-Hofer norm from below. Our Theorem \ref{theorem:order-comp} shows that $\gamma_{\alpha}\le \gamma_{\alpha}^{+}$ and $\nu_{\alpha}\le \nu_{\alpha}^{+}$. This completes the proof of Theorem \ref{theorem:spectral-oscillation}.\hfill$\square$

\emph{Remark.} It is also possible to give a direct proof of Theorem \ref{theorem:spectral-oscillation} using \cite[Proposition 1.6]{cant-arXiv-2023}, Theorem \ref{theorem:when-unit-eternal}, and various properties about the Shelukhin-Hofer distance from \cite{shelukhin-JSG-2017}, although we leave the details of such an argument to the reader.

\emph{Remark.} In the context of Theorem \ref{theorem:shape-invar} on contact isotopies $\varphi_{t}$ of $ST^{*}T^{n}$ generated by Hamiltonians $H(p)$, the result \cite[Proposition 4.11]{allais-arlove} yields: $$c_{+}^{\alpha}(\varphi_{t})=\max_{\abs{p}=1}H(p),$$ so in this case we have $c_{\alpha}(1;\varphi_{t})=c_{+}^{\alpha}(\varphi_{t})$.

\subsection{Non existence of eternal classes}
\label{sec:vanishing-eternal}

Our next result is a topological criterion for the non-existence of eternal classes:
\begin{theorem}\label{theorem:smooth-displaceability}
  If $W$ is any semipositive and convex-at-infinity manifold such that every compact subset is smoothly displaceable then $\mathrm{SH}_{e}(W)=0$.
\end{theorem}
One should note that the result does not assume Hamiltonian displaceability. E.g., if $W=T^{*}L$ and $L$ has Euler characteristic zero then $\mathrm{SH}_{e}=0$. Similarly, $W\times T^{*}S^{1}$, where $W$ is a Liouville manifold, always has $\mathrm{SH}_{e}=0$.

The method used to prove Theorem \ref{theorem:smooth-displaceability} is to show that the continuation map $\mathrm{HF}(\varphi_{0,t})\to \mathrm{HF}(\varphi_{1,t})$ vanishes if $\varphi_{0,t},\varphi_{1,t}$ lie in the negative, positive cones, respectively.

\subsection{Lagrangians and non-zero eternal classes}
\label{sec:cotang-bundl-with}

The next result shows that closed Lagrangians can sometimes be used to produce non-zero eternal classes.
\begin{theorem}\label{theorem:odd-euler-char}
  If $L\subset W$ is a compact monotone Lagrangian submanifold with minimal Maslov number at least $2$ and with odd Euler characteristic, then the class in $\mathrm{SH}(W)$ obtained from the PSS construction applied to $L$ is a non-zero eternal class.
\end{theorem}
In particular the result applies to $W=T^{*}L$ when $L$ has odd Euler characteristic. Here \emph{monotone} means that $\omega=c\mu$ on $\pi_{2}(W,L)$ for some $c\ge 0$.

In fact the argument we will give easily generalizes to:
\begin{theorem}\label{theorem:odd-euler-char-1}
  If there exist two compact monotone transversally intersecting Lagrangians $L',L\subset W$ with minimal Maslov numbers at least $2$, and the intersection number $\#(L'\cap L)$ is odd, then the PSS class of $L$ is non-zero and lies in $\mathrm{SH}_{e}$.
\end{theorem}

The result applies in the case when $W$ is a surface of positive genus with at least one puncture (with a symplectic structure which is convex-at-infinity).

\emph{Remark.} It seems to be an interesting question whether Theorem \ref{theorem:odd-euler-char-1} can be generalized to include more general Lagrangians. For instance, do the results of \cite{ritter_negative_line_bundles} (on the eternal classes in certain negative line bundles) follow from the existence of certain Lagrangian submanifolds? Let us note that (semipositive) negative line bundles always satisfy $\mathrm{SH}=\mathrm{SH}_{e}$; see \S\ref{sec:extens-posit-loops}.

\subsection{On the quantum homology product}
\label{sec:quantum-homology}

One interesting perspective on eternal classes is their relationship to the quantum homology product. Briefly, the idea is that there is a factorization of algebras:
\begin{equation}\label{eq:QH-factor}
  \lim \mathrm{HF}\to \mathrm{QH}_{*}\to \mathrm{QH}^{*}\to \mathrm{SH}=\colim \mathrm{HF};
\end{equation}
where:
\begin{enumerate}
\item $\mathrm{QH}_{*}$ is the \emph{quantum homology} (compact cycles with the quantum intersection product), and
\item $\mathrm{QH}^{*}$ is the \emph{quantum cohomology} (properly embedded cycles with the quantum intersection product).
\end{enumerate}
Detailed discussion can be found in \cite{ritter_negative_line_bundles,ritter_circle_actions}; see also the discussion in \cite{albers-kang-vanishing-RFH}. One can show, e.g., that the existence of a non-zero eternal unit element implies the quantum homology product is deformed. More generally, if $\mathrm{SH}_{e}$ has an element $x$ such that $x^{d}\ne 0$ for every $d=1,2,\dots$, then the quantum intersection product on $\mathrm{QH}_{*}$ is deformed. This is because every element in $\mathrm{QH}_{*}$ is nilpotent with the non-quantum intersection product.

We should also note that our Theorem \ref{theorem:smooth-displaceability} can be understood in terms of the factorization \eqref{eq:QH-factor}: in a convex-at-infinity manifold, the quantum Poincaré duality map $\mathrm{QH}_{*}\to \mathrm{QH}^{*}$ vanishes if every compact set is displaceable.

\subsection{Extensible positive loops}
\label{sec:extens-posit-loops}

The argument given in \cite{cant-hedicke-kilgore} and in \cite{ritter_negative_line_bundles,ritter_circle_actions,merry-uljarevic-israeljmath-2019} proves:
\begin{theorem}[{see \cite[\S1.2.6]{cant-hedicke-kilgore}}]
  If there exists a contact-at-infinity Hamiltonian isotopy $\phi_{t}$ such that $\phi_{1}=\phi_{0}$ and whose ideal restriction is a positive loop, then $\mathrm{SH}_{e}=\mathrm{SH}$.\hfill$\square$
\end{theorem}
The argument actually shows that the map $\lim \mathrm{HF}\to \mathrm{SH}$ is an isomorphism, and hence $\mathrm{RFH}$ vanishes in the presence of such a positive extensible loop, generalizing the vanishing result for negative line bundles proved in \cite{albers-kang-vanishing-RFH}. Interestingly enough, the fact that $\mathrm{SH}=\mathrm{SH}_{e}$ holds for negative line bundles means our spectral invariants cannot be applied to them. This should be contrasted with \cite{albers-shelukhin-zapolsky} which constructs spectral invariants for contact isotopies of ideal boundaries of negative line bundles.

\subsection{Acknowledgements}
\label{sec:acknowledgements}

This project benefitted from insightful discussions with many people. The author wishes to thank Habib Alizadeh, Pierre-Alexandre Arlove, Marcelo Atallah, Paul Biran, Filip Brocic, Octav Cornea, Georgios Dimitroglou-Rizell, Jakob Hedicke, Eric Kilgore, Yong-Geun Oh, Baptiste Serraille, Egor Shelukhin, Vuka\v{s}in Stojisavljevi\'c, Igor Uljarevi\'c, and Jun Zhang. The author was supported in this research by funding from the Fondation Courtois and the ANR project CoSy.

\tableofcontents

\section{Review of the Floer cohomology of a contact isotopy}
\label{sec:revi-floer-cohom}

In this section we recall the Floer cohomology for contact-at-infinity isotopies of the ideal boundary $Y$ of a semipositive and convex-at-infinity symplectic manifold $W$. The rough outline of the construction is fairly standard, and the specific focus on contact-at-infinity isotopies is described in various settings in, e.g., \cite{djordjevic-uljarevic-zhang-arXiv-2023,cant-arXiv-2023,cant-hedicke-kilgore}. One difference in the present work is that we use Novikov coefficients, in a manner similar to \cite{ritter-2009,ritter_negative_line_bundles,ritter_circle_actions}; the reason for using these coefficients is to ensure various sums converge.

\subsection{Novikov field}
\label{sec:novikov-field}

Let $\Lambda$ denote the universal Novikov field over $\Z/2$. Recall that this means that $\Lambda$ is the set of functions $a:\R\to \Z/2$ so that $a$ has finite support in $(-\infty,L)$ for each $L$. For $a\in \R$, one defines an elementary field element to be any function $\tau^{a}:\R\to \Z/2$ so that $\tau^{a}(a')=1$ if $a=a'$ and zero otherwise. Then every $\lambda\in \Lambda$ can be written as an infinite sum:
\begin{equation*}
  \textstyle\lambda=\sum_{i} \tau^{a_{i}},
\end{equation*}
where $a_{i}\to \infty$ as $i\to\infty$. Addition is given by addition of functions, and multiplication is given by discrete convolution, i.e.,
\begin{equation*}
  (\lambda_{1}\lambda_{2})(a)=\sum_{b+c=a}\lambda(b)\lambda'(c);
\end{equation*}
this definition of multiplication agrees with usual one in terms of formal power series; this description can be found in, e.g., \cite[pp.\,4]{hutchings-zeta-24}.

\subsection{Floer cohomology of a contact-at-infinity Hamiltonian isotopy}
\label{sec:floer-cohom-cont}

Let us recall that a convex-at-infinity manifold $W$ can be presented as the completion of a compact symplectic manifold with a convex contact type boundary. This defines a vector field $Z$ on the non-compact end called the Liouville vector field. A contact-at-infinity Hamiltonian isotopy is one whose flow commutes with the flow by $Z$, outside of a compact set.

See, e.g., \cite{brocic-cant-JFPTA-2024,cant-arXiv-2023,alizadeh-atallah-cant} and \S\ref{sec:groups-of-ham-diff} below for further details on this class of Hamiltonian systems.

\subsubsection{Floer chain vector space}
\label{sec:floer-chain-vector}

Suppose that $\psi_{t}:W\to W$ is a Hamiltonian isotopy which is contact-at-infinity and suppose that $\psi_{1}$ has non-degenerate fixed points. The associated Floer chain vector space $\mathrm{CF}(\psi_{t})$ is the free $\Lambda$-vector space generated by the finite set of fixed points of $\psi_{1}$.

\subsubsection{On the choice of almost complex structure}
\label{sec:choice-almost-compl}

Fix an $\omega$-tame almost complex structure $J$ which is equivariant with respect to the Liouville flow in the non-compact end of $W$. Standard a priori estimates imply that all non-constant holomorphic spheres are contained in a fixed compact set of $W$; see, e.g., \cite[\S2.9]{alizadeh-atallah-cant}.

Because we use the semipositivity framework for ensuring compactness, we also suppose that the moduli space of simple $J$-holomorphic spheres $\mathscr{M}^{*}(J)$ is cut transversally, in the sense of, e.g., \cite{mcduffsalamon}.

Throughout we will keep $J$ fixed; we will perturb the Hamiltonian systems to ensure the various moduli spaces in Floer theory are cut transversally.\footnote{One caveat, in \S\ref{sec:non-zero-eternal} we will assume that $J$ is chosen generically enough that the moduli space of simple holomorphic disks on a compact Lagrangian is also cut transversally.} For this reason, we will typically suppress $J$ from the notation in moduli spaces.

\subsubsection{Floer differential moduli spaces}
\label{sec:floer-differential-moduli}

For any contact-at-infinity Hamiltonian isotopy $\psi_{t}$ whose time-1 map has non-degenerate fixed points, one can consider the space $\mathscr{M}(\psi_{t})$ of finite energy\footnote{The \emph{energy} in this case is the integral of $\omega(\bd_{s}u,\bd_{t}u-X_{t}(u))$ over the cylinder. We always assume such energy integrals are finite.} solutions to:
\begin{equation*}
  \left\{
    \begin{aligned}
      &u:\R\times \R/\Z\to W,\\
      &\bd_{s}u+J(u)(\bd_{t}u-X_{t}(u))=0,
    \end{aligned}
  \right.
\end{equation*}
Here $X_{t}\circ \psi_{\beta(3t-1)}=\bd_{t}\psi_{\beta(3t-1)},$ and $\beta(t)$ is a standard cut-off function satisfying $\beta(t)=0$ for $t\le 0$ and $\beta(t)=1$ for $t\ge 1$. The equation is a time-reparametrized version of the usual Floer's equation. To avoid over-complicating the notation, we use the notation $X_{t}$ for the generator of the isotopy $\psi_{\beta(3t-1)}$.

Let us define the \emph{evaluation map}:
\begin{equation}\label{eq:evaluation}
  (u,t)\in \mathscr{M}(\psi_{t})\times \R/\Z\to u(0,t);
\end{equation}
the transversality statement we require for ensuring compactness (up to breaking) of the moduli spaces used to define the Floer differential is:
\begin{lemma}\label{lemma:transversality-floer}
  Fix a pseudocycle $\Sigma$ in $W$. For a generic compactly supported system $\delta_{t}$, the moduli space $\mathscr{M}(\psi_{t}\delta_{t})$ is cut transversally, in the sense that the linearized operator is surjective at every solution, and the evaluation \eqref{eq:evaluation} is transverse to $\Sigma$.
\end{lemma}
\begin{proof}
  This follows from standard transversality results, e.g., \cite{floer_hofer_salamon_transversality,mcduffsalamon}. Observe that the perturbed equation changes in the following way:
  \begin{equation*}
    \bd_{s}u+J(u)(\bd_{t}u-X_{t}(u)-Y_{t}(\psi_{t}^{-1}(u)))=0,
  \end{equation*}
  where $Y_{t}$ is an arbitrary compactly Hamiltonian vector field which is compactly supported where $t\in (1/3,2/3)$. The total linearized operator is thus of the form:
  \begin{equation}\label{eq:total-lin-op}
    (\eta,Y)\mapsto D_{u}(\eta)-Y_{t}(\psi_{t}^{-1}(u)).
  \end{equation}
  There are two cases to consider. Either $\bd_{s}u=0$ holds identically, in which case it is known that $D_{u}$ is surjective, or $\bd_{s}u(s_{0},t_{0})\ne 0$ holds in at least one point. In this case, following, e.g., \cite{floer_hofer_salamon_transversality} (see also \cite[\S4]{brocic-cant-JFPTA-2024} and the references therein), there exists an open and dense set of points $(s_{0},t_{0})$ such that:
  \begin{enumerate}[label=(\roman*)]
  \item $\bd_{s}u(s_{0},t_{0})\ne 0$,
  \item $u(s,t_{0})=u(s_{0},t_{0})\implies s=s_{0}$.
  \end{enumerate}
  Using these points, one can define $Y_{t}$ as a bump function localized near the point $\psi_{t_{0}}^{-1}(u(s_{0},t_{0}))$ and $t=t_{0}$ and prove surjectivity of the linearized operator.

  Standard arguments then imply that a generic choice of $Y_{t}$ will define an equation whose linearized operator is surjective at all solutions.

  One subtlety in the argument is that one needs to exclude the case that $u(s+\sigma,t)=u(s,t)$ holds for some $\sigma$. This case is excluded using the finite energy condition.\footnote{It seems to be an interesting question whether transversality can be achieved for solutions of infinite energy.}

  The condition that the evaluation map is transverse to the pseudocycle $\Sigma$ can be simultaneously achieved. The trick is to show that the universal evaluation:
  \begin{equation*}
    \mathrm{ev}_{t}:(u,Y_{t})\mapsto u(0,t)\in W
  \end{equation*}
  is a submersion \emph{for every $t$}; see, e.g., \cite[Theorem 3.1]{hofer-salamon-95}. In particular, the universal total evaluation $\mathrm{ev}(u,Y_{t},t)=u(0,t)$ is transversal to $\Sigma$. Then, for a generic $Y_{t}$, the total evaluation is also transversal to $\Sigma$.
\end{proof}

Often we will simply say $\psi_{t}$ is a generic system, implicitly replacing $\psi_{t}$ by the perturbation $\psi_{t}\delta_{t}$.

By taking the pseudocycle $\Sigma$ to be the collection of points passing through simple $J$-holomorphic spheres (which defines a pseudochain because of the semipositivity assumption), one can preclude bubbling and therefore ensure compactness-up-to-breaking. For further discussion of how semipositivity is used in this manner, we defer to \cite{hofer-salamon-95,mcduffsalamon,alizadeh-atallah-cant}.

\subsubsection{Chern classes and Conley-Zehnder indices}
\label{sec:chern-classes-conley}

Let $\mathfrak{s}$ be a section of the determinant line bundle $\det_{\C}(TW)$ which is non-vanishing along each orbit $\gamma$ of the system $\psi_{t}$. Because $\mathfrak{s}|_{\gamma}$ is non-vanishing, it determines a canonical homotopy class of symplectic trivializations of $TW|_{\gamma}$. With respect to these trivializations, the linearization of the flow of $\psi_{t}$ along $\gamma$ is associated a \emph{Conley-Zehnder} index $\mathrm{CZ}_{\mathfrak{s}}(\gamma)$, as in, e.g., \cite{cant-thesis-2022}.

It is important to note that the Poincaré dual of zero set $\mathfrak{s}^{-1}(0)$ represents the Chern class. Indeed, it represents a lift of the Chern class to the cohomology of $W$ relative the set of orbits of $\psi_{t}$.

\subsubsection{The Floer differential}
\label{sec:floer-differential-1}

For a generic system $\psi_{t}$, counting the one-dimensional components of $\mathscr{M}(\psi_{t})$ defines a $\Lambda$-linear differential:
\begin{equation*}
  d:\mathrm{CF}(\psi_{t})\to \mathrm{CF}(\psi_{t}),
\end{equation*}
by the cohomological formula:
\begin{equation*}
  d(\tau^{b}y)=\sum_{x,a} \#(\mathscr{M}_{1}(x;a;y;\psi_{t})/\R)\cdot \tau^{a+b}x,
\end{equation*}
where $\mathscr{M}_{1}(x;a;y;\psi_{t})$ is the component of $\mathscr{M}(\psi_{t})$ satisfying:
\begin{enumerate}
\item $\gamma_{-}(0)=x$, $\gamma_{+}(0)=y$,
\item $\omega(u)=a$,
\item $\mathrm{CZ}_{\mathfrak{s}}(\gamma_{+})-\mathrm{CZ}_{\mathfrak{s}}(\gamma_{-})+2\mathfrak{s}^{-1}(0)\cdot[u]=1$.
\end{enumerate}

where $\gamma_{\pm}$ are the asymptotic orbits and $\omega(u)$ is the symplectic area of $u$. For generic system, the semipositivity assumptions imply that this sum converges, and that $d^{2}=0$.

\subsubsection{Continuation data}
\label{sec:continuation-data}

Continuation data refers to a square $\psi_{s,t}$ in the space of contact-at-infinity Hamiltonian diffeomorphisms so that:
\begin{enumerate}
\item $\psi_{s,0}=\id$,
\item $\psi_{s,1}$ has a non-negative ideal restriction.
\end{enumerate}
The second condition means that $\lambda(\bd_{s}\psi_{s,1}(x))\ge 0$ holds outside of a compact set, where $\lambda$ is the Liouville form in the non-compact end of $W$. We refer the reader to \cite{cant-arXiv-2023,cant-hedicke-kilgore} for further discussion.

\subsubsection{Continuation cylinders}
\label{sec:cont-cylind}

Let $\psi_{s,t}$ be continuation data, and suppose that $\psi_{0,t}$ and $\psi_{1,t}$ are generic so that the moduli spaces $\mathscr{M}(\psi_{i,t})$, $i=0,1$, can be used to define the Floer differential.

Consider the family $\xi_{s,t}$ defined for $s\in \R$ and $t\in [0,1]$ given by:
\begin{equation*}
  \xi_{s,t}=\psi_{\beta(1-s),\beta(3t-1)},
\end{equation*}
where $\beta$ is a standard non-decreasing cut-off function.

Note that $\xi_{s,t}$ satisfies the following properties:
\begin{enumerate}
\item $\xi_{s,t}=\psi_{0,\beta(3t-1)}$ for $s\ge 1$,
\item $\xi_{s,t}=\psi_{1,\beta(3t-1)}$ for $s\le 0$,
\item $\xi_{s,t}=\xi_{s,1}$ for $t\ge 2/3$,
\item $\xi_{s,t}=\id$ for $t\le 1/3$.
\end{enumerate}
Introduce the generators:
\begin{equation*}
  Y_{s,t}\circ \xi_{s,t}=\bd_{s}\xi_{s,t}\text{ and }X_{s,t}\circ \xi_{s,t}=\bd_{t}\xi_{s,t},
\end{equation*}
noting that $X_{+\infty,t}$ is the generator for $\psi_{0,\beta(3t-1)}$ (the input system) and $X_{-\infty,t}$ is the generator for $\psi_{1,\beta(3t-1)}$ (the output system).

Define $\mathscr{M}(\psi_{s,t})$ to be the moduli space of finite energy solutions to:
\begin{equation}\label{eq:cont-map}
  \left\{
    \begin{aligned}
      &u:\R\times \R/\Z\to W,\\
      &\bd_{s}u-\rho(t)Y_{s,t}(u)+J(u)(\bd_{t}u-X_{s,t}(u))=0,\\
    \end{aligned}
  \right.
\end{equation}
where $\rho(t)=\beta(3-3t)$.

The fact that $\psi_{s,t}$ is continuation data will imply that solutions to \eqref{eq:cont-map} satisfy an a priori energy estimate in terms of their asymptotics and the homology class of the connecting cylinder. Indeed, we have:
\begin{lemma}\label{lemma:energy-estimate}
  Let $u\in \mathscr{M}(\psi_{s,t})$. Then:
  \begin{equation*}
    0\le \int \omega(\bd_{s}u-\rho(t)Y_{s,t}(u),\bd_{t}u-X_{s,t}(u))\le \omega(u)+\mathrm{const}(\psi_{s,t});
  \end{equation*}
  here the constant term is independent of $u$ and depends only on the values of certain Hamiltonian functions on the compact part of $W$.
\end{lemma}
\emph{Remark.} The integral appearing in the lemma is (by definition) the energy of a solution to the continuation map equation.

\begin{proof}
  This follows from the computation in \cite[\S2.2.5]{cant-hedicke-kilgore}. We review the computation in a slightly different context in \S\ref{sec:param-moduli-space} below.
\end{proof}

In order for the moduli spaces to satisfy the required compactness results, it is necessary to again appeal to transversality and the semipositivity condition. The required statement is the following:
\begin{lemma}\label{lemma:transversality-continuation}
  Let $\psi_{s,t}$ be continuation data. For a generic perturbation of the form $\psi_{s,t}\delta_{s,t}$ where $\delta_{s,t}$ is compactly supported in $W$ and $\delta_{s,t}=\id$ for $s=0,1$ and $t=0$, the moduli space $\mathscr{M}(\psi_{s,t}\delta_{s,t})$ is cut transversally. Moreover, for any fixed pseudochain $\Sigma$ in $W$, a generic perturbation of the above form will ensure the evaluation:
  \begin{equation*}
    (u,s,t)\in \mathscr{M}(\psi_{s,t}\delta_{s,t})\times \R\times \R/\Z\mapsto u(s,t)
  \end{equation*}
  is transverse to $\Sigma$.
\end{lemma}
\begin{proof}
  The proof is once again based on standard transversality techniques, similar (and in theory simpler) than the argument in Lemma \ref{lemma:transversality-floer}.
\end{proof}

As in the case of the Floer differential, the required compactness-up-to-breaking then holds for generic continuation data (i.e., we replace $\psi_{s,t}$ by the small perturbation $\psi_{s,t}\delta_{s,t}$), because the relevant evaluation maps miss all the simple $J$-holomorphic spheres.

\subsubsection{Continuation maps}
\label{sec:continuation-maps}

For a generic continuation data $\psi_{s,t}$, counting the zero-dimensional components of $\mathscr{M}(\psi_{s,t})$ defines a $\Lambda$-linear map:
\begin{equation*}
  \mathfrak{c}:\mathrm{CF}(\psi_{0,t})\to \mathrm{CF}(\psi_{1,t}),
\end{equation*}
similarly to the differential, by the cohomological formula:
\begin{equation*}
  \mathfrak{c}(\tau^{b}y)=\sum \#\mathscr{M}_{0}(x;a;y;\psi_{s,t})\tau^{a+b}x,
\end{equation*}
where $\mathscr{M}_{0}(x;a;y;\psi_{s,t})$ is the component of solutions where:
\begin{enumerate}
\item $\gamma_{-}(0)=x$, $\gamma_{+}(0)=y$,
\item $\omega(u)=a$,
\item $\mathrm{CZ}_{\mathfrak{s}}(\gamma_{+})-\mathrm{CZ}_{\mathfrak{s}}(\gamma_{-})+2\mathfrak{s}^{-1}(0)\cdot[u]=0$.
\end{enumerate}
The above energy bounds and the semipositivity assumptions imply that this sum converges, and that $\mathfrak{c}$ defines a chain map with respect to the Floer differential $d$; see, e.g., \cite{hofer-salamon-95} for further discussion.

The following standard lemma asserts that the resulting maps depend only on the homotopy class of the continuation data.
\begin{lemma}\label{lemma:deformation}
  Let $\psi_{\sigma,s,t}$, $\sigma\in [0,1]$, be a 1-parameter family of continuation data such that:
  \begin{enumerate}
  \item $\psi_{\sigma,0,t}=\psi_{0,t}$ and $\psi_{\sigma,1,t}=\psi_{1,t}$ are both in $\mathscr{C}^{\times}$,
  \item $\psi_{i,s,t}$, $i=0,1$, are generic for defining the continuation map: $$\mathfrak{c}_{i}:\mathrm{HF}(\psi_{0,t})\to\mathrm{HF}(\psi_{1,t}),$$ i.e., the relevant moduli spaces are cut transversally,
  \end{enumerate}
  then for a generic perturbation $\delta_{\sigma,s,t}$ so that $\delta_{\sigma,s,t}=\id$ for $\sigma=0,1$, $s=0,1$, and $t=0$, the parametric moduli space of pairs $(\sigma,u)$ where: $$u\in \mathscr{M}(\psi_{\sigma,s,t}\delta_{\sigma,s,t})$$
  is cut transversally and the 0- and 1-dimensional components are compact up to breaking of Floer cylinders. Consideration of 1-dimensional components proves the algebraic relation: $$\mathfrak{c}_{1}-\mathfrak{c}_{0}=\d \mathrm{K}+\mathrm{K}\d,$$ where $\mathrm{K}$ is an appropriate count of the 0-dimensional components of the parametric moduli space.
\end{lemma}
\begin{proof}
  Since we have already established the relevant a priori estimates, the argument is standard Floer theory; see, e.g., \cite[Theorem 4]{floer-comm-math-phys-1989}, \cite[Lemma 6.3]{salamon-zehnder-comm-pure-appl-math-1992}, \cite[Lemma 6.13]{abouzaid-monograph-2015}, \cite[Theorem 5.2]{hofer-salamon-95}. See also Lemma \ref{lemma:CH-lag} below for a related argument.
\end{proof}

Furthermore, if $\psi_{s,t}$ and $\eta_{s,t}$ are both continuation data with $\psi_{1,t}=\eta_{0,t}$, then we can concatenate the continuation data $\psi_{s,t}\#\eta_{s,t}$ by gluing the squares along their common vertical boundary (one should smooth the interface via a reparametrization in the $s$-variable so the result is smooth).
\begin{lemma}\label{lemma:gluing-lemma}
If $\mathfrak{c}_{0},\mathfrak{c}_{1},\mathfrak{c}_{2}$ represent the continuation maps associated to $\psi_{s,t},\eta_{s,t}$ and $\psi_{s,t}\#\eta_{s,t}$, respectively then $\mathfrak{c}_{2}=\mathfrak{c}_{1}\circ \mathfrak{c}_{0}$ holds on homology.
\end{lemma}
\begin{proof}
  The argument is again standard Floer theory. A similar but more complicated argument is given in \S\ref{sec:cont-maps-are}, with more details. One sets up a parametric moduli space of continuation cylinder type equations, where the parameter is a gluing parameter; see Figure \ref{fig:gluing-two-cont}. As the gluing parameter tends to $\infty$, the solutions break into configurations which represent the composition $\mathfrak{c}_{1}\circ \mathfrak{c}_{0}$. The parametric moduli space should be constructed so that the other end of the moduli space (when the parameter is zero, say) is exactly the moduli space used to define $\mathfrak{c}_{2}$. The details are left to the reader.
\end{proof}

\begin{figure}[h]
  \centering
  \begin{tikzpicture}
    \draw (0,0) coordinate(X) --+ (10,0) +(0,1) --+ (10,1);
    \draw (X)+(0,0.5) circle(0.2 and 0.5) +(3,0.5) circle (0.2 and 0.5) +(4,0.5) circle (0.2 and 0.5)
    +(6,0.5) circle (0.2 and 0.5) +(7,0.5) circle (0.2 and 0.5) +(10,0.5) circle (0.2 and 0.5);
    \draw (X)+(3.5,0)node[below]{$\eta_{s,t}\vphantom{\psi_{s,t}}$} +(6.5,0)node[below]{$\psi_{s,t}$}; \draw[<->] (X)+(4,-0.2)--node[below]{$\sigma$}+(6,-0.2);
    
  \end{tikzpicture}
  \caption{Rough schematic of the process of gluing two continuation data. The length $\sigma$ is the gluing parameter.}
  \label{fig:gluing-two-cont}
\end{figure}
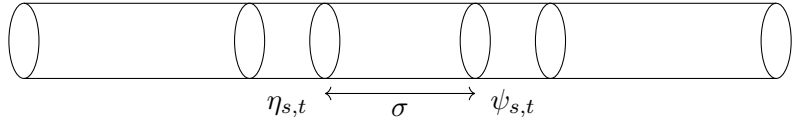

\subsubsection{Functorial structure of Floer cohomology}
\label{sec:funct-struct-floer}

As in, e.g., \cite{cant-arXiv-2023}, the Floer cohomology groups $\mathrm{HF}(\psi_{t})$ with continuation maps forms a functor valued in the category of $\Lambda$-vector spaces. The domain category $\mathscr{C}$ has:
\begin{enumerate}
\item objects equal to contact-at-infinity Hamiltonian systems $\psi_{t}$,
\item morphisms equal to homotopy classes of continuation data.
\end{enumerate}
Initially, the functor $\mathrm{HF}(-)$ is only defined on a full subcategory $\mathscr{C}^{\times}\subset \mathscr{C}$ of systems whose time-1 maps have non-degenerate fixed points and are sufficiently generic to achieve the relevant transversality conditions. Then the functor $\mathrm{HF}(-)$ is extended to all objects by the completion:
\begin{equation}\label{eq:extension}
  \mathrm{HF}(\psi_{t})=\lim \mathrm{HF}(\psi_{t}'),
\end{equation}
where the limit is over the slice category of objects $\psi_{t}'\in \mathscr{C}^{\times}$ equipped with a morphism $\psi_{t}\to \psi_{t}'$. The functorial structure of $\mathrm{HF}$ extends to $\mathscr{C}$ by abstract nonsense.

\begin{lemma}
  For any Reeb flow $R^{\alpha}_{s}$, there exists a sequence $\varphi_{n,t}$ so that:
  \begin{enumerate}[label=(\arabic*)]
  \item\label{L1} the ideal restriction of $\varphi_{n,t}$ is $R^{\alpha}_{s_{n}t}$,
  \item\label{L2} $0<s_{n+1}<s_{n}$, with $s_{n}\to 0$,
  \item\label{L3} $\varphi_{n,t}\psi_{t}\in \mathscr{C}^{\times}$.
  \end{enumerate}
  Moreover, the following morphism, guaranteed by universal property of \eqref{eq:extension},
  \begin{equation}\label{eq:induced}
    \mathrm{HF}(\psi_{t})\to \lim_{n\to\infty}\mathrm{HF}(\varphi_{n,t}\psi_{t}),
  \end{equation}
  is an isomorphism.
\end{lemma}
\begin{proof}
  It is clear that data $\varphi_{n,t}$ can always be chosen to satisfy \ref{L1}, \ref{L2}, and \ref{L3}. Moreover, there is an obvious sequence:
  \begin{equation}\label{eq:diagram}
    \begin{tikzcd}
      {\dots}\arrow[r,"{}"] &{\varphi_{n+1,t}\psi_{t}}\arrow[from=2-2,"{}"]\arrow[r,"{}"] &{\varphi_{n,t}\psi_{t}}\arrow[from=2-3,"{}"]\arrow[r,"{}"] &{\dots}\\
      {\dots}\arrow[r,"{}"] &{\psi_{t}}\arrow[r,"{\id}"] &{\psi_{t}}\arrow[r,"{}"] &{\dots}
    \end{tikzcd}
  \end{equation}
  induced by the canonical homotopy classes of continuation data whose ideal restrictions are affine reparametrizations of subintervals of the twisted Reeb flow $s\mapsto R_{st}\psi_{t}$.

  Thus, by the definition of the limit, there exists a morphism \eqref{eq:induced}. It remains to prove that this morphism is an isomorphism.

  Suppose there is a morphism $\psi_{t}\to \psi_{t}'$ in $\mathscr{C}$ with $\psi_{t}'\in \mathscr{C}^{\times}$. Since $\psi_{t}'\in \mathscr{C}^{\times}$, it follows that $\psi_{1}'$ has no discriminant points. Consequently, $R_{s}\psi_{1}'$ has no discriminant points for $s$ in some interval $s\in [0,\epsilon]$. Let $\varphi_{t}$ be a contact isotopy so that $\varphi_{t}\psi_{t}'\in \mathscr{C}^{\times}$ and the ideal restriction of $\varphi_{t}$ is $R_{\epsilon t}$. Then it is straightforward to construct a canonical factorization in $\mathscr{C}$:
  \begin{equation}\label{eq:square-psi-prime}
    \begin{tikzcd}
      {\varphi_{n,t}\psi_{t}}\arrow[from=2-1,"{}"]\arrow[r,"{}"] &{\varphi_{t}\psi_{t}'}\arrow[from=2-2,"{}"]\\
      {\psi_{t}}\arrow[r,"{}"] &{\psi_{t}'}
    \end{tikzcd}
  \end{equation}
  for $n$ sufficiently large. Here the left vertical morphism is from \eqref{eq:diagram}. It follows from \cite{uljarevic-zhang-JFPTA-2022} that the continuation map associated to the right vertical morphism is an isomorphism, since the continuation data never develops discriminant points.

  Thus, for $n$ large enough, we can apply $\mathrm{HF}$ to the top and right morphisms in \eqref{eq:square-psi-prime}, and invert the right morphism to obtain a map:
  \begin{equation*}
    \lim_{n\to\infty}\mathrm{HF}(\varphi_{n,t}\psi_{t})\to \mathrm{HF}(\psi_{t}').
  \end{equation*}
  Straightforward book-keeping (i.e., diagram chasing), implies the induced morphism is independent of $\epsilon$, and is natural, and thereby induces a map:
  \begin{equation*}
    \lim_{n\to\infty}\mathrm{HF}(\varphi_{n,t}\psi_{t})\to \lim\mathrm{HF}(\psi_{t}')=\mathrm{HF}(\psi_{t}),
  \end{equation*}
  where the right limit is over the slice category. This morphism is the desired inverse of the canonical morphism \eqref{eq:induced}.
\end{proof}

\subsection{Eternal classes in symplectic cohomology}
\label{sec:etern-class-sympl}

An element of $\lim \mathrm{HF}$ is a natural transformation $\Z/2\to \mathrm{HF}$. Such an natural transformation induces an element in the colimit $\mathfrak{e}\in \mathrm{SH}=\colim \mathrm{HF}$. As in the introduction, we let $\mathrm{SH}_{e}\subset \mathrm{SH}$ be the subspace spanned by these elements. In this section we will prove structural theorems which allow us to understand $\mathrm{SH}_{e}$ and $\mathrm{SH}$ in terms of sequence of isotopies, rather than the full category $\mathscr{C}$.

\subsubsection{Final and cofinal sequences}
\label{sec:finality-cofinality}

Let $R_{s}^{\alpha}$ be a Reeb flow and let $\varphi_{t}$ be a contact isotopy on the ideal boundary $Y$. We extend the Reeb vector field to a Hamiltonian vector field of $W$, and extend $\varphi_{t}$ to a contact-at-infinity isotopies of $W$. These choices induce a functor:
\begin{equation*}
  s\in \R\mapsto \varphi_{t}^{-1}\circ R_{st}^{\alpha}\in \mathscr{C},
\end{equation*}
sending a morphism $s_{1}\le s_{2}$ to the continuation data obtained by linearly interpolating the speed of the Reeb flow, as written down in \S\ref{sec:spectr-invar-non-intro}.

This functor is independent of the choice of extensions to $W$, up to natural isomorphism. The natural isomorphism is given by continuation data whose ideal restriction is independent of $s$; see \cite{cant-arXiv-2023,cant-hedicke-kilgore} for detailed discussion on this point.

We can then precompose $\mathrm{HF}:\mathscr{C}\to \mathrm{Vect}$ with the functor $\R\to \mathscr{C}$ to obtain the functor $V^{\alpha}(\varphi_{t}):\R\to \mathrm{Vect}$ given by:
\begin{equation*}
  s\mapsto V_{s}^{\alpha}(\varphi_{t}):=\mathrm{HF}(\varphi_{t}^{-1}\circ R_{st}^{\alpha});
\end{equation*}
such a functor is, by definition, a persistence module. The persistence module is independent of the choice of extension of $R^{\alpha}$ or $\varphi_{t}$ to the filling, up to natural isomorphism in the category of persistence modules.

The universal properties for limits and colimits imply there is a commutative square:
\begin{equation}\label{eq:final-cofinal-square}
  \begin{tikzcd}
    {\lim_{s} V_{s}^{\alpha}(\varphi_{t})}\arrow[from=2-1,"{}"]\arrow[r,"{}"] &{\colim_{s} V_{s}^{\alpha}(\varphi_{t})}\arrow[d,"{}"]\\
    {\lim \mathrm{HF}}\arrow[r,"{}"] &{\mathrm{SH}}
  \end{tikzcd}
\end{equation}
A key structural lemma is the following:
\begin{lemma}\label{lemma:final-cofinal}
  The vertical morphisms in \eqref{eq:final-cofinal-square} are isomorphisms.
\end{lemma}
The key technical lemma needed to prove this is the following contact-geometric result:
\begin{lemma}\label{lemma:technical-filtering}
  Let $\psi_{t}$ be any contact-at-infinity isotopy.
  \begin{enumerate}[label=(\arabic*)]
  \item\label{f1} There exists morphisms in $\mathscr{C}$ from $\psi_{t}$ to $\varphi_{t}^{-1}\circ R_{st}$ for $s$ sufficiently positive.
  \item\label{f2} Any two morphisms in $\mathscr{C}$ from $\psi_{t}$ to $\varphi_{t}^{-1}\circ R_{s_{1}t}$ become equal when post-composed with the morphism $\varphi_{t}^{-1}\circ R_{s_{1}t}\to \varphi_{t}^{-1}\circ R_{s_{2}t}$ for $s_{2}$ sufficiently positive.
  \item\label{f3} There exists morphisms in $\mathscr{C}$ from $\varphi_{t}^{-1}\circ R_{st}$ to $\psi_{t}$ for $s$ sufficiently negative.
  \item\label{f4} Any two morphisms in $\mathscr{C}$ from $\varphi_{t}^{-1}\circ R_{s_{1}t}$ to $\psi_{t}$ become equal when pre-composed with the morphism $\varphi_{t}^{-1}\circ R_{s_{2}t}\to \varphi_{t}^{-1}\circ R_{s_{1}t}$ for $s_{2}$ sufficiently negative.
  \end{enumerate}
\end{lemma}
It is a standard abstract nonsense to use Lemma \ref{lemma:technical-filtering} to prove Lemma \ref{lemma:final-cofinal}; see, e.g., \cite[\S2.2.8]{cant-hedicke-kilgore}.
\begin{proof}[Proof of Lemma \ref{lemma:technical-filtering}]
  Part \ref{f1} and \ref{f2} are a special case of \cite[\S2.2.11]{cant-hedicke-kilgore}. On the other hand, \ref{f3} and \ref{f4} follow from \cite[\S2.2.11]{cant-hedicke-kilgore} to conclude morphisms between the inverse objects $\psi_{t}^{-1}\to R_{-st}\circ \varphi_{t}$ for sufficiently negative $s$, noting that the existence and equality of morphisms in $\mathscr{C}$ behaves well with respect to inversion.
\end{proof}

\subsubsection{Barcode decomposition}
\label{sec:barc-decomp}

The significance of Lemma \ref{lemma:final-cofinal} is that it allows us to prove things about the subspace of eternal classes $\mathrm{SH}_{e}$ by arguing with the persistence module $V^{\alpha}(\varphi_{t})$ and its barcode decomposition. Such an argument will be used to prove Theorems \ref{theorem:when-unit-eternal} and \ref{theorem:ideal-pop} on the behaviour of eternal classes with respect to the unit and the pair-of-pants product.

Another use of Lemma \ref{lemma:final-cofinal} is that the persistence module $V^{\alpha}(\varphi_{t})$ is used to define the spectral invariants for $\varphi_{t}$. This implies that the right vertical morphism in \eqref{eq:final-cofinal-square} sends the subspace of $\mathrm{colim} V^{\alpha}(\varphi_{t})$ spanned by the fully infinite bars isomorphically onto $\mathrm{SH}_{e}$. On the other hand, every non-zero element in $\zeta\in \mathrm{SH}/\mathrm{SH}_{e}$ can be represented uniquely as a linear combination of finitely many half-infinite bars in $\mathrm{colim}V^{\alpha}(\varphi_{t})$; it follows that the spectral invariant $c_{\alpha}(\zeta;\varphi_{t})$ is the greatest endpoint appearing in this linear combination.

One final note is that the barcode decomposition implies the following structural lemma for eternal classes:
\begin{lemma}\label{lemma:barcode-dcomp}
  A class $\mathfrak{e}\in \mathrm{SH}$ lies in $\mathrm{SH}_{e}$ if and only if $\mathfrak{e}$ lies in the image of the structure map $V^{\alpha}_{s}(\varphi_{t})\to \mathrm{SH}$ for every $s$.
\end{lemma}
\begin{proof}
  The ``only if'' direction is easy, and so we show the ``if'' direction. The barcode decomposition implies that any element $\mathfrak{e}\in \mathrm{SH}$ can be expressed as a linear combination:
  \begin{equation*}
    \mathfrak{e}=\mathfrak{e}_{1}+\dots+\mathfrak{e}_{q}+\zeta_{1}+\dots+\zeta_{p}
  \end{equation*}
  where $\mathfrak{e}_{i}$ are basis elements corresponding to fully infinite bars, and $\zeta_{j}$ are basis elements corresponding to half-infinite bars.

  Suppose that $\mathfrak{e}\in \mathrm{SH}$ lies in the image of $V^{\alpha}_{s}(\varphi_{t})\to \mathrm{SH}$ for every $s$. Then, by the barcode decomposition, we must have $p=0$, i.e., $\mathfrak{e}=\mathfrak{e}_{1}+\dots+\mathfrak{e}_{q}$.

  However, the barcode decomposition implies $\mathfrak{e}_{i}$ are elements of the limit $\lim V^{\alpha}(\varphi_{t})$, i.e., they are natural transformations from $\Z/2$ to $V^{\alpha}(\varphi_{t})$. Thus $\mathfrak{e}$ actually lies in the image of $\lim V^{\alpha}(\varphi_{t})\to \mathrm{SH}$. Lemma \ref{lemma:final-cofinal} implies $\mathfrak{e}$ lies in the image of $\lim \mathrm{HF}\to \mathrm{SH}$, and this completes the proof.
\end{proof}

\section{Flat Hamiltonian connections and the pair-of-pants product}
\label{sec:ham-conn-pair-pants-product}

The goal of this section is to develop enough of the theory of Hamiltonian connections to prove some of the theorems from the introduction, in particular, the sub-additivity of the spectral invariants.

The first part of this section concerns results about Hamiltonian connections which will be needed for two reasons. First, Hamiltonian connections are used to construct the pair-of-pants operation on Floer theory. The pair-of-pants product has a long history, and is part of the larger framework of the TQFT structure on Floer cohomology; see, e.g., \cite{schwarz-thesis,piunikhin-salamon-schwarz-1996,salamon99-quantum-products,seidel-topology-2003,mcduffsalamon,abouzaid-monograph-2015,seidel-eq-pop,kislev-shelukhin,alizadeh-atallah-cant}.\footnote{Such a structure is attributed to Donaldson around the early 1990s}

Second, in our arguments concerning the persistence modules: $$V^{\alpha}_{s}(\psi_{t})=\mathrm{HF}(\psi_{t}^{-1}\circ R_{st}^{\alpha}),$$ we will need to ``switch'' the order from $\psi_{t}^{-1}\circ R_{st}^{\alpha}$ to $R_{st}^{\alpha}\circ \psi_{t}^{-1}$, and the tool used to relate their Floer cohomologies involves maps defined using Hamiltonian connections on cylinders.

\subsection{Hamiltonian connections}
\label{sec:basic-framework}

For the uniniated we recall the set-up of Floer's equation for Hamiltonian connections over surfaces, as in \cite{mcduffsalamon}. This section is primarily a recollection of results in \cite[\S A]{alizadeh-atallah-cant}.

\subsubsection{On the class of Hamiltonian functions}
\label{sec:hamilt-funct}

Fix a starshaped compact domain $\Omega$, so that $\bd \Omega$ lies in the convex-end and projects isomorphically to the ideal boundary. Let $\mathscr{H}(\Omega)$ be the set of smooth functions $H:W\to \R$ so that:
\begin{equation*}
  \d H(Z)=H+c
\end{equation*}
holds outside of $\Omega$, where $c$ is locally constant function. Define:
\begin{equation*}
  \textstyle\mathscr{H}=\bigcup_{\Omega}\mathscr{H}(\Omega).
\end{equation*}
For a manifold $P$, a \emph{smooth family} $H_{p}\in \mathscr{H}$, $p \in P$, is a defined to be a family which locally factors through some $\mathscr{H}(\Omega)$ as a smooth family. In particular, the Hamiltonian vector fields $X_{H_{p}}$ induce a smooth family of contact vector fields on the ideal boundary on $W$.

Introduce $\mathscr{H}_{0}$ be the space of \emph{normalized} functions. Since $W$ is open, we adopt the following normalization scheme. Pick a distinguished non-compact end of $W$ (i.e., pick a connected component of its ideal boundary). Let us say that $H\in \mathscr{H}$ is \emph{normalized} if $\d H(Z)=H$ holds in the distinguished compact end.

\subsubsection{Connection potentials}
\label{sec:conn-potent}

Consider the bundle $W\times \Sigma\to \Sigma$ over a punctured Riemann surface $\Sigma$. Let us declare a \emph{connection potential} to be a one-form $\mathfrak{a}$ on $W\times \Sigma$ whose restriction to $W\times U$ is of the form: $$\mathfrak{a}=K_{s,t}\d s+H_{s,t}\d t$$ for conformal $z=s+it$ coordinate patches $U$ on $\Sigma$, and such that $K_{s,t},H_{s,t}$ are smooth families in $\mathscr{H}$; if they lie in $\mathscr{H}_{0}$, then we say $\mathfrak{a}$ is \emph{normalized}.

Introduce the closed two-form $\Omega=\mathrm{pr}^{*}_{W}\omega-\d\mathfrak{a}$, and define the associated connection by:
\begin{equation*}
  \mathfrak{H}:=TW^{\perp \Omega}
\end{equation*}
This is a connection of the sense of Ehresmann, i.e., is a complement to the vertical distribution. Such connections are called \emph{Hamiltonian connections}. 

Importantly:
\begin{lemma}
  Every Hamiltonian connection $\mathfrak{H}$ is generated by a unique normalized connection potential $\mathfrak{a}$.  
\end{lemma}
\begin{proof}
  This is a consequence of a computation and the fact that normalized constant functions are zero. It is proved in \cite[\S A]{alizadeh-atallah-cant}.
\end{proof}

\subsubsection{Curvature of a connection}
\label{sec:curvature-connection}

Curvature can be defined for any Ehresmann connection. It is a tensor which takes two vectors $v_{1},v_{2}$ in $T\Sigma_{z}$ and returns a vertical vector field $\mathfrak{R}(v_{1},v_{2})$ on the fiber over $z$. The formula is:
\begin{equation*}
  \mathfrak{R}(v_{1},v_{2})=[v_{1}^{\mathfrak{H}},v_{2}^{\mathfrak{H}}]-[v_{1},v_{2}]^{\mathfrak{H}},
\end{equation*}
where the $\mathfrak{H}$-superscript denotes horizontal lift. One first extends $v_{i}$ to vector fields around $z$ but then shows the result is independent of the extension, similarly to how one proves the Riemann curvature is a tensor.

Hamiltonian connections are special in that $\mathfrak{R}(v_{1},v_{2})$ is always a Hamiltonian vector field. In fact, standard computations show the following holds when pulled back to the fiber over $z$:
\begin{equation*}
  \mathfrak{R}(v_{1},v_{2})\intprod \omega=-\d \mathfrak{r}(v_{1},v_{2}),
\end{equation*}
where $\mathfrak{r}(v_{1},v_{2})=v_{1}^{\mathfrak{H}}\intprod v_{2}^{\mathfrak{H}}\intprod (\mathrm{pr}^{*}\omega-\d \mathfrak{a})$.

\begin{lemma}
  If $\mathfrak{a}=K\d s+H\d t$ holds in a local coordinate chart then:
  \begin{equation*}
    \mathfrak{r}=(\pd{H}{s}-\pd{K}{t}-\omega(X_{H},X_{K}))\d s\wedge \d t.
  \end{equation*}
  In particular, if $\mathfrak{a}$ is normalized, then $\mathfrak{r}(v_{1},v_{2})$ is normalized for all pairs of vectors $v_{1},v_{2}$ in the base. Moreover, still assuming $\mathfrak{a}$ is normalized, $\mathfrak{R}=0$ if and only if $\mathfrak{r}=0$.
\end{lemma}
\emph{Remark.} The two-form $\mathfrak{r}$ plays an important role in the energy identity for solutions to Floer's equation; see Lemma \ref{lemma:basic-facts-floers-equation}.
\begin{proof}
  The stated formula is straightforward computation; see \cite[\S A]{alizadeh-atallah-cant}. Now suppose $\mathfrak{R}=0$. The fact that $\d\mathfrak{r}(v_{1},v_{2})=0$ holds when pulled back to the fiber implies that $\mathfrak{r}(v_{1},v_{2})$ is constant on each fiber. Since $\mathfrak{r}(v_{1},v_{2})$ is normalized, it follows that $\mathfrak{r}(v_{1},v_{2})=0$, as desired.
\end{proof}

\subsubsection{Monodromy of a connection}
\label{sec:monodromy-connection}

Fix a Hamiltonian connection $\mathfrak{H}$, and let $\eta:[0,1]\to \Sigma$ be a smooth path. Let $v_{t}$ be a compactly supported time-dependent vector field on $\Sigma$ so that $v_{t}(\eta(t))=\eta'(t)$.

The horizontal lift $v_{t}^{\mathfrak{H}}$ is a complete vector field; this follows from the requirement that the Hamiltonian functions appearing in $\mathfrak{a}$ form a smooth family in $\mathscr{H}$. Thus the time-$t$ map of the flow of $v_{t}^{\mathfrak{H}}$ is a diffeomorphism taking $W\times \eta(0)$ to $W\times \eta(t)$, and is therefore identified with a smooth isotopy of $W$. The resulting isotopy is in fact a contact-at-infinity Hamiltonian isotopy, which we call the \emph{monodromy} of $\mathfrak{H}$ along $\eta$. For the proofs of these assertions, we refer the reader to \cite[\S A.2]{alizadeh-atallah-cant}.

\subsubsection{Groups of Hamiltonian diffeomorphisms}
\label{sec:groups-of-ham-diff}

Let $\mathrm{HI}(\Omega)$ be the group of Hamiltonian isotopies $\varphi_{t}$ generated by smooth families $H_{t}\in\mathscr{H}(\Omega)$, where $t\in [0,1]$. Similarly to the definition of $\mathscr{H}$, let:
\begin{equation*}
  \textstyle\mathrm{HI}=\bigcup_{\Omega}\mathrm{HI}(\Omega),
\end{equation*}
and declare a smooth family $\varphi_{p,t}\in \mathrm{HI}$, $p\in P$, to be one which locally factors through some $\mathrm{HI}(\Omega)$ as a smooth family.

A \emph{homotopy with fixed endpoints} is a smooth family $\varphi_{s,t}\in \mathrm{HI}$, $s\in [0,1]$, so the the time-1 map $\varphi_{s,1}$ is independent of $s$ (note that $\varphi_{s,0}=\id$ is always independent of $s$). Two elements in $\mathrm{HI}$ are said to equivalent if they differ by a homotopy with fixed endpoints. Let us denote by $\mathrm{UH}$ the group of equivalence classes. This group should be thought of as the universal cover of the group of contact-at-infinity Hamiltonian diffeomorphisms. A smooth family $\varphi_{p}\in \mathrm{UH}$ is one which locally admits a lift to a smooth family in the group $\mathrm{HI}$.

Let $\mathfrak{H}$ be a flat Hamiltonian connection on $W\times \Sigma$. The assignment which sends a path $\eta$ in $\Sigma$ to its monodromy in $\mathrm{UH}$ descends to a functor from the fundamental groupoid of $\Sigma$ to $\mathrm{UH}$ (note that $\mathrm{UH}$ is a group, and hence is a groupoid). This is specific to flat connections, i.e., the monodromy of non-flat connections is sensitive to the path $\eta$ and not just its homotopy class.

The \emph{monodromy representation} of a flat connection is the restriction of this functor to $\pi_{1}(\Sigma,z_{0})$, and should be thought of as a homomorphism: $$\pi_{1}(\Sigma,z_{0})\to \mathrm{UH}.$$ Different base points $z_{0}$ give conjugated homomorphisms, and so the monodromy representation is best thought of as a conjugacy class.

\subsubsection{Coordinate transformations}
\label{sec:coord-transf}

Let $\mathfrak{H}$ be a Hamiltonian connection on $W\times \Sigma$, and let $g:\Sigma\to \mathrm{UH}$ be a smooth family. The data of $g$ enables one to define a diffeomorphism $W\times \Sigma\to W\times \Sigma$ sending $(w,z)$ to $(g_{z}(w),z)$, and we let $g_{*}\mathfrak{H}$ denote the push-forward connection; this is a well-defined Ehresmann connection since the diffeomorphism preserves vertical tangent spaces.

\begin{lemma}
  If $\pi_{2}(\Sigma)=0$, then $g_{*}\mathfrak{H}$ is another Hamiltonian connection. Furthermore, if $\mathfrak{H}$ is flat then so is $g_{*}\mathfrak{H}$.
\end{lemma}
\begin{proof}
  Since $\mathrm{UH}$ is simply connected (it is a universal cover), and we assume $\pi_{2}(\Sigma)=0$, it follows that the smooth family $g:\Sigma\to \mathrm{UH}$ can be contracted to a point; i.e., there exists a smooth family $g_{t}:\Sigma\to \mathrm{UH}$ so that $g_{1}=g$ and $g_{0}=\id$. The desired result then follows from \cite[\S A.2.4]{alizadeh-atallah-cant}. The statement about flatness is obvious; flat connections locally admit flat sections, and this property is invariant under push-forward by fiber-preserving diffeomorphisms.
\end{proof}

We will use this construction in the following way.
\begin{lemma}\label{lemma:coord-monod}
  Suppose $\mathfrak{H}_{1},\mathfrak{H}_{2}$ are two flat Hamiltonian connections, and their monodromy representations $\pi_{1}(\Sigma,z_{0})\to \mathrm{UH}$ are conjugate, then there exists a smooth family $g:\Sigma\to \mathrm{UH}$ so that $g_{*}\mathfrak{H}_{1}=\mathfrak{H}_{2}$.
\end{lemma}
\begin{proof}
  First we prove the case when their monodromy representations are the same. Fix a basepoint $z_{0}$. For each $z$, pick an arbitrary path $\eta$ joining $z_{0}$ to $z$, and define:
  \begin{equation*}
    g_{z}=(\text{monod.\@ of $\mathfrak{H}_{2}$ along $\eta$})(\text{monod.\@ of $\mathfrak{H}_{1}$ along $\eta$})^{-1};
  \end{equation*}
  where $\mathrm{monod.}$ stands for the monodromy valued in $\mathrm{UH}$.
  
  This is well-defined (independent of $\eta$), since we assume their monodromy representations are the same. It follows that $(w,z)\mapsto (g_{z}(w),z)$ takes any path which is flat for $\mathfrak{H}_{1}$ to one which is flat for $\mathfrak{H}_{2}$. Since tangent lines to flat paths span the horizontal subspaces follows that $\mathfrak{g}_{*}\mathfrak{H}_{1}=\mathfrak{H}_{2}$ as Ehresmann connections, as desired.

  In the case when the monodromy representations are merely conjugate, we can first replace $\mathfrak{H}_{1}$ by $g_{*}\mathfrak{H}_{1}$ where $g\in \mathrm{UH}$ is a constant.
\end{proof}

\subsubsection{Correcting flat connections on cylinders}
\label{sec:corr-flat-conn}

Let $\mathfrak{H}$ be a flat connection on $[0,\infty)\times \R/\Z$ or $(-\infty,0]\times \R/\Z$. Suppose the monodromy of $\mathfrak{H}$ along $0\times \R/\Z$ equals $\varphi\in \mathrm{UH}$, and suppose $H_{t}\in \mathscr{H}(W)$ is such that the time-one map in $\mathrm{UH}$ equals $c\varphi c^{-1}$ for some $c\in \mathrm{UH}$.

Then, by Lemma \ref{lemma:coord-monod}, there exists $g$ so that $g_{*}\mathfrak{H}$ is the Hamiltonian connection whose potential is $\mathfrak{a}=H_{t}\d t$. Moreover, $g$ is homotopic to the constant map valued at $\id$, since $\mathrm{UH}$ is simply connected. Let $g_{\tau}$, $\tau\in [0,1]$, be such a deformation so $g_{1}=g$ and $g_{0}=\id$.

Then the transformed connection: $$\mathfrak{H}'=g_{\beta(\abs{s})}^{*}\mathfrak{H}$$ is a connection which agrees with $\mathfrak{H}$ near $s=0$ and has connection potential $\mathfrak{a}=H_{t}\d t$ for $\abs{s}\ge 1$. In this fashion we are able to ``correct'' any connection $\mathfrak{H}$ so that it appears in a standard form in cylindrical ends.

\subsection{Floer's equation associated to a Hamiltonian connection}
\label{sec:floers-equat-assoc}

Let $\mathfrak{H}$ be a Hamiltonian connection on $W\times \Sigma\to \Sigma$, and let $J$ be an admissible almost complex structure on $W$, as in \S\ref{sec:choice-almost-compl}. This induces a unique almost complex structure $J_{\mathfrak{H}}$ on $W\times \Sigma$ so that:
\begin{enumerate}
\item $J_{\mathfrak{H}}|_{TW\times \set{z}}=J$,
\item $J_{\mathfrak{H}}(\mathfrak{H})\subset \mathfrak{H}$,
\item $\d\pi$ is $J_{\mathfrak{H}},j$ holomorphic.
\end{enumerate}
We define $\mathscr{M}(\mathfrak{H})$ to be the set of finite energy maps $u:\Sigma\to W$ whose graph in $W\times \Sigma$ is $J_{\mathfrak{H}}$-holomorphic. Let $w:\Sigma\to \Sigma\times W$ be the parametrization of the graph $w(z)=(z,u(z))$. The \emph{energy} of a solution is defined to be:
\begin{equation*}
  E(u)=\int (\Pi_{\mathfrak{H}}\d w)^{*}\omega,
\end{equation*}
where $\Pi_{\mathfrak{H}}$ is the projection on $T(W\times \Sigma)\to TW$ whose kernel is $\mathfrak{H}$. Let us note that the energy integrand $(\Pi_{\mathfrak{H}}\d w)^{*}\omega$ is a non-negative two-form, since the linear map $\Pi_{\mathfrak{H}}\d w$ is complex linear and $J$ is $\omega$-tame.

A bit more practically:
\begin{lemma}\label{lemma:basic-facts-floers-equation}
If $\mathfrak{H}$ is defined by a connection potential which locally appears as $\mathfrak{a}=K\d s+H\d t$ in a conformal coordinate chart $z=s+it$ on $\Sigma$, then $u\in \mathscr{M}(\mathfrak{H})$ if and only if:
\begin{equation*}
  \bd_{s}u-X_{K}(u)+J(u)(\bd_{t}u-X_{H}(u))=0,
\end{equation*}
and the local contribution to the energy from the coordinate chart equals:
\begin{equation*}
  E(u)=\int \omega(\bd_{s}u-X_{K}, \bd_{t}u-X_{H})ds dt.
\end{equation*}
Suppose that $\Omega\subset \Sigma$ is a compact domain with smooth boundary $\bd\Omega$. Then:
\begin{equation*}
  E(u|_{\Omega})=\omega(u|_{\Omega})-\int_{\bd\Omega}u^{*}\mathfrak{a}+\int_{\Omega}u^{*}\mathfrak{r}.
\end{equation*}
\end{lemma}
\begin{proof}
  We refer the reader to \cite[\S A.4]{alizadeh-atallah-cant}.  
\end{proof}

\subsection{Operations on Floer theory via flat connections on cylinders}
\label{sec:hamilt-conn-cylind}

In this section, we explain how flat connections on cylinders can be used to define isomorphisms between Floer cohomology groups. These operations arise naturally when considering pair-of-pants products, and it will be important for us to show these isomorphisms act identically on $\mathrm{SH}$.

Let $\mathfrak{H}$ be a flat Hamiltonian connection over $\R\times \R/\Z$, and suppose that the normalized connection potential $\mathfrak{a}$ satisfies:
\begin{enumerate}[label=(F\arabic*)]
\item\label{item:F1} $\mathfrak{a}=H_{0,t}\d t$ on the region $\set{s>1}$,
\item\label{item:F2} $\mathfrak{a}=H_{1,t}\d t$ on the region $\set{s<0}$.
\end{enumerate}
The idea is that counting solutions in $\mathscr{M}(\mathfrak{H})$ will define a sort of generalized continuation map.

\subsubsection{Definition of the map}
\label{sec:definition-map}

Let $\mathfrak{H}$ be a connection satisfying \ref{item:F1} and \ref{item:F2} where $H_{i,t}$ is the generator of $\psi_{i,\beta(3t-1)},$ so that each solution $u\in \mathscr{M}(\mathfrak{H})$ solves the Floer differential equation from \S\ref{sec:floer-differential-moduli} for $\psi_{0,t}$ in the end $s>1$ and for $\psi_{1,t}$ in the end $s<0$.

\begin{lemma}
  Let $\delta=\delta^{1}_{s,t}\d s+\delta^{2}_{s,t}\d t$ where $(s,t,w)\mapsto \delta^{i}_{s,t}(w)$ is compactly supported in $(0,1)\times \R/\Z\times W$, and let $\mathfrak{H}_{\delta}$ be the connection whose potential is $\mathfrak{a}+\delta$. For generic $\delta$, the moduli space $\mathscr{M}(\mathfrak{H}_{\delta})$ is cut transversally and the total evaluation map is transverse to any given smooth map.
\end{lemma}
\begin{proof}
  The proof is simpler than the argument in Lemma \ref{lemma:transversality-floer}, because $\delta$ is not required to respect any symmetries, and the details are left to the reader.
\end{proof}

Let $\mathscr{M}_{d}(x;a;y;\mathfrak{H}_{\delta})\subset \mathscr{M}(\mathfrak{H}_{\delta})$ be the component of $u$ so that:
\begin{enumerate}
\item the asymptotics satisfy $\gamma_{-}(0)=x$, $\gamma_{+}(0)=y$,
\item $\omega(u)=a$,
\item $\mathrm{CZ}_{\mathfrak{s}}(\gamma_{+}) - \mathrm{CZ}_{\mathfrak{s}}(\gamma_{-}) + 2\mathfrak{s}^{-1}(0)\cdot[u]=0$,
\end{enumerate}
where $\mathfrak{s}$ is a section of the determinant line bundle; the set-up is essentially exactly the same as with the continuation map from \S\ref{sec:continuation-maps}.

Because $\mathfrak{H}_{\delta}$ has zero curvature outside of a compact set (the perturbation $\delta$ introduces some curvature, but only on the region where $\delta\ne 0$) solutions $u\in \mathscr{M}(\mathfrak{H}_{\delta})$ satisfy an a priori energy bound in terms of $\omega(u)$; see \S\ref{sec:floers-equat-assoc}.

In particular, for generic $\delta$, the usual arguments prove $\mathscr{M}_{0}(x;a;y;\mathfrak{H}_{\delta})$ is a finite set of points, and $\mathscr{M}_{1}(x;a;y;\mathfrak{H}_{\delta})$ is a 1-manifold which is compact up to the breaking of Floer differential cylinders. We therefore define:
\begin{equation*}
  \mathrm{C}(\mathfrak{H}_{\delta})(\tau^{b}y)=\sum \#\mathscr{M}_{0}(x;a;y;\mathfrak{H}_{\delta})\tau^{a+b}x.
\end{equation*}
This operation satisfies:
\begin{lemma}
  For generic $\delta$, the map $\mathrm{C}(\mathfrak{H}_{\delta}):\mathrm{CF}(\psi_{0,t})\to \mathrm{CF}(\psi_{1,t})$ is well-defined and is a chain map. Moreover, the chain homotopy class of the map is independent of $\delta$, and depends only on the connected component of $\mathfrak{H}$ in the space of connections satisfying \ref{item:F1} and \ref{item:F2}.
\end{lemma}
\begin{proof}
  This is standard Floer theory, similar to the analogous facts about the continuation map from \S\ref{sec:continuation-maps}.
\end{proof}

Let us denote by $\mathrm{C}(\mathfrak{H})$ the induced map on homology (we drop the $\delta$ since the chain homotopy class is independent of the perturbation).

\subsubsection{On the space of connections satisfying \ref{item:F1} and \ref{item:F2}}
\label{sec:space-conn-satisfy}

Let $\mathscr{C}(\psi_{1,t},\psi_{0,t})$ be the space of flat connections satisfying \ref{item:F1} and \ref{item:F2}, where $H_{i,t}$ are the generators of $\psi_{i,\beta(3t-1)}$. First of all:
\begin{lemma}
  The time-1 maps of $\psi_{0,t},\psi_{1,t}$ are conjugate in $\mathrm{UH}$ if and only if $\mathscr{C}(H_{0,t},H_{1,t})$ is non-empty.
\end{lemma}
\begin{proof}
  If the time-1 maps are conjugate, then there exists a connection $\mathfrak{H}$ satisfying \ref{item:F1} and \ref{item:F2}; this follows from the construction in \S\ref{sec:corr-flat-conn}. On the other hand, because the connection is flat, the monodromy along the path $s\mapsto (s,0)$ conjugates the monodromy along the loops $t\mapsto (0,t)$ and $t\mapsto (1,t)$. These monodromies are $\psi_{\beta(3t-1),1}$ and $\psi_{\beta(3t-1),0}$, and hence the reverse implication holds.
\end{proof}

A priori, the map $\mathrm{C}(\mathfrak{H})$ is sensitive to the connected component of $\mathfrak{H}$ in the space $\mathscr{C}(\psi_{1,t},\psi_{0,t})$. However, we gain a rough understanding of the connectivity of this space from the following lemma:
\begin{lemma}\label{lemma:torsor}
  Let $\mathfrak{H}_{0}$ be the connection whose potential is $\mathfrak{a}=H_{0,t}\d t$. Then any $\mathfrak{H}\in \mathscr{C}(\psi_{1,t},\psi_{0,t})$ is of the form $g_{*}\mathfrak{H}_{0}$ where $g_{s,t}$ is the monodromy of $\mathfrak{H}$ along the path joining $(1,t)$ to $(s,t)$.
\end{lemma}
\begin{proof}
   As in \S\ref{sec:coord-transf}, $g$ sends paths which are flat for $\mathfrak{H}_{0}$ to paths which are flat for $\mathfrak{H}$; it suffices to prove this for paths lying in the slice $s=1$ and the slices $t=\mathrm{const}$, in which case it is obvious. This completes the proof.
\end{proof}

\emph{Remark}. While not needed for our subsequent arguments, it is noteworthy that $\mathrm{C}(\mathfrak{H})$ is always an isomorphism (after passing to homology). Indeed, the inverse is of the same form and is equal to $\mathrm{C}(\mathfrak{H}')$ where $\mathfrak{H}'$ is obtained by reflecting $\mathfrak{H}$ about $t=1/2$. To prove this, one uses standard Floer theory gluing to prove:
\begin{equation*}
  \mathrm{C}(\mathfrak{H}')\circ \mathrm{C}(\mathfrak{H})=\mathrm{C}(\mathfrak{H}'\#\mathfrak{H})
\end{equation*}
where $\mathfrak{H}'\#\mathfrak{H}$ is the flat connection in $\mathscr{C}(\psi_{0,t},\psi_{0,t})$ obtained by gluing:
\begin{equation}\label{eq:gluing-connections}
  (\mathfrak{H}'\#\mathfrak{H})_{s,t,w}=\left\{
    \begin{aligned}
      &\mathfrak{H}'_{s,t,w}&&\text{ if }s\le 1,\\
      &\mathfrak{H}_{s-2,t,w}&&\text{ if }s\ge 2.
    \end{aligned}
  \right.
\end{equation}
Using the same idea as Lemma \ref{lemma:torsor}, one proves that $\mathfrak{H}'\#\mathfrak{H}=g_{*}\mathfrak{H}_{0},$ where $g_{s,t}=\id$ for $s\not\in [0,3]$. Moreover, one uses the reflection symmetry to show that $g_{s,t}$ is homotopic to the constant map $\id$, preserving $g_{s,t}=\id$ for $s\not\in [0,3]$ during the homotopy. Thus $\mathfrak{H}'\#\mathfrak{H}$ lies in the same connected component of $\mathfrak{H}_{0}$, and so $\mathrm{C}(\mathfrak{H}'\#\mathfrak{H})=\mathrm{C}(\mathfrak{H}_{0})=\id,$ as desired.

\subsubsection{Action on symplectic cohomology}
\label{sec:acti-sympl-cohom}

It is important to show that $\mathrm{C}(\mathfrak{H})$ acts identically on the colimit, in the following sense:
\begin{lemma}
  Let $\mathfrak{H}\in \mathscr{C}(\psi_{1,t},\psi_{0,t})$. Then the following triangle commutes:
  \begin{equation*}
    \begin{tikzcd}
      &{\mathrm{SH}}\arrow[from=2-1,"{\mathfrak{c}}",out=90,in=180]\arrow[from=2-2,"{\mathfrak{c}}"]\\
      {\mathrm{HF}(\psi_{0,t})}\arrow[r,"{\mathrm{C}(\mathfrak{H})}"] &{\mathrm{HF}(\psi_{1,t})},
    \end{tikzcd}
  \end{equation*}
  where the maps to $\mathrm{SH}$ are the structure maps from the universal property of the colimit.
\end{lemma}
\begin{proof}
  The idea of the proof is to show that $\mathfrak{c}\circ \mathrm{C}(\mathfrak{H})$ is chain homotopic to $\mathfrak{c}$ via the usual parametric moduli space idea. The proof is unfortunately rather long and technical, as we need to explicitly construct a path between two connections being careful to never introduce any positive curvature in the non-compact end.

  The first step is a small trick; we will correct $\mathfrak{H}$ by a time reparametrization. Let $F(s,t,w)=(s,f(t),w)$ where $f:[0,1]\to [0,1]$ is the function illustrated in Figure \ref{fig:deforming-function}. Note that $F$ is homotopic to the identity. Redefine:
  \begin{equation*}
    \mathfrak{H}=F^{*}\mathfrak{H},
  \end{equation*}
  which is easily seen to be a flat Hamiltonian connection in the same connected component of $\mathfrak{H}$, so $\mathrm{C}(\mathfrak{H})$ is unchanged under this replacement.

  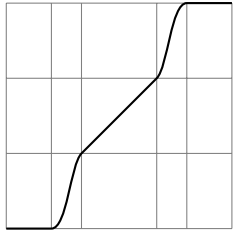
\begin{figure}[h]
    \centering
    \begin{tikzpicture}
      \draw[black!50!white] (0,0) grid (3,3) (0.6,0)--+(0,3) (2.4,0)--+(0,3);
      \draw[line width=0.8pt] (0,0)to(0.6,0)to[out=0,in=225,looseness=0.5](1,1)to(2,2)to[out=45,in=180,looseness=0.5](2.4,3)to(3,3);
    \end{tikzpicture}
    \caption{Function $f(t)$ used to deform the initial connection $\mathfrak{H}$. It is important that $f$ is increasing, $f(t)=1$ holds on $[2/3+\delta,1]$, $f(t)=0$ on $[0,1/3-\delta]$, and that $f(t)=t$ for all $t\in [1/3,2/3]$. In particular, $\beta(3f(t)-1)=\beta(3t-1)$.}
    \label{fig:deforming-function}
  \end{figure}
  
  From Lemma \ref{lemma:torsor}, we know that $\mathfrak{H}=g_{*}\mathfrak{H}_{0}$ where $g_{s,t}$ is the monodromy of $\mathfrak{H}$ from $(1,t)$ to $(s,t)$. Because of our earlier replacement, we know that:
  \begin{enumerate}
  \item\label{gst1} $g_{s,t}=g_{s,0}$ for $t\le 1/3-\delta$ and $g_{s,t}=g_{s,1}$ for $t\ge 2/3+\delta$, and,
  \item\label{gst2} $\bd_{s}g_{s,t}=0$ for $s\le 0$ and $g_{s,t}=\id$ for $s\ge 1$.
  \end{enumerate}

  Next we deform:
  \begin{equation*}
    g^{\sigma}_{s,t}=\left\{
      \begin{aligned}
        &g_{\sigma\beta(s)+(1-\sigma)s,t}&&\text{ for }\sigma\in [0,1],\\
        &g_{(2-\sigma)\beta(s)+(\sigma-1),t}&&\text{ for }\sigma\in [1,2],
      \end{aligned}
    \right.
  \end{equation*}
  noting that properties \eqref{gst1} and \eqref{gst2} are preserved along the deformation. Note that $g^{2}_{s,t}=\id$ holds identically.

  Let $\mathfrak{H}^{\sigma}=(g^{\sigma})_{*}\mathfrak{H}_{0}$, so that $\mathfrak{H}^{\sigma}=\mathfrak{H}$ for $\sigma=0$ while $\mathfrak{H}^{\sigma}=\mathfrak{H}_{0}$ for $\sigma=2$. The monodromy isotopy of $\mathfrak{H}^{\sigma}$ over the loop $\set{1}\times \R/\Z$ is equal to $\psi_{0,\beta(3t-1)}$, while the monodromy isotopy over $\set{0}\times \R/\Z$ is equal to:
  \begin{equation*}
    \eta^{\sigma}_{t}=g_{0,t}^{\sigma}\psi_{0,\beta(3t-1)}(g_{0,0}^{\sigma})^{-1}
  \end{equation*}
  By construction, $\eta^{\sigma}_{t}=\psi_{1,\beta(3t-1)}$ for $\sigma=0$ and $\eta_{t}^{\sigma}=\psi_{0,\beta(3t-1)}$ for $\sigma=2$.

  Pick an auxiliary Reeb flow $R^{\alpha}$ and pick a speed $c$ large enough that: $$s\mapsto R_{c s}^{\alpha}(\eta_{s}^{\sigma})^{-1}\eta_{1}^{\sigma}$$ has a positive ideal restriction, for each $\sigma\in [0,2]$. Define:
  \begin{equation*}
    \xi^{\sigma}_{s,t}=R_{c \beta(-s)\beta(3t-1)}^{\alpha}(\eta_{\beta(-s)t}^{\sigma})^{-1}\eta_{t}^{\sigma},
  \end{equation*}
  which plays the role of the reparametrized continuation data from $\eta_{t}^{\sigma}$, for $s\ge 1$, to $R_{c\beta(3t-1)}^{\alpha}$, for $s\le 0$, for each $\sigma$, similarly to \S\ref{sec:cont-cylind}.

  Let $Y_{s,t}^{\sigma},X_{s,t}^{\sigma}$ be the generators of $\xi_{s,t}^{\sigma}$, as in \S\ref{sec:cont-cylind}, and let $K_{s,t},H_{s,t}$ be their normalized generators. Define the Hamiltonian connection:
  \begin{equation*}
    \mathfrak{K}^{\sigma}=\rho(t)K_{s,t}^{\sigma}\d s+H_{s,t}^{\sigma}\d t,
  \end{equation*}
  over the cylinder, where $\rho(t)=\beta(3-3t)$ as in \S\ref{sec:cont-cylind}. Pick $\delta$ small enough in the definition of the function $f(t)$ so that $\rho'(t)=0$ holds on $[0,2/3+\delta]$.

  Importantly, by construction, the moduli space $\mathscr{M}(\mathfrak{K}^{0})$ is precisely the moduli space of continuation cylinders for the continuation data $R^{\alpha}_{cst}\psi_{1,st}^{-1}\psi_{1,t}$, as in \S\ref{sec:cont-cylind}. Similarly, $\mathscr{M}(\mathfrak{K}^{2})$ is the moduli space of continuation cylinders for the continuation data $R^{\alpha}_{cst}\psi_{0,st}^{-1}\psi_{0,t}$.

  For $\sigma\in (0,2)$, $\mathscr{M}(\mathfrak{K}^{\sigma})$ is not exactly a moduli space of continuation cylinders, however, it is extremely similar to one. In particular, $\mathfrak{K}^{\sigma}$ has nonpositive curvature outside of a compact set, which we will demonstrate momentarily, which implies solutions of $\mathscr{M}(\mathfrak{K}^{\sigma})$ satisfy a priori energy estimates like solutions of the continuation cylinder equation (see \S\ref{sec:cont-cylind}).

  To see that $\mathfrak{K}^{\sigma}$ has nonpositive curvature $\mathfrak{r}=r\d s\wedge \d t$, we compute:
  \begin{equation*}
    \begin{aligned}
      r      &=\textstyle\bd_{s}H^{\sigma}-\rho(t)\bd_{t}K^{\sigma}-\rho'(t)K^{\sigma}-\rho(t)\omega(X_{H},X_{K})\\      &=-\rho'(t)K^{\sigma}_{s,t}+\rho(t)(\bd_{s}H^{\sigma}-\bd_{t}K^{\sigma}-\omega(X_{H^{\sigma}},X_{K^{\sigma}}))\\
      &=-\rho'(t)K^{\sigma}_{s,1},
    \end{aligned}
  \end{equation*}
  where we have used $\pd{H^{\sigma}}{s}=\rho(t)\pd{H^{\sigma}}{s}$ since $H^{\sigma}=0$ holds when $\rho(t)\ne 1$, $K^{\sigma}_{s,t}=K^{\sigma}_{s,1}$ when $\rho'(t)\ne 0$, and the well-known curvature identity when $K^{\sigma},H^{\sigma}$ are the generators for a family $\xi_{s,t}^{\sigma}$.

  Since $K^{\sigma}_{s,1}$ is the generator of $s\mapsto \xi_{s,1}^{\sigma}$, which is a non-positive path in the contactomorphism group, and hence $K^{\sigma}_{s,1}\le 0$ holds outside of a compact set, and hence $-\rho'(t)K^{\sigma}_{s,t}\le 0$ since $\rho'(t)\le 0$. Thus the curvature is non-positive outside of a compact set, and hence the Lemma \ref{lemma:basic-facts-floers-equation} implies solutions will satisfy an an a priori energy estimate.
  
  Finally, consider the family of connections: $$\mathfrak{G}^{\sigma}=(\mathfrak{K}^{\sigma})\# (g^{\sigma}_{*}\mathfrak{H}_{0}),$$ where the $\#$ operation is defined in \eqref{eq:gluing-connections}. This family of connections still has non-positive curvature outside of a compact set, since $g^{\sigma}_{*}\mathfrak{H}_{0}$ and $\mathfrak{K}^{\sigma}$ both have this property.

  The rest of the proof is based on the usual Floer theory argument fact that counting rigid elements in $\mathscr{M}(\mathfrak{G}^{0})$ and $\mathscr{M}(\mathfrak{G}^{2})$ define chain homotopic maps $\mathrm{HF}(\psi_{0,t})\to \mathrm{HF}(R_{ct})$. Since the proof is already long enough, we will be extremely brief for this part of the proof, as it is quite similar to many other arguments in this paper and in Floer theory in general.
  
  Note that $\mathfrak{G}^{2}=\mathfrak{K}^{2}$, and hence counting the rigid elements in $\mathscr{M}(\mathfrak{G}^{2})$ in the usual way defines a continuation map:
  \begin{equation*}
    \mathfrak{c}:\mathrm{HF}(\psi_{0,t})\to \mathrm{HF}(R_{ct}^{\alpha})
  \end{equation*}
  On the other hand, $\mathfrak{G}^{0}=\mathfrak{K}^{0}\# \mathfrak{H}$, and hence counting the rigid elements in $\mathscr{M}(\mathfrak{G}^{0})$ in the usual way defines the composite map $\mathfrak{c}\circ \mathrm{C}(\mathfrak{H})$:
  \begin{equation*}
    \mathrm{HF}(\psi_{0,t})\to \mathrm{HF}(\psi_{1,t})\to \mathrm{HF}(R_{ct}).
  \end{equation*}
  Thus, by the usual chain homotopy invariance of operations defined in Floer theory using homotopic data, it holds that $\mathfrak{c}=\mathfrak{c}\circ \mathrm{C}(\mathfrak{H})$, where the continuation maps land in $\mathrm{HF}(R_{ct})$. The desired result then follows, since the continuation map into $\mathrm{SH}$ factors through the continuation into $\mathrm{HF}(R_{ct})$. This completes the proof.
\end{proof}

\subsubsection{Equivalence of spectral invariants}
\label{sec:equiv-spectr-invar}

In this section we will apply \S\ref{sec:acti-sympl-cohom} to deduce two variations on the spectral invariants of a contact isotopy are actually the same.

For a contact isotopy $\varphi_{t}$ and a choice of Reeb flow $R^{\alpha}_{s}$ of the ideal boundary, both extended to the filling $W$, consider the two Floer cohomologies:
\begin{enumerate}
\item $V_{s}=\mathrm{HF}(\varphi_{t}^{-1}\circ R^{\alpha}_{st})$,
\item $V_{s}'=\mathrm{HF}(R^{\alpha}_{st}\circ \varphi_{t}^{-1})$.
\end{enumerate}
The structure maps $V_{s}\to \mathrm{SH}$ and $V_{s}'\to \mathrm{SH}$ enable us to define spectral invariants for non-zero classes $\zeta\in \mathrm{SH}/\mathrm{SH}_{e}$, as in \S\ref{sec:spectr-invar-non-intro} and \S\ref{sec:finality-cofinality}. Namely, we consider the infimal $s$ for which $\zeta$ lies in the image of the map to $\mathrm{SH}$. Let us call these two spectral invariants $c_{\alpha}(\zeta;\varphi_{t})$ and $c_{\alpha}'(\zeta;\varphi_{t})$.

We claim $c_{\alpha}(\zeta;\varphi_{t})=c_{\alpha}'(\zeta;\varphi_{t})$. The key lemma is:
\begin{lemma}
  For any two isotopies $\varphi_{t},\phi_{t}$, the isotopies $\varphi_{t}\circ \phi_{t}$ and $\phi_{t}\circ \varphi_{t}$ have conjugate time-1 maps in the universal cover.
\end{lemma}
\begin{proof}
  Clearly $\varphi_{t}\circ \phi_{t}$ and $\varphi_{1}^{-1}\circ \varphi_{t}\circ \phi_{t}\circ \varphi_{1}$ have conjugate time-1 maps in the universal cover. However, the deformation:
  \begin{equation*}
    \varphi_{\sigma}^{-1}\circ \varphi_{\sigma t}\circ \phi_{t}\circ \varphi_{\sigma}\circ \varphi_{\sigma t}^{-1}\circ \varphi_{t}
  \end{equation*}
  has fixed endpoints, namely $\id$ at $t=0$ and $\phi_{1}\circ \varphi_{1}$ at $t=1$. Thus it follows that the isotopies at $\sigma=0$ and $\sigma=1$ have the same time-1 map in the universal cover. The desired result follows.
\end{proof}

By \S\ref{sec:space-conn-satisfy}, there is some flat connection $\mathfrak{H}\in \mathscr{C}(\varphi_{t}\circ \phi_{t},\phi_{t}\circ \varphi_{t})$. Then the map $\mathrm{C}(\mathfrak{H}):\mathrm{HF}(\varphi_{t}\circ \phi_{t})\to \mathrm{HF}(\phi_{t}\circ \varphi_{t})$ is defined, and by \S\ref{sec:acti-sympl-cohom} we have:
\begin{lemma}\label{lemma:swap}
  If $\zeta\in \mathrm{SH}$ lies in the image of $\mathrm{HF}(\phi_{t}\circ \varphi_{t})\to \mathrm{SH}$, then it is also lies in the image of $\mathrm{HF}(\varphi_{t}\circ \phi_{t})\to \mathrm{SH}$.\hfill$\square$
\end{lemma}
Applying this in an obvious way yields the equality
$c_{\alpha}(\zeta;\varphi_{t})= c_{\alpha}'(\zeta;\varphi_{t}),$ and the claim is proved.\hfill$\square$

\subsection{The pair-of-pants product}
\label{sec:pair-pants-product}

In this section we define pair-of-pants products using flat Hamiltonian connections in a similar manner to how we defined the operations $\mathrm{C}(\mathfrak{H})$. The connections used are explained in \S\ref{sec:flat-conn-pair}. We will show in \S\ref{sec:pair-pants-product-1} and \S\ref{sec:comm-with-cont} that these operations induce a well-defined product $\mathrm{SH}\otimes \mathrm{SH}\to \mathrm{SH}$. In \S\ref{sec:subsp-etern-class} we prove Theorem \ref{theorem:ideal-pop} that $\mathrm{SH}_{e}\subset \mathrm{SH}$ is an ideal, and in \S\ref{sec:subadd-spectr-invar} we prove Theorem \ref{theorem:sub-additive} on the sub-additivity of spectral invariants.

\subsubsection{Flat connections on the pair-of-pants surface}
\label{sec:flat-conn-pair}

Let $\Sigma=\C\setminus \set{0,1}$ be the pair-of-pants surface. Let $C_{0},C_{1},C_{\infty}$ be the cylindrical ends:
\begin{equation*}
  C_{0}=D(r)^{\times}\hspace{1cm}C_{1}=1+D(r)^{\times}\hspace{1cm}C_{\infty}=\C\setminus D(R),
\end{equation*}
where $D(r)^{\times}$ is the punctured disk and $r<1/2$ and $R>3/2$ are chosen so that the cylindrical ends are disjoint. We think of $C_{0},C_{1}$ as positive ends, conformally equivalent to $[0,\infty)\times \R/\Z$ via $re^{-2\pi (s+it)}$ and $1-re^{-2\pi(s+it)}$, respectively, and (the closure of) $\C\setminus D(R)$ as a negative end, conformally equivalent to $(-\infty,0]\times \R/\Z$ via $Re^{-2\pi (s+it)}$.

\begin{figure}[h]
  \centering
  \begin{tikzpicture}
    \path[every node/.style={fill,circle,inner sep=1pt}] (0,0)node{}--(1,0)node{};

    \draw[line width=.8pt] (1/3,0)--(2/3,0) (0,0) circle (1/3) (1,0) circle (1/3) (0,0) circle ({3/2});

  \end{tikzpicture}
  \caption{Pair of pants surface $\Sigma=\C\setminus \set{0,1}$. Shown are the three circles $\bd D(r)$, $1+\bd D(r)$, and $\bd D(R)$ and the connecting arc $[r,1-r]$.}
  \label{fig:popsurface}
\end{figure}
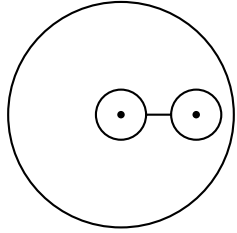

Let us introduce $\mathscr{C}(\psi_{\infty,t};\psi_{0,t},\psi_{1,t})$ as the space of flat Hamiltonian connections $\mathfrak{H}$, with normalized connection potential $\mathfrak{a}$, such that:
\begin{enumerate}
\item $\mathfrak{a}=H_{i,t}\d t$ in the cylindrical end $C_{i}$, where $H_{i,t}$ is the normalized generator for $\psi_{i,\beta(3t-1)}$,
\item the monodromy of $\mathfrak{H}$ along the arc $[r,1-r]$ joining $\bd C_{0}$ to $\bd C_{1}$ is the identity isotopy;
\end{enumerate}
see Figure \ref{fig:popsurface} for an illustration of the relevant paths on the pair-of-pants surface. Then we have the following structural result:
\begin{lemma}\label{lemma:pop-structural}
  The following holds:
  \begin{enumerate}[label=(\alph*)]
  \item\label{Sa} The space $\mathscr{C}(\psi_{\infty,t};\psi_{0,t},\psi_{1,t})$ is non-empty if and only if the time-1 map of $\psi_{\infty,t}$ is conjugate to the time-1 map of the product $\psi_{0,t}\psi_{1,t}$, where the time-1 maps are taken in $\mathrm{UH}$.

  \item\label{Sb} Moreover, if $\mathfrak{H}_{0},\mathfrak{H}_{1}$ both lie in $\mathfrak{C}(\psi_{\infty,t};\psi_{0,t},\psi_{1,t})$, then there is a smooth family interpolating of flat connections between $\mathfrak{H}_{0}$ and a connection of the form:
    \begin{equation*}
      (\iota_{*}\mathfrak{H}\#\mathfrak{H}_{1})_{z}=\left\{
        \begin{aligned}
          &(\iota_{*}\mathfrak{H})_{z}\text{ if }z\in C_{\infty},\\ 
          &(\mathfrak{H}_{1})_{z}\text{ if }z\in D(R),
        \end{aligned}
      \right.
    \end{equation*}
    where $\mathfrak{H}\in \mathscr{C}(\psi_{\infty,t},\psi_{\infty,t})$ and $\iota$ is the diffeomorphism:
    \begin{equation*}
      (w,s,t)\in W\times (-\infty,1]\times \R/\Z\mapsto (w,Re^{-2\pi (s-1+it)})\in W\times C_{\infty}.
    \end{equation*}
    The smooth family can be taken to be fixed in the ends $C_{0},C_{1},e^{2\pi}C_{\infty}$.
    
  \item\label{Sc} Finally, if $\mathfrak{H}_{0}\in \mathscr{C}(\psi_{0,t}^{0}\psi_{1,t}^{0};\psi_{0,t}^{0},\psi_{1,t}^{0})$, and $\psi_{0,t}^{\sigma},\psi_{1,t}^{\sigma}$ are smooth families, then there is a smooth family $\mathfrak{H}_{\sigma}\in \mathscr{C}(\psi_{0,t}^{\sigma}\psi_{1,t}^{\sigma};\psi_{0,t}^{\sigma},\psi_{1,t}^{\sigma})$ extending $\mathfrak{H}_{0}$.
  \end{enumerate}
\end{lemma}
\begin{proof}
  First of all, the monodromy representation based at $(r,0)$ of: $$\mathfrak{H}\in \mathscr{C}(\psi_{\infty,t};\psi_{0,t},\psi_{1,t}),$$ valued in $\mathrm{UH}$ as in \S\ref{sec:groups-of-ham-diff}, sends the loop winding clockwise around $\bd C_{0}$ to $\psi_{0,t}$, and the loop going from $[r,1-r]$, then around $\bd C_{1}$, then back along $[r,1-r]$, to $\psi_{1,t}$, in $\mathrm{UH}$ (the actual monodromy isotopies will be time reparametrized versions, but time reparametrization does not affect their image in $\mathrm{UH}$). Note that it is for this reason that we require that the monodromy along the arc $[r,1-r]$ is the identity. The concatenation of these loops is freely homotopic to the loop winding around $\bd C_{\infty}$. Since $\mathfrak{H}$ is a flat Hamiltonian connection, it follows that $\psi_{\infty,t}$ is conjugate to $\psi_{0,t}\psi_{1,t}$ in $\mathrm{UH}$. This proves the ``only if'' part of \ref{Sa}. The ``if'' part of \ref{Sa} follows from \ref{Sc}, since one can take $\mathfrak{H}$ to the be trivial connection $\mathfrak{a}=0$ and $\psi_{0,t}^{0}=\psi_{1,t}^{0}=\id$, and the construction in \S\ref{sec:corr-flat-conn}.

  Moving on to \ref{Sb}, we argue as follows. By the above paragraph, $\mathfrak{H}_{0},\mathfrak{H}_{1}$ both have the same monodromy representation, and so by Lemma \ref{lemma:coord-monod} there exists a map $g:\Sigma\to \mathrm{UH}$ so that $g_{*}\mathfrak{H}_{1}=\mathfrak{H}_{0}$. Moreover, by the formula for $g$ given in the proof of Lemma \ref{lemma:coord-monod}, we may suppose that:
  \begin{equation}\label{eq:g-is-id-somewhere}
    \text{$g_{z}$ equals to the identity for $z\in C_{0}\cup C_{1}\cup [r,1-r]$.}
  \end{equation}
  There is a smooth homotopy of maps $f_{\sigma}$ satisfying:
  \begin{enumerate}
  \item $f_{0}=\id:\Sigma\to \Sigma$,
  \item $f_{\sigma}(C_{0}\cup C_{1}\cup [r,1-r])\subset C_{0}\cup C_{1}\cup [r,1-r]$ for all $\sigma$,
  \item $f_{1}(D(R))\subset C_{0}\cup C_{1}\cup [r,1-r]$,
  \item $f_{\sigma}(z)=z$ for $z\not\in D(e^{2\pi}R)$.
  \end{enumerate}
  For instance, one can ``squash'' $D(R)$ horizontally, and then squash vertically, and cut-off this homotopy outside a slightly larger disk.

  Then $g_{\sigma}=g\circ f_{\sigma}$ is a smooth family of maps $\Sigma\to \mathrm{UH}$, and we consider the family of connections $\mathfrak{H}_{\sigma}=(g_{\sigma})_{*}\mathfrak{H}_{1}$. This family satisfies the conclusion of \ref{Sb}, as can be easily verified by the reader. The key is that $g_{\sigma}\ne g_{0}$ holds only on $D(e^{2\pi}R)\setminus D(R)$.

  Finally we turn to \ref{Sc}. First we claim there is some family: $$\mathfrak{H}_{\sigma}'\in \mathscr{C}(\psi_{0,t}^{\sigma}\psi_{1,t}^{\sigma};\psi_{0,t}^{\sigma},\psi_{1,t}^{\sigma}),$$ not necessarily one which extends $\mathfrak{H}_{0}$. The construction of such a family $\mathfrak{H}_{\sigma}'$ is perfomed in \cite[\S A.3]{alizadeh-atallah-cant}, and also follows from the trick of \cite{kislev-shelukhin} which uses the technique of holomorphic embeddings of strips into $\Sigma$, together with the correction trick of \S\ref{sec:corr-flat-conn}.

  Since $\mathfrak{H}_{0},\mathfrak{H}_{0}'$ lie in $\mathscr{C}(\psi_{0,t}^{0}\psi_{1,t}^{0};\psi_{0,t}^{0},\psi_{1,t}^{0})$, there is some $g$ so $g_{*}\mathfrak{H}_{0}'=\mathfrak{H}_{0}$ and so that $g(z)=\id$ holds on $C_{0}\cup C_{1}\cup [r,1-r]$, as in the proof of \ref{Sb}. Then the modification $\mathfrak{H}''_{\sigma}=g_{*}\mathfrak{H}_{\sigma}'$ extends $\mathfrak{H}_{0}$. A final application of the correction trick in \S\ref{sec:corr-flat-conn} yields $\mathfrak{H}_{\sigma}=\gamma_{\sigma,*}\mathfrak{H}_{\sigma}''$ where $\gamma_{\sigma}$ is supported in a small neighborhood of $C_{\infty}$ and so that $\mathfrak{H}_{\sigma}$ actually lies in $\mathscr{C}(\psi_{0,t}^{\sigma}\psi^{\sigma}_{1,t};\psi^{\sigma}_{0,t},\psi^{\sigma}_{1,t})$. This completes the proof.
\end{proof}

\subsubsection{The pair-of-pants product valued in $\mathrm{SH}$}
\label{sec:pair-pants-product-1}

Pick a connection: $$\mathfrak{H}\in \mathscr{C}(\psi_{0,t}\psi_{1,t};\psi_{0,t},\psi_{1,t}),$$ and let $\mathfrak{H}_{\delta}$ be a generic perturbation supported in a compact region of $\Sigma\times W$. Pick a section $\mathfrak{s}$ of the determinant line bundle whose zero section is disjoint from the orbits of the input and output systems. We consider the moduli space:
\begin{equation*}
  \mathscr{M}_{d}(z;a;x,y;\mathfrak{H}_{\delta})\subset \mathscr{M}(\mathfrak{H}_{\delta})
\end{equation*}
of finite energy solutions $u$ whose restrictions to the cylindrical ends are asymptotic to orbits $\gamma_{\infty},\gamma_{0},\gamma_{1}$ such that:
\begin{enumerate}
\item $\gamma_{\infty}(0)=z$, $\gamma_{0}(0)=x$, $\gamma_{1}(0)=y$,
\item $\omega(u)=a$,
\item $d=-n+2\mathfrak{s}^{-1}(0)\cdot [u]+\mathrm{CZ}_{\mathfrak{s}}(\gamma_{0})+\mathrm{CZ}_{\mathfrak{s}}(\gamma_{1})-\mathrm{CZ}_{\mathfrak{s}}(\gamma_{\infty})$.
\end{enumerate}
For generic perturbation, $\mathscr{M}_{0}(z;a;x,y;\mathfrak{H}_{\delta})$ is a transversally cut-out finite set, and $\mathscr{M}_{1}(z;a;x,y;\mathfrak{H}_{\delta})$ is compact up to the breaking of Floer cylinders at the cylindrical ends. The details of this assertion follow the same lines as those in previous sections, e.g., \S\ref{sec:floer-differential-moduli}, \S\ref{sec:cont-cylind}, and \S\ref{sec:definition-map}. The statement about the dimension of the moduli space follows from the general index formula for Cauchy-Riemann operators and appears in the present form in, e.g., \cite{cant-thesis-2022}. The $-n$ is because a pair-of-pants has Euler characteristic $-1$.

Define the operation:
\begin{equation*}
  \mathrm{P}(\mathfrak{H}_{\delta})(\tau^{b_{1}}x,\tau^{b_{2}}y)=\sum_{a,z}\#\mathscr{M}_{0}(z;a;x,y;\mathfrak{H}_{\delta})\tau^{a+b_{1}+b_{2}}z,
\end{equation*}
which is extended to all semi-infinite sums in $\mathrm{CF}(\psi_{0,t})\otimes \mathrm{CF}(\psi_{1,t})$ in the obvious way. Let us comment briefly on why the map is well-defined. In any tensor product of semi-infinite sums, the quantity $b_{1}+b_{2}$ is bounded from below, say by $-B$. By non-negativity of energy, there is a minimal $a$, say $-A$, such that:
\begin{equation*}
  \mathscr{M}_{0}(z;a;x,y;\mathfrak{H}_{\delta})\ne \emptyset,
\end{equation*}
which is uniform as $z,x,y$ vary over all possible orbits. Thus the exponent $a+b_{1}+b_{2}$ appearing in non-zero terms is bounded from below by $-A-B$ and so the product is indeed valued in the correct vector space.

The operation $\mathrm{P}(\mathfrak{H}_{\delta})$ is also a chain map, when the source is endowed with the tensor product differential. This fact is proved by examination of the non-compact ends of $\mathscr{M}_{1}(z;a;x,y;\mathfrak{H}_{\delta})$.

Finally, the usual argument (see Lemma \ref{lemma:deformation}) shows the chain homotopy class of the map is independent of the perturbation and the connected component of $\mathfrak{H}$ in the space $\mathscr{C}(\psi_{0,t}\psi_{1,t};\psi_{0,t},\psi_{1,t})$. We denote the resulting map on homology by:
\begin{equation*}
  \mathrm{P}(\mathfrak{H}):\mathrm{HF}(\psi_{0,t})\otimes \mathrm{HF}(\psi_{1,t})\to \mathrm{HF}(\psi_{\infty,t}).
\end{equation*}
\emph{Remark.} It is important to note that we do not show that $\mathrm{P}(\mathfrak{H})$ is independent of $\mathfrak{H}$, since the author was unable to determine whether the relevant space of flat connections is connected.

Let us now consider two connections $\mathfrak{H}_{0},\mathfrak{H}_{1}\in \mathscr{C}(\psi_{0,t}\psi_{1,t};\psi_{0,t},\psi_{1,t})$. By \ref{Sb} in Lemma \ref{lemma:pop-structural}, we know that $\mathfrak{H}_{0}$ can be deformed to $\iota_{*}\mathfrak{H}\#\mathfrak{H}_{1}$ remaining in the space of flat connections. It then follows from the usual TQFT structure of Floer theory that:
\begin{equation}\label{eq:compo-formula}
  \mathrm{P}(\mathfrak{H}_{0})=\mathrm{C}(\mathfrak{H})\circ \mathrm{P}(\mathfrak{H}_{1}).
\end{equation}
We therefore define:
\begin{equation}\label{eq:pop-valued-in-SH}
  \mathrm{P}:\mathrm{HF}(\psi_{0,t})\otimes\mathrm{HF}(\psi_{1,t})\to \mathrm{SH}
\end{equation}
by the formula $\mathrm{P}=\mathfrak{c}\circ \mathrm{P}(\mathfrak{H}_{0})$ where $\mathfrak{c}$ is the map $\mathrm{HF}(\psi_{0,t}\psi_{1,t})\to \mathrm{SH}$. It follows from the technical lemma in \S\ref{sec:acti-sympl-cohom} and \eqref{eq:compo-formula} that $\mathrm{P}$ is independent of the connection used. This is the desired product valued in $\mathrm{SH}$.

\emph{Remark.} It is important to note that $\mathrm{P}$ factors through $\mathrm{HF}(\psi_{0,t}\psi_{1,t})$; this observation will be used in \S\ref{sec:subsp-etern-class} and \S\ref{sec:subadd-spectr-invar}.

\subsubsection{Compatibility with continuation maps}
\label{sec:comm-with-cont}

The goal in this section is to prove that the $\mathrm{SH}$-valued product $\mathrm{P}$ commutes with continuation maps, in the following sense:
\begin{lemma}\label{lemma:compat}
  Let $\psi_{0,s,t},\psi_{1,s,t}$ be continuation data, considered as a morphism between generic objects in $\mathscr{C}^{\times}$, used as in \S\ref{sec:continuation-maps} to define continuation maps:
  \begin{equation*}
    \mathfrak{c}_{i}:\mathrm{HF}(\psi_{i,0,t})\to \mathrm{HF}(\psi_{i,1,t}).
  \end{equation*}
  Then it holds that $\mathrm{P}\circ \mathfrak{c}_{0}\otimes \mathfrak{c}_{1}=\mathrm{P}$.
\end{lemma}
As a consequence of this compatibility, we can extend $\mathrm{P}$ to all objects of $\mathscr{C}$ following similar arguments to \S\ref{sec:funct-struct-floer}, although the details of this extension are not so deep and are left to the reader.

More importantly for us, Lemma \ref{lemma:compat} implies that $\mathrm{P}$ extends to a product $\mathrm{SH}\otimes \mathrm{SH}\to \mathrm{SH}$, as follows: given objects $\zeta_{1},\zeta_{2}\in \mathrm{SH}$ pick any pair $\psi_{0,t},\psi_{1,t}$ so that $\zeta_{i}$ is the image of $\xi_{i}
\in \mathrm{HF}(\psi_{i,t})$ under the structure map to $\mathrm{SH}$. Define $\zeta_{1}\zeta_{2}=\mathrm{P}(\xi_{1},\xi_{2})$. This definition is independent of the choice of $\xi_{i}$. Indeed, if $\psi_{i,t}',\xi_{i}'$ is another choice, then there are continuation maps:
\begin{equation*}
  \mathfrak{c}:\mathrm{HF}(\psi_{i,t})\to \mathrm{HF}(\psi_{i,t}'')\text{ and }\mathfrak{c}':\mathrm{HF}(\psi_{i,t}')\to \mathrm{HF}(\psi_{i,t}''),
\end{equation*}
so that $\mathfrak{c}(\xi_{i})=\mathfrak{c}'(\xi_{i}')$, so it holds that $\mathrm{P}(\xi_{1},\xi_{2})=\mathrm{P}(\xi_{1}',\xi_{2}')$. To see why there must exist $\psi_{i,t}''$ and $\mathfrak{c},\mathfrak{c}'$, we can take $\psi_{i,t}''$ to be a sufficiently fast Reeb flow, and appeal to \S\ref{sec:finality-cofinality} to reduce to a statement about direct limits, where it becomes an exercise in undergraduate algebra.

\begin{proof}[Proof of Lemma \ref{lemma:compat}]
  The idea is summarized in Figure \ref{fig:pop-com}. We begin with two continuation data $\psi_{i,s,t}$, $i=0,1$, which we split into two pieces:
  \begin{equation*}
    \psi_{i,s,t}^{\sigma,-}=\psi_{i,\sigma s,t}\text{ and }\psi_{i,s,t}^{\sigma,+}=\psi_{i,\sigma+(1-\sigma)s,t};
  \end{equation*}
  i.e., $\psi_{i,s,t}^{\sigma,-}$ goes from $\psi_{i,0,t}$ to $\psi_{i,\sigma,t}$, and $\psi_{i,s,t}^{\sigma,+}$ goes from $\psi_{i,\sigma,t}$ to $\psi_{i,1,t}$.

  Part \ref{Sc} of Lemma \ref{lemma:pop-structural} furnishes a smooth family of flat connections:
  \begin{equation*}
    \mathfrak{H}_{\sigma}\in \mathscr{C}(\psi_{0,\sigma,t}\psi_{1,\sigma,t};\psi_{0,\sigma,t},\psi_{1,\sigma,t})
  \end{equation*}
  on the pair of pants. Using a similar gluing construction to the one used in Lemma \ref{lemma:pop-structural}, glue the nonpositively curved continuation cylinder connections onto the ends, producing a family $\mathfrak{G}_{\sigma}$ of connections on the pair-of-pants, with nonpositive curvature, of the form $H_{i,0,t}\d t$ in the ends $i=0,1$, and of the form $(H_{0,1,t}\# H_{1,1,t})\d t$ in the end $i=\infty$; see Figure \ref{fig:pop-com}.

  \begin{figure}[h]
    \centering
    \begin{tikzpicture}
      \draw[draw=none] (-2,0)coordinate(Z0)--(0,0)coordinate(Z1)--(2,0)--+(1,1)coordinate(X0)--+(3,1)coordinate(X1)--+(1,-1)coordinate(Y0)--+(3,-1)coordinate(Y1);

      \draw (Z0) circle (0.2 and 0.4) (Z1) circle (0.2 and 0.4) (X0) circle (0.2 and 0.4) (X1) circle (0.2 and 0.4) (Y0) circle (0.2 and 0.4) (Y1) circle (0.2 and 0.4);
      \foreach \x in {X,Y,Z} {
        \draw (\x0)+(0,0.4)coordinate(\x2)--coordinate(\x6)+(2,0.4)coordinate(\x3) +(0,-0.4)coordinate(\x4)--+(2,-0.4)coordinate(\x5);
      }
      \draw (Z3)to[out=0,in=180](X2) (X4)to[out=180,in=180](Y2) (Z5)to[out=0,in=180](Y4);

      \node at (X6) [above] {$\psi_{0,s,t}^{\sigma,-}$};
      \node at (Y6) [above] {$\psi_{1,s,t}^{\sigma,-}$};
      \node at (Z6) [above] {$\psi_{0,s,t}^{\sigma,+}\psi_{1,s,t}^{\sigma,+}$};
      
    \end{tikzpicture}
    \caption{Configuration of three cylinders and a pair-of-pants used to prove $\mathrm{P}$ commutes with continuation maps}
    \label{fig:pop-com}
  \end{figure}
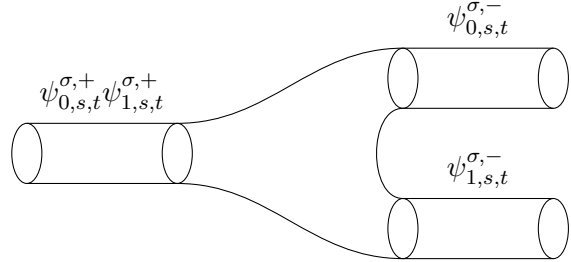

  Floer theoretic gluing/breaking arguments imply that:
  \begin{enumerate}
  \item the count of rigid elements in $\mathscr{M}(\mathfrak{G}_{0})$ equals:
    \begin{equation*}
      \mathfrak{c}_{\infty}\circ \mathrm{P}(\mathfrak{H}_{0}),
    \end{equation*}
    where $\mathfrak{c}_{\infty}$ is the continuation map associated to $\psi_{0,s,t}\psi_{1,s,t}$,
  \item the count of rigid elements in $\mathscr{M}(\mathfrak{G}_{1})$ equals:
    \begin{equation*}
      \mathrm{P}(\mathfrak{H}_{1})\circ \mathfrak{c}_{0}\otimes \mathfrak{c}_{1}.
    \end{equation*}    
  \end{enumerate}  
  The usual deformation arguments, as in Lemma \ref{lemma:deformation}, imply the 1-dimensional components of the parametric moduli space of solutions $(\sigma,u)$, $u\in \mathscr{M}(\mathfrak{G}_{\sigma})$, give a chain homotopy between the above two counts. The relevant compactness holds along the deformation since we have ensured the curvature remains non-positive.

  Since the structure maps to $\mathrm{SH}$ commutes with the continuation map $\mathfrak{c}_{\infty}$, by definition of the colimit, we conclude the desired result.
\end{proof}

\emph{Remark.} It can be shown that $\mathrm{P}:\mathrm{SH}\otimes \mathrm{SH}\to \mathrm{SH}$ is actually a commutative and associative product; see, e.g., \cite[\S5.5.1.3]{schwarz-thesis} and \cite{ritter_tqft}. Since the arguments proving this are well-known, we will omit the verification in this paper. These properties can be verified by working entirely with systems whose ideal restriction is a Reeb flow $R^{\alpha}_{st}$, which is the setting in which \cite{ritter_tqft} works.

\subsubsection{The subspace of eternal classes is an ideal}
\label{sec:subsp-etern-class}

In this section we prove Theorem \ref{theorem:ideal-pop} that $\mathrm{SH}_{e}\subset \mathrm{SH}$ is an ideal. A convenient simplification is obtained by working with the persistence module:
\begin{equation*}
  V_{s}^{\alpha}(\id)=\mathrm{HF}(R^{\alpha}_{st})
\end{equation*}
which is the precomposition of $\mathrm{HF}$ to final and cofinal functor by \S\ref{sec:finality-cofinality}.

Indeed, let $\mathfrak{e}\in \mathrm{SH}_{e}$ and $\zeta\in \mathrm{SH}$, and write:
\begin{equation*}
  \mathfrak{e}=\mathfrak{c}(\xi_{1})\text{ where }\xi_{1}\in V_{s_{1}}^{\alpha}(\id)\text{ and }\zeta=\mathfrak{c}(\xi_{2})\text{ where }\xi_{2}\in V_{s_{2}}^{\alpha}(\id),
\end{equation*}
where $\mathfrak{c}$ stands for the structure map to $\mathrm{SH}$. By the remark at the end of \S\ref{sec:pair-pants-product-1}, it follows that:
\begin{equation*}
  \mathfrak{e}\zeta=\mathrm{P}(\xi_{1},\xi_{2})=\mathfrak{c}(\xi_{3})\text{ where }\xi_{3}\in V_{s_{1}+s_{2}}^{\alpha}(\id).
\end{equation*}
Since $s_{1}$ can be made arbitrarily negative, it follows from Lemma \ref{lemma:barcode-dcomp} that $\mathfrak{e}\zeta$ is eternal, as desired.\hfill$\square$

\subsubsection{Subadditivity of spectral invariants}
\label{sec:subadd-spectr-invar}

The pair-of-pants product defined using zero curvature connections will be used in this section to prove Theorem \ref{theorem:sub-additive}.

Pick two contact isotopies $\varphi_{t},\phi_{t}$, a contact form $\alpha$, and classes $\zeta_{0},\zeta_{1}\in \mathrm{SH}$. By definition of the spectral invariants and the equivalence of the two spectral invariants in \S\ref{sec:equiv-spectr-invar}, we know:
\begin{enumerate}
\item for any $s_{0}>c(\zeta_{0};\varphi_{t})$, $\zeta_{0}$ lies in the image of:
  \begin{equation*}
    \mathrm{HF}(\varphi_{t}^{-1}\circ R_{s_{0}t})\to \mathrm{SH};
  \end{equation*}
\item for any $s_{1}>c(\zeta_{1};\phi_{t})$, $\zeta_{1}$ lies in the image of:
  \begin{equation*}
    \mathrm{HF}(R_{s_{1}t}\circ \phi_{t}^{-1})\to \mathrm{SH}.
  \end{equation*}
\end{enumerate}
Using a flat connection from \S\ref{sec:flat-conn-pair}: $$\mathfrak{H}\in \mathscr{C}(\varphi_{t}^{-1}\circ R_{(s_{0}+s_{1})t}\circ \phi_{t}^{-1};\varphi_{t}^{-1}\circ R_{s_{0}t},R_{s_{1}t}\circ \phi_{t}^{-1}),$$ and the operation $\mathrm{P}(\mathfrak{H})$ from \S\ref{sec:pair-pants-product-1}, we conclude that:
\begin{enumerate}[resume]
\item $\zeta_{0}\zeta_{1}$ lies in the image of:
  \begin{equation*}
    \mathrm{HF}(\varphi_{t}^{-1}\circ R_{(s_{0}+s_{1})t}\circ \phi_{t}^{-1})\to \mathrm{SH}.
  \end{equation*}
\end{enumerate}
Applying Lemma \ref{lemma:swap} from \S\ref{sec:equiv-spectr-invar}, we conclude that:
\begin{enumerate}[resume]
\item $\zeta_{0}\zeta_{1}$ lies in the image of:
  \begin{equation*}
    \mathrm{HF}(\phi_{t}^{-1}\circ \varphi_{t}^{-1}\circ R_{(s_{0}+s_{1})t})\to \mathrm{SH}.
  \end{equation*}
\end{enumerate}
and hence it follows that $c(\zeta_{0}\zeta_{1};\varphi_{t}\circ \phi_{t})<s_{0}+s_{1}$. Taking the infimum over $s_{0},s_{1}$ yields the desired sub-addivity.\hfill$\square$

\subsection{The unit element}
\label{sec:unit-elem-mathrmsh}

In this section we recall in \S\ref{sec:defin-unit-elem} the definition of the unit element using the PSS construction of \cite{piunikhin-salamon-schwarz-1996}. We will explain in \S\ref{sec:naturality-unit-elem} why it is sufficiently natural with respect to continuation maps so as to define a distinguished element $1\in \mathrm{SH}$, and in \S\ref{sec:unitality-unit} why it is actually the unit for the pair-of-pants product in $\mathrm{SH}$. Finally in \S\ref{sec:crit-unit-elem} we will prove Theorem \ref{theorem:when-unit-eternal} giving the criterion for the unit to be eternal in terms of continuation maps from the negative cone.

\subsubsection{PSS of the fundamental class}
\label{sec:defin-unit-elem}

Let $\psi_{s,t}$ be a continuation datum satisfying $\psi_{0,t}=\id$ and $\psi_{1,t}\in \mathscr{C}^{\times}$. Such continuation data are called \emph{PSS continuation data}. One obvious way to obtain such data is $\psi_{s,t}=\psi_{st}$ where $\sigma\mapsto \psi_{\sigma}$ is a positive path.

As in the definition of the continuation map in \S\ref{sec:cont-cylind}, we will consider perturbations of the continuation data of the form $\psi_{s,t}\delta_{s,t}$ where $\delta_{s,t}$ is compactly supported in $W$ and $\delta_{s,t}=\id$ for $s=0,1$ and $t=0$.

Consider the moduli space $\mathscr{M}(\psi_{s,t}\delta_{s,t})$ of finite energy continuation cylinders, solving the same exact equation as in \S\ref{sec:cont-cylind}. One notable difference with the present set-up is that $u\in \mathscr{M}(\psi_{s,t}\delta_{s,t})$ is actually holomorphic near in the right half of the cylinder region $s\ge 1$, and therefore has a removable singularity as $s\to\infty$. For the purposes of having a Fredholm linearization we will actually think of $\mathscr{M}(\psi_{s,t}\delta_{s,t})$ as the set of maps $v:\C\to W$ so that:
\begin{equation*}
  u(s,t)=v(e^{-2\pi (s+it)})
\end{equation*}
solves the continuation map equation \eqref{eq:cont-map}.

Let $\mathscr{M}_{d}(x;a;\psi_{s,t}\delta_{s,t})\subset \mathscr{M}(\psi_{s,t}\delta_{s,t})$ be the component of solutions $u$, with left asymptotic $\gamma$, satisfying:
\begin{enumerate}
\item $\gamma(0)=x$,
\item $\omega(u)=a$,
\item $d=n+2\mathfrak{s}^{-1}(0)\cdot [u]-\mathrm{CZ}_{\mathfrak{s}}(\gamma)$,
\end{enumerate}
where $\mathfrak{s}$ is, as usual, a section of the determinant line bundle whose zero set is disjoint from the orbits of $\psi_{1,t}$. Standard arguments similar to those in, e.g., \S\ref{sec:cont-cylind} imply that $\mathscr{M}_{d}(x;a;\psi_{s,t}\delta_{s,t})$ is a transversally cut-out $d$-manifold, and is compact for $d=0$ and compact up-to-breaking of Floer cylinders at the negative end for $d=1$.

We define, for generic perturbation term,
\begin{equation*}
  \mathrm{PSS}(\psi_{s,t}\delta_{s,t};W):=\sum_{x,a}\#\mathscr{M}_{0}(x;a;\psi_{s,t}\delta_{s,t})\tau^{a}x\in \mathrm{CF}(\psi_{1,t}).
\end{equation*}
By analysis of the moduli space with $d=1$, one shows that $\mathrm{PSS}(\psi_{s,t}\delta_{s,t};W)$ is a cycle with respect to the Floer differential. Moreover, standard deformation arguments similar to Lemma \ref{lemma:deformation} show that the homology class of the cycle is independent of $\delta_{s,t}$ and the homotopy class of the continuation data $\psi_{s,t}$ (within the space of continuation data from $\id$ to $\psi_{1,t}$). We therefore drop $\delta_{s,t}$ from the notation and write: $$\mathrm{PSS}(\psi_{s,t};W)\in \mathrm{HF}(\psi_{1,t})$$ for the resulting homology class. 

\emph{Remark.} The reason $W$ appears in the notation $\mathrm{PSS}(\psi_{s,t};W)$ is because it represents the PSS construction applied to the fundamental class Poincaré dual to $W$. We will analyze the PSS construction in greater detail in \S\ref{sec:vanish-etern-class} where we will define other cycles; in particular, we will define $\mathrm{PSS}(\psi_{s,t};L)$, where $L$ is a compact Lagrangian, which appears in the statements of Theorems \ref{theorem:odd-euler-char} and \ref{theorem:odd-euler-char-1}.

\subsubsection{Weak naturality}
\label{sec:naturality-unit-elem}

The goal in this section is to prove the following:
\begin{lemma}
  Let $\psi_{0,s,t}$ and $\psi_{1,s,t}$ be two PSS continuation data, and let $$\mathfrak{c}_{i}:\mathrm{HF}(\psi_{i,1,t})\to \mathrm{SH}$$ be the structure maps. Then: $$\mathfrak{c}_{0}(\mathrm{PSS}(\psi_{0,s,t};W))=\mathfrak{c}_{1}(\mathrm{PSS}(\psi_{1,s,t};W));$$ i.e., the image of the PSS element in $\mathrm{SH}$ is independent of the choice of PSS continuation data.
\end{lemma}
The resulting element $1=\mathfrak{c}(\mathrm{PSS}(\psi_{s,t};W))$ in $\mathrm{SH}$ is called the unit element.
\begin{proof}
  The PSS continuation data $\psi_{i,s,t}$ represents a morphism in $\mathscr{C}$ from $\id$ to $\psi_{i,1,t}$. By Lemma \ref{lemma:technical-filtering}, there exists a speed $s$ and morphisms $\eta_{i,s,t}$ going from $\psi_{i,1,t}\to R_{st}^{\alpha}$ in $\mathscr{C}$ so that the square commutes:
  \begin{equation*}
    \begin{tikzcd}
      {\id}\arrow[d,"{\psi_{0,s,t}}"]\arrow[r,"{\psi_{1,s,t}}"] &{\psi_{1,1,t}}\arrow[d,"{\eta_{1,s,t}}"]\\
      {\psi_{0,1,t}}\arrow[r,"{\eta_{0,s,t}}"] &{R_{st}^{\alpha}}.
    \end{tikzcd}
  \end{equation*}
  Let $\mathfrak{c}_{i}:\mathrm{HF}(\psi_{i,1,t})\to \mathrm{HF}(R_{st}^{\alpha})$ be the continuation morphism induced by $\eta_{i,s,t}$, as in \S\ref{sec:continuation-maps}. Straightforward Floer theory arguments show that:
  \begin{equation*}
    \begin{aligned}
      \mathfrak{c}_{0}(\mathrm{PSS}(\psi_{0,s,t};W))&=\mathrm{PSS}(\eta_{0,s,t}\#\psi_{0,s,t};W)\\
      &=\mathrm{PSS}(\eta_{1,s,t}\#\psi_{1,s,t};W)=\mathfrak{c}_{1}(\mathrm{PSS}(\psi_{1,s,t};W))
    \end{aligned}
  \end{equation*}
  The first and last inequality arguments are proved by Floer theory gluing similarly to Lemma \ref{lemma:gluing-lemma}, and the middle equality is proved by a deformation argument similar to the one in Lemma \ref{lemma:deformation}. Here $\eta_{i,s,t}\#\psi_{i,s,t}$ is the concatenated continuation data, and the homotopy class of the continuation data is independent of $i$, by assumption that the above square commutes in $\mathscr{C}$. This completes the proof.
\end{proof}

\subsubsection{Unitality}
\label{sec:unitality-unit}

In this section we prove that the unit element is unital. We will show:
\begin{lemma}
  Let $R_{t}$ be a generic contact isotopy in the positive cone (e.g., a Reeb flow) and let $\xi_{1}=\mathrm{PSS}(R_{st};W)\in \mathrm{HF}(R_{t})$. Let $\psi_{t}$ be any contact isotopy, and let $\xi_{0}\in \mathrm{HF}(\psi_{t})$. Then:
  \begin{equation*}
    \mathrm{P}(\xi_{0},\xi_{1})=\mathfrak{c}(\xi_{0}),
  \end{equation*}
  where $\mathfrak{c}:\mathrm{HF}(\psi_{t})\to \mathrm{SH}$ is the structure map. Since $1=\mathfrak{c}(\xi_{1})$, it follows from the definition of $\mathrm{P}:\mathrm{SH}\otimes \mathrm{SH}\to \mathrm{SH}$ that $1$ is the unit element.
\end{lemma}
\begin{proof}
  The argument is quite similar to the argument in \S\ref{sec:comm-with-cont} which showed that the pair-of-pants product was compatible with continuation maps. The construction we will use is illustrated in Figure \ref{fig:pop-unital}.

  Fix a Reeb flow $R_{s}$, and consider $R_{st}$, $s\in [0,1]$, as a PSS continuation datum from $\id$ to $R_{t}$. As in \S\ref{sec:comm-with-cont}, for each $\sigma\in [0,1]$, we split this continuation into two pieces:
  \begin{enumerate}
  \item[($-$)] $R_{st}^{\sigma,-}=R_{\sigma s t}$, considered as PSS continuation from $1$ to $R_{\sigma t}$,
  \item[($+$)] $R_{st}^{\sigma,+}=R_{(1-\sigma)st+\sigma t}$ considered as a continuation from $R_{\sigma t}$ to $R_{t}$.
  \end{enumerate}
  As illustrated in Figure \ref{fig:pop-unital}, we construct a Hamiltonian connection on the cylinder by gluing together three pieces:
  \begin{enumerate}
  \item A zero-curvature Hamiltonian connection:
    \begin{equation*}
      \mathfrak{H}_{\sigma}\in \mathscr{C}(\psi_{t}R_{\sigma t};\psi_{t},R_{\sigma t}),
    \end{equation*}
    on the pair-of-pants;
  \item the PSS-continuation $R^{\sigma,-}_{st}$ glued to the $1$-end of the pair-of-pants; note that this makes the $1$-end have a removable singularity;
  \item the continuation $\psi_{t}R^{\sigma,+}_{st}$ glued to the $\infty$-end of the pair-of-pants.
  \end{enumerate}
  The resulting glued connection denoted $\mathfrak{G}_{\sigma}$ is non-positively curved. Because of the removable singularity at the $1$-end, we consider $\mathfrak{G}_{\sigma}$ as a Hamiltonian connection on the cylinder, with negative end asymptotic to $\psi_{t}$ and positive end asymptotic to $\psi_{t}R_{t}$.

  We construct the initial flat connection $\mathfrak{H}^{0}$ in a special way: we let $\mathfrak{H}^{0}$ be the flat connection whose connection potential is $H_{t}\d t$, where $H_{t}$ is the generator for $\psi_{\beta(3t-1)}$. In particular, $H_{t}$ is supported where $t\in [1/3,2/3]$, and so we can embed a small disk in the large region where $H_{t}=0$. There is a biholomorphism of the cylinder and $\C^{\times}$ so that the positive, resp., negative, end of the cylinder is identified with the $\infty$, resp., $0$, end of $\C^{\times}$. Under this biholomorphism, the disk around $1$ is identified with a small disk disjoint from the strip $t\in [1/3,2/3]$; see Figure \ref{fig:popsurface} and \ref{fig:floer-cyl-pop}.

  \begin{figure}[h]
    \centering
    \begin{tikzpicture}
      \draw[draw=none] (-2,0)coordinate(A0)--(0,0)coordinate(B0)--(2,0)--+(1,1)coordinate(X0)--+(1,-1)coordinate(Y0)--+(3,-1)coordinate(C0);

      \foreach \x in {A,B,X,Y,C} {
        \draw (\x0) circle (0.2 and 0.4);
        \path (\x0)+(0,0.4)coordinate(\x1) +(0,-0.4)coordinate(\x2);
      }
      \draw (B1)to[out=0,in=180](X1) (X2)to[out=180,in=180](Y1) (B2)to[out=0,in=180](Y2);
      \draw (A1)--coordinate(AB1)(B1) (A2)--(B2);
      \draw (Y1)--coordinate(YC1)(C1) (Y2)--(C2);
      
      \node at (X0) [right,shift={(0.2,0)}] {$\psi_{t}$};
      \node at (YC1) [above] {$R_{st}^{\sigma,-}$};
      \node at (AB1) [above] {$\psi_{t}R_{st}^{\sigma,+}$};
      
    \end{tikzpicture}
    \caption{Configuration of a continuation cylinder, a PSS cylinder, and a pair-of-pants used to prove the unit element is unital for the pair-of-pants product.}
    \label{fig:pop-unital}
  \end{figure}
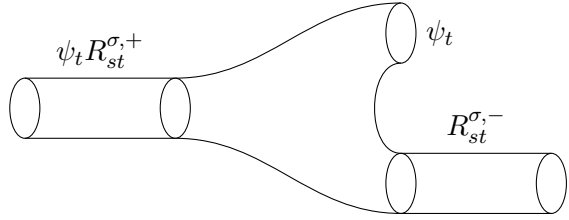

  \begin{figure}[h]
    \centering
    \begin{tikzpicture}
      \def\L{7}
      \def\H{0.7}
      \def\W{0.2}
      \draw (0,0) circle ({\W} and \H) +(0,\H) coordinate(X1)+(0,-\H)coordinate(X2)+(-\W,0)node[left]{$\psi_{t}$} (\L,0) circle ({\W} and \H) +(0,\H) coordinate(Y1)+(0,-\H)coordinate(Y2)+(\W,0)node[right]{$\psi_{t}$};
      \draw (0,0) +(40:{\W} and \H) coordinate(X3) +(10:{\W} and \H) coordinate(X4) (\L,0) +(40:{\W} and \H) coordinate(Y3) +(10:{\W} and \H) coordinate(Y4);
      \fill[pattern=north west lines] (X3)--(Y3)--(Y4)--(X4)--cycle;
      \draw (X1)--(Y1) (X2)--(Y2) (X3)--(Y3) (X4)--(Y4);
      \draw[fill,pattern=north west lines] ({\L/2},{-\H/2}) circle ({\H/3}) node[draw,circle,inner sep=1pt,fill=black]{};
    \end{tikzpicture}
    \caption{The initial pair-of-pants $\mathfrak{H}^{0}$ is built using a Floer cylinder. The connection potential $\mathfrak{a}=H_{t}\d t$ is supported on the shaded strip $t\in [1/3,2/3]$, and is identically zero on the shaded circle; $H_{t}$ is the generator of $\psi_{\beta(3t-1)}$.}
    \label{fig:floer-cyl-pop}
  \end{figure}
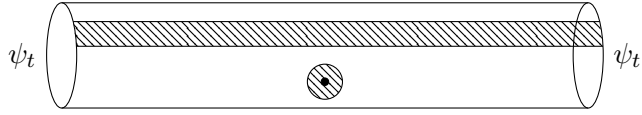
  
  By this construction, it follows easily that the rigid Floer theoretic count of elements in $\mathscr{M}(\mathfrak{G}^{0})$ is precisely the continuation map $$\mathfrak{c}':\mathrm{HF}(\psi_{t})\to \mathrm{HF}(\psi_{t}R_{t}).$$ Here we note that $R^{0,-}_{st}=\id$ is the identity, and hence gluing the PSS-continuation to the $1$-end does nothing in this case.

  Applying part \ref{Sc} of Lemma \ref{lemma:pop-structural}, we can find a smooth family $\mathfrak{H}_{\sigma}$ extending our $\mathfrak{H}_{0}$. By the usual gluing arguments, as in \S\ref{sec:comm-with-cont}, the rigid count of elements in $\mathscr{M}(\mathfrak{G}_{1})$ is equal to $\mathrm{P}(\mathfrak{H}_{1})(-,\mathrm{PSS}(R_{st};W))$. Then by the deformation argument, it follows that:
  \begin{equation*}
    \mathfrak{c}'(-)=\mathrm{P}(\mathfrak{H}_{1})(-,\mathrm{PSS}(R_{st};W)),
  \end{equation*}
  and the desired result holds after postcomposing with the maps to $\mathrm{SH}$.
\end{proof}

\subsubsection{Criterion for the unit element to be eternal}
\label{sec:crit-unit-elem}

We prove Theorem \ref{theorem:when-unit-eternal}. If the ideal restriction of $\varphi_{t}$ lies in the negative cone, then there is a continuation data $\varphi_{s,t}$ so that $\varphi_{0,t}=\varphi_{t}$ and $\varphi_{1,t}=R_{-\epsilon t}^{\alpha}$, where $\epsilon$ is sufficiently small. Thus, if the unit lies in the image of the structure map $\mathrm{HF}(\varphi_{t})\to \mathrm{SH}$, then the unit also lies in the image of $\mathrm{HF}(R_{-\epsilon t}^{\alpha})\to \mathrm{SH}$.

Applying the pair-of-pants product, it follows that $1=1^{k}$ lies in the image of $\mathrm{HF}(R_{-k\epsilon t}^{\alpha})\to \mathrm{SH}$. Since $k\epsilon$ can be made arbitrarily large, it follows from Lemma \ref{lemma:barcode-dcomp} that the unit is eternal. Thus we conclude Theorem \ref{theorem:when-unit-eternal}.\hfill$\square$

\subsection{Geometry of the spectral invariant of the unit}
\label{sec:geom-spec-invar-unit}

In this section we use the properties of $c_{\alpha}(1;\varphi_{t})$ established in the previous sections to explore its relation to the geometry of the universal cover of the contactomorphism group.

\subsubsection{Proof of the comparison with order measurements}
\label{sec:comp-with-order-1}

This section contains the proof of Theorem \ref{theorem:order-comp} on the comparison with the measurements of \cite{allais-arlove}. First, if $R_{st}\le \varphi_{t}$, there is a factorization:
\begin{equation*}
  \mathrm{HF}(\varphi_{t}^{-1}\circ R_{(s-\epsilon)t})\to \mathrm{HF}(R_{-\epsilon t})\to \mathrm{SH}.
\end{equation*}
Since $R_{-\epsilon t}$ lies in the negative cone, the rightmost structure map does not hit $1\in \mathrm{SH}$ by Theorem \ref{theorem:when-unit-eternal}. Thus $c_{\alpha}(\varphi_{t})\ge s-\epsilon$. Taking the limit $\epsilon \to 0$ and the supremum over $s$ yields $c_{-}^{\alpha}(\varphi_{t})\le c_{\alpha}(\varphi_{t})$.

On the other hand, if $\varphi_{t}\le R_{st}$, then there is a factorization:
\begin{equation*}
  \mathrm{HF}(R_{+\epsilon t})\to \mathrm{HF}(\varphi_{t}^{-1}\circ R_{(s+\epsilon )t})\to \mathrm{SH}.
\end{equation*}
Since the unit always lies in the image of $\mathrm{HF}(R_{+\epsilon t})\to \mathrm{SH}$, by its definition in terms of rigid PSS cylinders in \S\ref{sec:unit-elem-mathrmsh}, it follows that $c_{\alpha}(\varphi_{t})\le s+\epsilon$. Taking the limits as above yields $c_{\alpha}(\varphi_{t})\le c_{+}^{\alpha}(\varphi_{t})$.

The finiteness statement follows from \cite{allais-arlove} and the fact that ideal boundaries of manifolds $W$ satisfying $1\not\in \mathrm{SH}_{e}$ are orderable.\hfill$\square$

\subsubsection{Spectral oscillation and the shape invariant}
\label{sec:example-cont-isot}

Before we begin the proof of Theorem \ref{theorem:shape-invar}, we state a lemma concerning the interaction between free homotopy classes and the Floer cohomology groups.

\begin{lemma}\label{lemma:technical-contractible}
  Let $\psi_{t}\in \mathscr{C}^{\times}$ be a contact-at-infinity isotopy of $W$ so that: $$\mathrm{HF}(\psi_{t})\to \mathrm{SH}$$ hits a non-zero unit element $1$. Then $\psi_{t}$ has contractible orbits. Similarly, if $\psi_{s,t}$ is continuation data with $\psi_{i,t}\in \mathscr{C}^{\times}$ for $i=0,1$, and $\psi_{s,t}$ never develops any contractible orbits outside of some compact set, then:
  \begin{equation*}
    \mathrm{HF}(\psi_{0,t})\to \mathrm{SH}\text{ hits the unit}\iff\mathrm{HF}(\psi_{1,t})\to \mathrm{SH}\text{ hits the unit}.
  \end{equation*}
\end{lemma}
\begin{proof}
  The first part is a straightforward application of the natural direct sum decomposition: $$\mathrm{HF}(\psi_{t})=\mathrm{HF}(\psi_{t};\kappa_{1})\oplus \mathrm{HF}(\psi_{t};\kappa_{2})\oplus \dots,$$ where the sum ranges over all free homotopy classes of loops in $W$, and:
  \begin{equation*}
    \mathrm{HF}(\psi_{t};\kappa)=\text{homology of }\mathrm{CF}(\psi_{t};\kappa),
  \end{equation*}
  where the right-hand side is the subcomplex generated by fixed points of $\psi_{1}$ so that the orbit $t\mapsto \psi_{t}(x)$ is in the class $\kappa$.

  One then observes that $\mathrm{PSS}(R^{\alpha}_{st};W)\in \mathrm{HF}(R^{\alpha}_{st})$ has vanishing projection to $\mathrm{HF}(R^{\alpha}_{st};\kappa)$ for every non-trivial class $\kappa$. By the argument in \S\ref{sec:naturality-unit-elem}, if the unit lies in the image $\mathrm{HF}(\psi_{t})\to \mathrm{SH}$, then the image of $1$ in $\mathrm{HF}(R^{\alpha}_{st})$ equals the PSS class $\mathrm{PSS}(R^{\alpha}_{st};W)$ for $s$ sufficiently large.

  If $\psi_{t}$ has no contractible orbits, then we conclude by naturality of the direct sum decomposition that $\mathrm{PSS}(R^{\alpha}_{st};W)$ has a vanishing projection to: $$\mathrm{HF}(\psi_{t};\text{contractible class}),$$ and hence $\mathrm{PSS}(R^{\alpha}_{st};W)$ is zero, contradicting the assumption that the unit was non-zero.

  The second part of the theorem follows from the arguments in \cite{uljarevic-zhang-JFPTA-2022}, similarly to Theorem \ref{theorem:spectral}. One shows that the continuation map:
  \begin{equation*}
    \mathrm{HF}(\psi_{0,t})\to \mathrm{HF}(\psi_{1,t})
  \end{equation*}
  acts isomorphically on the summand generated by contractible loops, and the same arguments used in the first part of the proof implies the stated result.
\end{proof}

\begin{proof}[Proof of Theorem \ref{theorem:shape-invar}]
  If $G(p)$ is negative for at least one $p\in S^{n-1}$, we claim there is a non-constant affine function $\ell(p)=a\cdot p+b$ with $\abs{a}>b>0$, and a number $\delta$ so that for all $p\in S^{n-1}$:
  \begin{enumerate}
  \item $G(p)<\delta\implies G(p)<\ell(p),$
  \item $G(p)\ge 0\implies \ell(p)\ge \delta$,
  \end{enumerate}
  one can pick $\ell$ so that $\set{G<0}$ contains $\ell^{-1}((-\infty,0])$. Replacing $\ell$ by $\epsilon \ell$, we may suppose: $$G(p)\le 0\implies G(p)<\ell,$$ and then the claim holds for small enough $\delta$.

  Let $f$ be a function so that $f(x)=x$ for $x\le 0$ and $f(x)\le \max\set{x,\delta}$. Then it is clear that $f(G(p))\le \ell(p)$ holds for $p\in S^{n-1}$.

  Given a function $G:S^{n-1}\to \R$, introduce the notation $\mathrm{HF}(G(p))$ for the Floer cohomology of time-1 map of any contact-at-infinity system $\psi_{t}$ on $T^{*}T^{n}$ whose generating Hamiltonian restricts to $G(p)$ on the unit sphere bundle, and is $1$-homogeneous outside the unit disk bundle. To be precise, as in \S\ref{sec:funct-struct-floer}, we actually define:
  \begin{equation*}
    \mathrm{HF}(G(p))=\lim_{\epsilon \to 0+}\mathrm{HF}(G(p)+\epsilon),
  \end{equation*}
  and $\mathrm{HF}(G(p)+\epsilon)$ is technically defined as the category theory limit of $\mathrm{HF}(\psi_{t})$ over the contractible choice of a non-degenerate extension $\psi_{t}$.

  Then we have the following diagram where the morphisms are continuation maps associated to the continuation data given by linear interpolation:
  \begin{equation*}
    \begin{tikzcd}
      {\mathrm{HF}(f(G(p)))}\arrow[dr,"{}",out=-40,in=140]\arrow[r,"{}"] &{\mathrm{HF}(G(p))}\\
      {\mathrm{HF}(a\cdot p)}\arrow[r,"{}"] &{\mathrm{HF}(a\cdot p+b)}
    \end{tikzcd}
  \end{equation*}
  The idea for the rest of the proof is simple:
  \begin{enumerate}[label=(\arabic*)]
  \item\label{sh1} apply the first part of Lemma \ref{lemma:technical-contractible} to conclude the unit does not lie in the image of $\mathrm{HF}(a\cdot p)\to \mathrm{SH}$,
  \item\label{sh2} apply the second part of Lemma \ref{lemma:technical-contractible} to conclude the unit does not lie in the image of $\mathrm{HF}(a\cdot p+b)\to \mathrm{SH}$, provided $b<\abs{a}$,
  \item\label{sh3} use the previous item and the above diagram to conclude the unit does not lie in the image of $\mathrm{HF}(f(G(p)))\to \mathrm{SH}$.
  \item\label{sh4} apply the second part of Lemma \ref{lemma:technical-contractible} once again to conclude the unit does not lie in the image of $\mathrm{HF}(G(p))\to \mathrm{SH}$.
  \end{enumerate}
  Before establishing these items, let us explain how it proves the theorem. It is clear that, with the notation of the proof:
  \begin{equation}\label{eq:shape-c}
    c_{\alpha}(1;\varphi_{t})=\inf\set{s:\text{unit lies in image of }\mathrm{HF}(s-H(p))\to \mathrm{SH}}.
  \end{equation}
  It is clear that, if $s>\max_{\abs{p}=1} H(p)$, then the unit will lie in the image by definition. On the other hand, if $s<\max_{\abs{p}=1} H(p)$, then $s-H(p)$ is negative in at least one point, and so \ref{sh4} with $G=s-H(p)$ implies the unit does not lie in the image. The combination of these two observations establishes \eqref{eq:shape-c}.

  Items \ref{sh1} follows from the fact that $a\cdot p$ already defines a smooth one-homogeneous Hamiltonian function on all of $T^{*}T^{n}$, and generates the canonical lift of a translation. Thus the flow cannot have any contractible orbits.

  Items \ref{sh2} and \ref{sh4} follow the same argument. In general, if $G:\R^{n}\to \R$ is one-homogeneous, then the Hamiltonian vector field associated to $G(p)$ equals $\bd_{p_{i}}G(p)\bd_{q_{i}}$. This has a contractible orbit if and only if $G$ has a critical point, which must occur with critical value zero since $G$ is one-homogeneous. This is equivalent to $G|_{S^{n-1}}$ having zero as a critical value.

  It is clear that $a\cdot p+b$ has no critical points with value zero if $0<b<\abs{a}$, and thus \ref{sh2} holds.

  Similarly, since $G|_{S^{n-1}}$ is assumed to have zero as a regular value (otherwise the recipe forces us to replace $G$ by $G+\epsilon$), and $f(x)=x$ for $x\le 0$, the linear deformation $G_{s}=(1-s)f(G)+sG$ also never has zero as a critical value. Thus \ref{sh4} holds. Item \ref{sh3} is obvious, and so the proof is complete.
\end{proof}

\emph{Remark.} Let $\phi_{t},\psi_{t}$ be generated by $H(p),G(p)$, respectively. If there is a morphism from $\phi_{t}$ to $\psi_{t}$ in $\mathscr{C}$, there is a morphism from $\id$ to $\phi_{t}^{-1}\psi_{t}$ in $\mathscr{C}$.

It then follows from the PSS construction of \S\ref{sec:defin-unit-elem} that the unit lies in the image of $\mathrm{HF}(\phi_{t}^{-1}\psi_{t})\to \mathrm{SH}$, and hence:
\begin{equation*}
  c_{\alpha}(1;\psi_{t}^{-1}\phi_{t})\le 0
\end{equation*}
However, $\psi_{t}^{-1}\phi_{t}$ is generated by $H(p)-G(p)$, and hence our result shows:
\begin{equation*}
  \max_{p}H(p)-G(p)\le 0\implies H(p)\le G(p)\text{ for all $p\in S^{n-1}$}.
\end{equation*}
Thus we obtain a Floer theory proof of the result of \cite{eliashberg-polterovich-GAFA-2000} that $\phi_{t}\le \psi_{t}$ (in the order relation on the universal cover of the contactomorphism group) if and only if $H(p)\le G(p)$ holds for all $p\in S^{n-1}$ (the ``if'' direction is obvious and we have just proved the ``only if'' direction).

\subsubsection{Proof of Theorem \ref{theorem:systole-osc}}
\label{sec:proof-of-systole-osc}

In this section we bound from below the spectral invariant of loop $\phi_{t}$ generated by the Hamiltonian $H=p$ on $W\times T^{*}S^{1}$, where $W$ is a Liouville manifold.

Let $k\ne 0$. Then $\phi_{t}^{-k}$ generates a system without any contractible orbits, and so Lemma \ref{lemma:technical-contractible} implies that $1$ does not lie in the image of $\mathrm{HF}(\phi_{t}^{-k})\to \mathrm{SH}$.

Now consider the systems $\phi_{t}^{-k}\circ R_{st}^{\alpha}$, $s>0$, where $\alpha$ is an arbitrary contact form on the ideal boundary $W\times T^{*}S^{1}$ (with the correct coorientation, of course). If $s$ is a speed for which this system develops a contractible orbit at infinity, then $R_{st}^{\alpha}$ must have an orbit in the free homotopy class of the $k$th iterate of $\set{w}\times S^{1}$. To see why, observe that, if there is a point $x$ so that $t\mapsto \gamma(t)=\phi_{t}^{-k}(R_{st}^{\alpha}(x))$ is a contractible loop, the loop $\gamma$ is homotopic with fixed endpoints to the loop:
\begin{equation*}
  \eta(t)=\left\{
    \begin{aligned}
      R_{s2t}^{\alpha}(x)\text{ for }t\le 1/2,\\
      \phi_{(2t-1)}^{k}(x)\text{ for }t\ge 1/2,\\
    \end{aligned}
  \right.
\end{equation*}
If $\gamma$ are contractible, then $t\mapsto R^{\alpha}_{st}(x)$ and $t\mapsto \phi_{t}^{-k}(x)$ must be inverse elements of $\pi_{1}(W\times T^{*}S^{1},x)$. Hence $s$ must be the speed of a Reeb orbit lying in the free homotopy class of the $k$th iterate of $\set{w}\times S^{1}$, as desired.

Thus, the continuation data from $\phi_{t}^{-k}$ to $\phi_{t}^{-k}\circ R^{\alpha}_{st}$ obtained by linearly increasing the speed develops no contractible orbits at infinity if $s<\mathrm{sys}_{k}(R^{\alpha})$. Thus Lemma \ref{lemma:technical-contractible} implies that the unit does not lie in the image of:
\begin{equation*}
  \mathrm{HF}(\phi_{t}^{-k}\circ R^{\alpha}_{st})\to \mathrm{SH}
\end{equation*}
if $s<\mathrm{sys}_{k}(R^{\alpha})$. It therefore follows that:
\begin{equation*}
  c_{\alpha}(1;\phi_{t}^{k})\ge \mathrm{sys}_{k}(R^{\alpha}).
\end{equation*}
This completes the proof of Theorem \ref{theorem:systole-osc}.\hfill$\square$

\subsubsection{Displacement and spectral invariants}
\label{sec:displ-spectr-invar-1}

We prove Theorem \ref{theorem:unit-displace}. The arguments are quite similar to analogous arguments in the setting of symplectic spectral invariants; see, e.g., \cite{schwarz-PJM-2000}.

Let $U\subset Y$ be the open set displaced by the ideal restriction of $\psi_{t}$. Suppose that $\psi_{t}$ has no orbits at infinity; this can be achieved by a small perturbation $\psi_{t}=R_{-\epsilon t}^{\alpha}\psi_{t}$ so that $\psi_{t}\in \mathscr{C}^{\times}$. By the same small perturbation, we may assume that $c_{\alpha}(1;\psi_{t})<0$. Moreover, for any compact set $K\subset U$, we can pick the perturbation small enough that $\psi_{t}$ still displaces $K$ from itself.

Let $\varphi_{t}$ be a contact isotopy supported in $U$. Let $\rho_{t}$ be a non-negative contact isotopy supported in $U$, chosen large enough that there is a continuation datum from $\varphi_{t}$ to $\rho_{t}$. Thus:
\begin{equation}\label{eq:non-neg-trick}
  c_{\alpha}(1;\varphi_{t})\le c_{\alpha}(1;\rho_{t}).
\end{equation}
Apply the subadditivity to conclude:
\begin{equation}\label{eq:displace-trick}
  c_{\alpha}(1;\rho_{t})\le c_{\alpha}(1;\psi_{t}\rho_{t})+c_{\alpha}(1;\psi_{t}^{-1}).
\end{equation}
Since $c_{\alpha}(1;\psi_{t})<0$, the unit lies in the image of $\mathrm{HF}(\psi_{t}^{-1})\to \mathrm{SH}$. Consider the continuation data:
\begin{equation*}
  s\mapsto (\psi_{t}\circ \rho_{(1-s)t})^{-1}.
\end{equation*}
We claim that, for each $s$, this system never develops orbits at infinity. If it did, then $\psi_{t}\circ \rho_{(1-s)t}$ would have a orbit. However this cannot happen since $\psi_{1}$ displaces the support of $\rho_{(1-s)t}$ and $\psi_{t}$ has no orbits at infinity. Thus the claim is established.

Applying the second part of Lemma \ref{lemma:technical-contractible}, we conclude that the unit lies in the image of:
\begin{equation*}
  \mathrm{HF}((\psi_{t}\rho_{t})^{-1})\to \mathrm{SH}.
\end{equation*}
Hence $c_{\alpha}(1;\psi_{t}\rho_{t})\le 0$. Thus \eqref{eq:non-neg-trick} and \eqref{eq:displace-trick}  yield $c_{\alpha}(1;\varphi_{t})\le c_{\alpha}(1;\psi_{t}^{-1})$, as desired.

It remains only to prove that:
\begin{equation*}
  c_{\alpha}(U)=\sup\set{c_{\alpha}(1;\varphi_{t}):\text{$\varphi_{t}$ is supported in $U$}}
\end{equation*}
is strictly positive if $U$ is non-empty. Take $\varphi_{t}$ autonomous and supported in $U$ so that its contact Hamiltonian is non-negative and strictly positive in at least one point. It is then clear (e.g., by the ergodic trick of \cite{eliashberg-polterovich-GAFA-2000}) that there exist a product of conjugates $g\varphi_{t}g^{-1}$ whose contact Hamiltonian is strictly positive everywhere. Thus:
\begin{equation*}
  0<c_{\alpha}(1;\prod_{i=1}^{k}g_{i}\varphi_{t}g_{i}^{-1})\le \sum_{i=1}^{k}c_{\alpha}(1;g_{i}\varphi_{t}g_{i}^{-1}).
\end{equation*}
The proof is complete provided we can show:
\begin{equation}
  \label{eq:conju-pos}
  c_{\alpha}(1;g\varphi_{t}g^{-1})>0\iff c_{\alpha}(1;\varphi_{t})>0.
\end{equation}
for all $g\in \mathrm{Cont}_{0}(Y)$.

Replace $\varphi_{t}$ by $R^{\alpha}_{-\epsilon t}\varphi_{t}\in \mathscr{C}^{\times}$. Then there is a zero curvature connection: $$\mathfrak{H}\in \mathscr{C}(g\varphi^{-1}_{t}R^{\alpha}_{\epsilon t}g^{-1},\varphi_{t}^{-1}R_{\epsilon t}^{\alpha}).$$ Using the operation $\mathrm{C}(\mathfrak{H})$ from \S\ref{sec:definition-map}, and its inverse, we conclude that:
\begin{enumerate}
\item the unit lies in the image of $\mathrm{HF}(g\varphi^{-1}_{t}R^{\alpha}_{\epsilon t}g^{-1})\to \mathrm{SH}$, if and only if,
\item the unit lies in the image of $\mathrm{HF}(\varphi^{-1}_{t}R^{\alpha}_{\epsilon t})\to \mathrm{SH}$.
\end{enumerate}
In other words:
\begin{equation*}
  c(1;gR^{\alpha}_{-\epsilon t}\varphi_{t}g^{-1})\le 0\iff c(1;R^{\alpha}_{-\epsilon t}\varphi_{t})\le 0.
\end{equation*}
Taking the limit $\epsilon\to 0$, we conclude that:
\begin{equation*}
  c(1;g\varphi_{t}g^{-1})\le 0\iff c(1;\varphi_{t})\le 0,
\end{equation*}
which is equivalent to \eqref{eq:conju-pos}. This completes the proof. \hfill$\square$

\subsubsection{Conjugation invariant measurements}
\label{sec:conjugation-invariant-proof}

We prove Theorem \ref{theorem:conjugation} on the conjugation invariance of the measurement $\ell$. The argument is quite similar to the argument given in \S\ref{sec:displ-spectr-invar-1}.

Indeed, $g\varphi_{t}^{-1}g^{-1}\phi_{t}^{k}$ and $g\varphi_{t}^{-1}\phi_{t}^{k}g^{-1}$ have the same time-1 maps in $\text{UH}$, since $\phi_{t}^{k}$ lies in the center of $\mathrm{UH}$, and the same argument given above proves that:
\begin{equation*}
  \mathrm{HF}(g\varphi_{t}^{-1}\phi_{t}^{k}g^{-1})\to \mathrm{SH}\text{ hits }\zeta\iff \mathrm{HF}(\varphi_{t}^{-1}\phi_{t}^{k})\to \mathrm{SH}\text{ hits }\zeta.
\end{equation*}
This we conclude that $\ell(g\varphi_{t}g^{-1})=\ell(\varphi_{t})$, as desired.\hfill$\square$

\section{Vanishing of eternal classes and smooth displaceability}
\label{sec:vanish-etern-class}

The strategy used to prove Theorem \ref{theorem:smooth-displaceability} is to analyze the chain homotopy class of continuation maps:
\begin{equation}\label{eq:continuation-map}
  \mathrm{CF}(\psi_{0,t})\to \mathrm{CF}(\psi_{1,t}),
\end{equation}
when $\psi_{0,t}$ admits a morphism to $\id$ and $\psi_{1,t}$ admits a morphism from $\id$, in the category $\mathscr{C}$, and \eqref{eq:continuation-map} is obtained by concatenating these morphism.

The idea is to deform the data used to define the continuation map to replace it by an intersection theory problem; briefly, \eqref{eq:continuation-map} is chain homotopic to a map counting intersections between the PSS moduli spaces introduced in \cite{piunikhin-salamon-schwarz-1996}; we have seen these moduli spaces already in \S\ref{sec:defin-unit-elem} when defining the unit element and will review them in greater details setting below. The upshot of this intersection theoretic approach is that, when the compact part of $W$ is smoothly displaceable, the intersection number (and hence the continuation map) is equal to zero.

\subsection{PSS and the positive cone}
\label{sec:pss-positive-cone}

Let $\psi_{s,t}$ be PSS continuation data with $\psi_{0,t}=\id$ and $\psi_{1,t}\in \mathscr{C}^{\times}$, as in \S\ref{sec:defin-unit-elem}.

Let $\delta_{s,t}$ be a perturbation term with compact support in $W$ and so $\delta_{s,t}=0$ for $s=0,1$ and $t=0$. Then $\kappa_{s,t}=\psi_{s,t}\delta_{s,t}$ is again valid continuation data with the same ideal restriction as $\psi_{s,t}$. As in \S\ref{sec:cont-cylind} and \S\ref{sec:defin-unit-elem}, define the reparametrization:
\begin{equation*}
  \xi_{s,t}=\kappa_{\beta(1-s),\beta(3t-1)},
\end{equation*}
where $\beta$ is the standard cut-off function, and consider the moduli space $\mathscr{M}_{+}(\kappa_{s,t})$ of solutions to:
\begin{equation}\label{eq:PSS-positive}
  \left\{
    \begin{aligned}
      &v:\C\to W\text{ smooth},\\
      &u(s,t)=v(e^{-2\pi(s+it)}),\\
      &\bd_{s}u-\rho(t)Y_{s,t}(u)+J(u)(\bd_{t}u-X_{s,t}(u))=0,\\
    \end{aligned}
  \right.
\end{equation}
where $Y_{s,t},X_{s,t}$ and $\rho(t)=\beta(3-3t)$ are as in \S\ref{sec:cont-cylind}. The plus sign signifies that the removable singularity is at the $s=+\infty$ end; see Figure \ref{fig:PSS-positive}. The important thing to notice is that:
\begin{equation*}
  \left\{\begin{aligned}
    &X_{s,t}=Y_{s,t}=0\text{ for $s\ge 1$},\\
    &X_{s,t}=X_{t}\text{ and }Y_{s,t}=0\text{ for $s\le 0$},\\
  \end{aligned}\right.
\end{equation*}
where $X_{t}$ is the generator of $\psi_{1,\beta(3t-1)}$.

\begin{figure}[h]
  \centering
  \begin{tikzpicture}
    \draw (0,0.5) circle (0.2 and 0.5);
    \node at (-0.2,0.5) [left] {$\gamma$};
    \draw (0,1)--(3,1)to[out=0,in=0,looseness=2]coordinate[pos=0.5](X)(3,0)--(0,0);
    \node at (X)[draw,circle,fill,inner sep=1pt]{};
  \end{tikzpicture}
  \caption{An element in $\mathscr{M}_{+}(\kappa_{s,t})$ asymptotic to an orbit $\gamma$.}
  \label{fig:PSS-positive}
\end{figure}
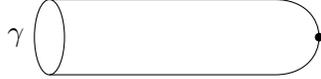

\begin{lemma}\label{lemma:PSS-genericity}
  For generic perturbation term $\delta_{s,t}$, the moduli space $\mathscr{M}_{+}(\kappa_{s,t})$ is cut transversally and is a smooth manifold. The local dimension of $\mathscr{M}_{+}$ near a point $v$ is given by:
  \begin{equation*}
    \dim_{\mathscr{M}_{+}}(v)=n-\mathrm{CZ}_{\mathfrak{s}}(\gamma)+2\mathfrak{s}^{-1}(0)\cdot v,
  \end{equation*}
  where $\mathrm{CZ}_{\mathfrak{s}}$ is as in \S\ref{sec:chern-classes-conley}.

  Moreover the total evaluation map $\mathscr{M}_{+}\times \C\to W$ and also the evaluation at zero $\mathscr{M}_{+}\to W$ can be assumed to be transverse to any countable collection of smooth maps $f:X\to W$.
\end{lemma}
\begin{proof}
  The dimension formula is well-known; see, e.g., \cite{schwarz-thesis,cant-thesis-2022}. The argument for transversality is the same as Lemma \ref{lemma:transversality-floer}.
\end{proof}

The evaluation at zero $v\mapsto v(0)$ is used in the aforementioned intersection theory interpretation of \eqref{eq:continuation-map}. As usual in the semipositive framework, the total evaluation will be assumed to be transverse to the pseudocycle of simple $J$-holomorphic spheres.

It will be important to observe that:
\begin{lemma}\label{lemma:PSS-compactness-1}
  The set of points of the form $v(0)$ where $v\in \mathscr{M}_{+}(\kappa_{s,t})$ and which satisfy $\omega(v)\le A$ is contained in a compact set (depending on $\psi_{s,t}$ and $A$).
\end{lemma}
\begin{proof}
  Suppose not, so there is a sequence $v_{n}\in \mathscr{M}_{+}(\kappa_{s,t})$ with $\omega(v_{n})\le A$ and where $v_{n}(0)$ leaves every compact set. This sequence has finite energy (here we use that $\psi_{s,t}$ is continuation data, so that the a priori energy estimate of \S\ref{sec:cont-cylind} holds). Hence bubbling analysis can be applied: after passing to a subsequence, there exist a sequence of small subdomains $\Sigma_{n}\subset \C$, contained in a fixed compact set, such that $v_{n}$ has a uniformly bounded derivative on the complement of $\Sigma_{n}$, and the restriction of $v_{n}$ to appropriate rescalings of $\Sigma_{n}$ converges uniformly to a collection of holomorphic spheres (we are describing the formation of bubble trees). Here we measure sizes using a Riemannian metric on $W$ which is invariant under the Liouville flow in the convex end.

  To be more precise about the formation of bubbles, there are embeddings $i_{n}:\Sigma_{n}\to \mathbb{C}P^{1}$ and limit holomorphic maps $w_{\infty}:\mathbb{C}P^{1}\to W$ so that the sup-distance between $v_{n}|_{\Sigma_{n}}$ and $w_{\infty}\circ i_{n}$ converges to zero. See, e.g., \cite{cant-chen-arXiv-2023}. Since all holomorphic spheres are contained in a fixed compact set, it follows that $v_{n}(\Sigma_{n})$ is contained in a fixed compact set, independent of $n$.

  If $D(R)$ is large enough, then $v_{n}|_{D(2R)\setminus D(R)}$ is a reparametrization of a sequence of finite length Floer cylinders for the Hamiltonian vector field $X_{t}$, with bounded derivative and bounded modulus. The maximum principle from, e.g., \cite[\S2.2.5]{brocic-cant-JFPTA-2024}, proves that $v_{n}(D(2R)\setminus D(R))$ must also be contained in a fixed compact set. The point $0$ can be joined to $\bd D(R)$ or $\bd \Sigma_{n}$ by an arc $a_{n}(t)$, of length at most $R$, contained in the complement of $\Sigma_{n}$. Since the derivative is uniformly bounded along $a_{n}(t)$, it follows that $v_{n}(0)$ remains a uniformly bounded distance from $v_{n}(D(2R)\setminus D(R))\cup v_{n}(\Sigma_{n})$; the latter set is contained in a fixed compact set, and hence $v_{n}(0)$ cannot diverge to infinity. This completes the proof.
\end{proof}

\subsubsection{PSS elements in symplectic cohomology}
\label{sec:pss-elements}

The usual application of the moduli spaces considered in \S\ref{sec:pss-positive-cone} is constructing PSS elements in $\mathrm{HF}(\psi_{1,t})$.

Fix a smooth proper map $f:P\to W$ (e.g., $P=W$ and $f=\id$, as in \S\ref{sec:defin-unit-elem}), and consider the fiber product of pairs $(p,v)\subset X\times \mathscr{M}_{+}(\kappa_{s,t})$ satisfying the incidence condition $f(p)=v(0)$.

Pick a generic perturbation term $\delta_{s,t}$ so that the evaluation of $\mathscr{M}_{+}(\kappa_{s,t})$ at zero is transverse to $f$ and the total evaluation is transverse to the simple $J$-holomorphic spheres. Consider the component $\mathscr{M}_{+,d}(x;a;\kappa_{s,t};f)$ of the parametric moduli space consisting of pairs $(x,v)$ satisfying:
\begin{enumerate}
\item $\gamma(0)=x,$ where $\gamma$ is the left asymptotic of $u$,
\item $\omega(v)=a$,
\item $\dim(P)+\mathrm{CZ}_{\mathfrak{s}}(\gamma)+2\mathfrak{s}^{-1}(0)\cdot [v]=d,$
\end{enumerate}

Then $\mathscr{M}_{+,0}(x;a;\kappa_{s,t};f)$ is a compact zero-dimensional manifold, and we define:
\begin{equation*}
  \mathrm{PSS}(\kappa_{s,t};f)=\sum_{x,a}\#\mathscr{M}_{+,0}(x;a;\kappa_{s,t};f) \tau^{a}x.
\end{equation*}
Then, following \cite{schwarz-thesis,piunikhin-salamon-schwarz-1996}, this count defines a Floer cycle in $\mathrm{CF}(\psi_{1,t})$. The usual deformation arguments, as in Lemma \ref{lemma:deformation}, prove the homology class of the cycle is invariant under deformations of the continuation data $\psi_{s,t}$ and is independent of the perturbation term $\delta_{s,t}$ used (recall $\kappa_{s,t}=\psi_{s,t}\delta_{s,t}$).

The resulting element of $\mathrm{HF}(\psi_{1,t})$ is denoted $\mathrm{PSS}(\psi_{s,t};f)$. The most important example of a PSS element is the unit element which is defined using $X=W$ and $f=\id$, in which case we denote it $\mathrm{PSS}(\psi_{s,t};W)$, as in \S\ref{sec:defin-unit-elem}. Another important example in this paper is the PSS element associated to a Lagrangian $L$. We simply take $X=L$ and $f$ equal to the inclusion. We denote this element by $\mathrm{PSS}(\psi_{s,t};L)$.

\subsection{PSS and the negative cone}
\label{sec:pss-negative-cone}

Let $\psi_{s,t}$ be continuation data with $\psi_{0,t}\in \mathscr{C}^{\times}$ and $\psi_{1,t}=\id$. In other words, $\psi_{s,t}$ represents a morphism from $\psi_{0,t}$ to $\id$. The moduli spaces considered in this section is a reflected version of moduli space \S\ref{sec:pss-positive-cone}. Let $\kappa_{s,t}=\psi_{s,t}\delta_{s,t}$. We consider:
\begin{equation*}
  \xi_{s,t}=\kappa_{\beta(1-s),\beta(3t-1)}.
\end{equation*}
The perturbation term $\delta_{s,t}$ is as in \S\ref{sec:pss-positive-cone}.

Define the moduli space $\mathscr{M}_{-}(\kappa_{s,t})$ of solutions to:
\begin{equation}\label{eq:PSS-negative}
  \left\{
    \begin{aligned}
      &v:\C\to W\text{ smooth},\\
      &u(s,t)=v(e^{2\pi(s+it)}),\\
      &\bd_{s}u-\rho(t)Y_{s,t}(u)+J(u)(\bd_{t}u-X_{s,t}(u))=0,\\
    \end{aligned}
  \right.
\end{equation}
very similarly to \eqref{eq:PSS-positive}. The minus sign signifies that the removable singularity is at the $s=-\infty$ end. The important thing to notice is that:
\begin{equation*}
  \left\{\begin{aligned}
    &X_{s,t}=Y_{s,t}=0\text{ for $s\le 0$},\\
    &X_{s,t}=X_{t}\text{ and }Y_{s,t}=0\text{ for $s\ge 1$},\\
  \end{aligned}\right.
\end{equation*}
and hence any solution $v$ to \eqref{eq:PSS-negative} is holomorphic on a small disk around $0$.

\begin{lemma}\label{lemma:PSS-genericity-2}
  For generic perturbation of $\delta_{s,t}$, $\mathscr{M}_{-}(\kappa_{s,t})$ is cut transversally and is a smooth manifold. The local dimension of $\mathscr{M}_{-}$ at a solution $v$, asymptotic to an orbit $\gamma$ at the positive end, is given by:
  \begin{equation*}
    \dim_{\mathscr{M}_{-}}(v)=n+\mathrm{CZ}_{\mathfrak{s}}(\gamma)+2\mathfrak{s}^{-1}(0)\cdot v.
  \end{equation*}
  The evaluation-at-zero map and total evaluation map can be assumed to be transverse to any countable collection of smooth maps $f:X\to W$.
\end{lemma}
\begin{proof}
  The argument is the same as Lemma \ref{lemma:PSS-genericity}.
\end{proof}
\begin{lemma}\label{lemma:PSS-compactness-2}
  The set of points of the form $v(0)$ where $v\in \mathscr{M}_{-}(\kappa_{s,t})$ and which satisfy $\omega(v)\le A$ is contained in a compact set depending on $\psi_{s,t}$ and $A$.
\end{lemma}
\begin{proof}
  The argument is exactly the same as Lemma \ref{lemma:PSS-compactness-1}.
\end{proof}

\subsection{Fiber products of PSS moduli spaces}
\label{sec:fiber-products-pss}

Let $q_{\tau}$, $\tau\in [0,\infty)$, be an arbitrary homotopy of smooth maps $W\to W$, with $q_{0}=\id$. Let $\psi_{0,s,t},\psi_{1,s,t}$ be PSS continuation data to $\id$ and from $\id$, respectively, with $\psi_{0,0,t},\psi_{1,1,t}\in \mathscr{C}^{\times}$. Let $\kappa_{0,s,t}=\psi_{0,s,t}\delta_{0,s,t}$ and $\kappa_{1,s,t}=\psi_{1,s,t}\delta_{1,s,t}$ be perturbed data, as in the previous sections. Define:
\begin{equation*}
  \mathscr{M}_{\mathrm{fp}}=\mathscr{M}_{\mathrm{fp}}(q,\kappa_{0,s,t},\kappa_{1,s,t})
\end{equation*}
to be the fiber product of triples $(u_{0},u_{1},\tau)$ such that:
\begin{enumerate}
\item $u_{0}\in \mathscr{M}_{-}(\kappa_{0,s,t})$ as in \S\ref{sec:pss-negative-cone},
\item $u_{1}\in \mathscr{M}_{+}(\kappa_{1,s,t})$ as in \S\ref{sec:pss-positive-cone},
\item $q_{\tau}(u_{0}(-\infty))=u_{1}(\infty)$.
\end{enumerate}
See Figure \ref{fig:fp}. The projection $\mathscr{M}_{\mathrm{fp}}\to [0,\infty)$ sending $(u_{0},u_{1},\tau)\mapsto \tau$ will play an important role in subsequence arguments.

\subsubsection{Continuation via the fiber product moduli space}
\label{sec:cont-via-fiber}

For generic $\delta_{0,s,t},\delta_{1,s,t}$, the fiber product moduli space $\mathscr{M}_{\mathrm{fp}}$ is cut transversally and has local dimension at a solution $(u_{0},u_{1},\tau)$ equal to:
\begin{equation*}
  \dim_{\mathscr{M}_{-}}(u_{0})+\dim_{\mathscr{M}_{+}}(u_{1})+1-2n=\mathrm{CZ}_{\mathfrak{s}}(\gamma_{+})-\mathrm{CZ}_{\mathfrak{s}}(\gamma_{-})+2\mathfrak{s}^{-1}(0)\cdot [u]+1,
\end{equation*}
where $[u]=[u_{0}]+[u_{1}]$. Let $\tau$ be a regular value of $\mathscr{M}_{\mathrm{fp}}\to [0,\infty)$ and consider the zero dimensional component of the fiber, denoted:
\begin{equation*}
  \mathscr{M}_{\mathrm{fp},0}^{\tau}(x_{-},a,x_{+};q,\kappa_{1,s,t};\kappa_{0,s,t}),
\end{equation*}
consisting of solutions $(u_{0},u_{1},\tau)$ satisfying:
\begin{enumerate}
\item $\gamma_{\pm}(0)=x_{\pm}$,
\item $\omega\cdot [u]=a$,
\item $\mathrm{CZ}_{\mathfrak{s}}(\gamma_{+})-\mathrm{CZ}_{\mathfrak{s}}(\gamma_{-})+2\mathfrak{s}^{-1}(0)\cdot [u]=0$
\end{enumerate}
This defines a morphism $\mathfrak{c}_{\mathrm{fp},\tau}:\mathrm{CF}(\psi_{0,t})\to \mathrm{CF}(\psi_{1,t})$ by the formula:
\begin{equation*}
  \mathfrak{c}_{\mathrm{fp},\tau}(\tau^{b}x_{+})=\sum_{a,x_{-}}\#\mathscr{M}_{\mathrm{fp},0}^{\tau}(x_{-},a,x_{+};q,\kappa_{1,s,t};\kappa_{0,s,t}) \tau^{a+b}x_{-}.
\end{equation*}
\emph{Warning.} One should not confused the Novikov variable $\tau$ appearing in $\tau^{b}$ with the time parameter $\tau$ appearing in $q_{\tau}$.

The key result about this map is the following:
\begin{lemma}
  For generic $\delta_{0,s,t},\delta_{1,s,t}$ there is a generic set of regular values $\tau$ containing $\tau=0$ so that $\mathfrak{c}_{\mathrm{fp},\tau}$ is a chain map. Moreover, the chain homotopy class of $\mathfrak{c}_{\mathrm{fp},\tau}$ is independent of $\tau$.
\end{lemma}
\begin{proof}
  This is fairly standard Floer theory, and we only outline the argument. The first step is to pick the data generically so that the total evaluation is transverse to the simple holomorphic spheres. This ensures properness of the $\tau$ coordinate on the one-dimensional components of $\mathscr{M}_{\mathrm{fp}}$, up to breaking of Floer cylinders. The fibers of this one-dimensional manifold over generic values defines the morphism. Standard Floer theory explains that the counts obtained from different fibers $\tau=\tau_{0}$ and $\tau=\tau_{1}$ define chain homotopic maps; the argument is similar to the proof of Lemma \ref{lemma:deformation} that the chain homotopy classes of the continuation maps are invariant under homotopies of continuation data.
\end{proof}
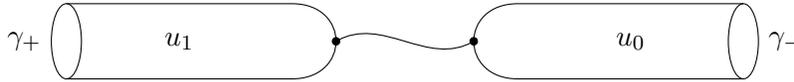
\begin{figure}[h]
  \centering
  \begin{tikzpicture}
    \begin{scope}[shift={(-6,0)}]
      \draw (0,0.5) circle (0.2 and 0.5);
      \node at (-0.2,0.5) [left] {$\gamma_{+}$};
      \draw (0,1)--(3,1)to[out=0,in=0,looseness=2]coordinate[pos=0.5](X)(3,0)--(0,0);
      \node at (1.5,0.5) {$u_{1}$};
      \node at (X)[draw,circle,fill,inner sep=1pt]{};
    \end{scope}
    \draw (3,0.5) circle (0.2 and 0.5);
    \node at (3+0.2,0.5) [right] {$\gamma_{-}$};
    \node at (1.5,0.5) {$u_{0}$};
    \draw (3,1)--(0,1)to[out=180,in=180,looseness=2]coordinate[pos=0.5](Y)(0,0)--(3,0);
    \node[draw,circle,fill,inner sep=1pt] at (Y){};

    \draw(X)to[out=30,in=210](Y);
  \end{tikzpicture}
  \caption{An element $(u_{0},u_{1},\tau)$ in the fiber product moduli space $\mathscr{M}_{\mathrm{fp}}$. The connecting trajectory is a flow line of the smooth isotopy $q$.}
  \label{fig:fp}
\end{figure}

\subsubsection{Smooth displaceability and the fiber product continuation map}
\label{sec:smooth-displ-fiber}

In \S\ref{sec:cont-maps-are} we show that $\mathfrak{c}_{\mathrm{fp},0}$ equals the usual continuation map. First we will show:
\begin{lemma}
  If every compact set in $W$ is smoothly displaceable, then the chain homotopy class of $\mathfrak{c}_{\mathrm{fp},\tau}$ is trivial.
\end{lemma}

\begin{proof}
  Lemmas \ref{lemma:PSS-compactness-1} and \ref{lemma:PSS-compactness-2} imply that:
  \begin{equation*}
    \begin{aligned}
      S_{0,a}&=\set{u_{0}(-\infty):(u_{0},u_{1},\tau)\in \mathscr{M}_{\mathrm{fp}}^{\tau}(x_{-},a,x_{+};q,\kappa_{1,s,t},\kappa_{0,s,t})},\\
      S_{1,a}&=\set{u_{1}(+\infty):(u_{0},u_{1},\tau)\in \mathscr{M}_{\mathrm{fp}}^{\tau}(x_{-},a,x_{+};q,\kappa_{1,s,t},\kappa_{0,s,t})},\\
    \end{aligned}
  \end{equation*}
  are both contained in a compact set. Moreover, this compact set is independent of $\tau$ or $q$, and depends only on $\kappa_{i,s,t}$ and an upper bound for $a$.

  If every compact set in $W$ is smoothly displaceable, then, for any $A$, we can pick $q$ so that $q_{\tau}(S_{0,a})\cap S_{1,a}=\emptyset$ holds for sufficiently large $\tau$ and for $a\le A$.

  Now suppose that $c_{\mathrm{fp},\tau}$ acts non-trivially on homology. Then there is an element $\zeta\in \mathrm{HF}(\psi_{0,0,t})$ such that $c_{\mathrm{fp},\tau}(\zeta)$ is non-zero in $\mathrm{HF}(\psi_{1,1,t})$.

  Define a non-archimedean filtration on $\mathrm{CF}(\psi_{1,1,t})$ by the formula:
  \begin{equation*}
    \ell(\sum \tau^{b_{i}}x_{i})=-\inf_{i}b_{i},\text{ with $\ell(0)=-\infty$.}
  \end{equation*}
  By the work of \cite[Theorem 1.3]{usher-compositio-2008} and \cite{usher-zhang-GT-2016} on the nonarchimedean Gram-Schmidt process, the following quantity is finite:
  \begin{equation}\label{eq:inf}
    \inf\set{\ell(z):z\in \mathrm{CF}(\psi_{1,1,t})\text{ and }[z]=c_{\mathrm{fp},\tau}(\zeta)\in \mathrm{HF}(\psi_{1,1,t})},
  \end{equation}
  since $c_{\mathrm{fp},\tau}(\zeta)$ is non-zero in the homology group. This quantity can be interpreted as the (log of the) non-archimedean distance from any representative cycle $z$ to the subspace of exact cycles.
  
  On the other hand, for any desired $A$, the above analysis shows that $c_{\mathrm{fp},\tau}$ can be represented by a chain map which sends $\tau^{b}x_{-}$ to $\tau^{b+a}x_{+}$ where $a>A$. It then follows easily that the infimum in \eqref{eq:inf} is $-\infty$, a contradiction. This proves that $c_{\mathrm{fp},\tau}$ must act trivially on homology, as desired.
\end{proof}

\subsection{Continuation maps are fiber product continuation maps}
\label{sec:cont-maps-are}

In this section we prove that the chain homotopy class of $\mathfrak{c}_{\mathrm{fp},\tau}:\mathrm{CF}(\psi_{0,0,t})\to \mathrm{CF}(\psi_{1,1,t})$ equals the chain homotopy class of the continuation map of \S\ref{sec:continuation-maps} associated to the continuation data obtained by concatenating $\psi_{0,s,t}$ and $\psi_{1,s,t}$. The idea is to deform the continuation cylinder by requiring $u$ to be holomorphic on a cylinder with large modulus in between the continuation for $\psi_{0,s,t}$ and $\psi_{1,s,t}$, as illustrated in Figure \ref{fig:deforming-cont-map-fp}.

\subsubsection{The parametric moduli space}
\label{sec:param-moduli-space-1}

As above, let $\xi_{i,s,t}=\kappa_{i,\beta(-s),\beta(3t-1)}$, and let $X_{i,s,t}$, $Y_{i,s,t}$ be the generating vector fields. Consider the parametric moduli space $\mathscr{M}(\kappa_{1,s,t},\kappa_{0,s,t})$ of pairs $(\sigma,u)$ to:
\begin{equation*}
  \left\{
    \begin{aligned}
      &\sigma\ge 0,\\
      &u:\R\times \R/\Z\to W,\\
      &\bd_{s}u-\rho(t)(Y_{1,s,t}+Y_{0,s-1-\sigma,t})+J(u)(\bd_{t}u-X_{1,s,t}-X_{0,s-1-\sigma,t})=0,\\
    \end{aligned}
  \right.
\end{equation*}
where $\rho(t)=\beta(3-3t)$. It is important to note that:
\begin{enumerate}
\item $u$ solves \eqref{eq:cont-map} for $\kappa_{1,s,t}$ on $(-\infty,1+\sigma]\times \R/\Z$,
\item $u(s+\sigma,t)$ solves \eqref{eq:cont-map} for $\kappa_{0,s,t}$ on $[1-\sigma,\infty)\times \R/\Z$,
\item $u$ is $J$-holomorphic on the cylinder $s\in [1,1+\sigma]$.
\end{enumerate}
As shown in Figure \ref{fig:compression}, we extend the moduli space to $\sigma\in [-1,\infty)$ via the following equation; for $\sigma\in [-1,0]$, consider the pairs $(\sigma,u)$ to:
\begin{equation*}
  \left\{
    \begin{aligned}
      &u:\R\times \R/\Z\to W,\\
      &\bd_{s}u-\rho(t)(Y_{1,f(\sigma)s,t}+Y_{0,f(\sigma)s-1,t})\\
      &\hspace{2cm}+J(u)(\bd_{t}u-X_{1,f(\sigma)s,t}-X_{0,f(\sigma)s-1,t})=0,\\
    \end{aligned}
  \right.
\end{equation*}
where $f:[-1,0]\to [1,2]$ is such that $f(x)=1$ holds for $x$ near $0$, and $f(x)=2$ for $x$ near $-1$, e.g., $f(x)=2-\beta(x+1)$. The purpose is so when $\sigma=-1$ the equation is exactly the continuation cylinder for the concatenated data.

We let $\mathscr{M}^{\sigma}(\kappa_{1,s,t},\kappa_{0,s,t})\subset \mathscr{M}(\kappa_{1,s,t},\kappa_{0,s,t})$ denote the fiber over $\sigma$. Furthermore, we let:
\begin{equation*}
  \mathscr{M}(a;\kappa_{1,s,t},\kappa_{0,s,t})\subset \mathscr{M}(\kappa_{1,s,t},\kappa_{0,s,t})
\end{equation*}
be the component consisting of those solutions $u$ with $\omega(u)=a$.

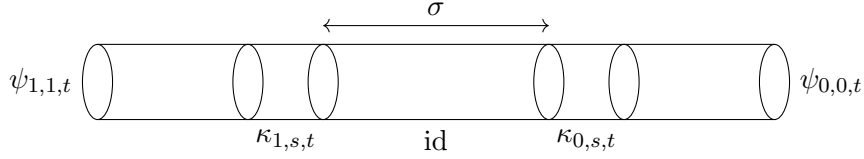
\begin{figure}[h]
  \centering
    \begin{tikzpicture}
    \draw (0,0) coordinate(X) --+ (9,0) +(0,1) --+ (9,1);
    \draw (X)+(0,0.5) circle(0.2 and 0.5) +(-0.2,0.5)node[left]{$\psi_{1,1,t}$} +(2,0.5) circle (0.2 and 0.5) +(3,0.5) circle (0.2 and 0.5)
    +(6,0.5) circle (0.2 and 0.5) +(7,0.5) circle (0.2 and 0.5) +(9,0.5) circle (0.2 and 0.5) +(9.2,0.5) node[right]{$\psi_{0,0,t}$};
    \draw (X)+(2.5,0)node[below]{$\kappa_{1,s,t}$} +(6.5,0)node[below]{$\kappa_{0,s,t}$} +(4.5,0)node[below]{$\id\vphantom{\psi_{0,s,t}}$};
    \draw[<->] (3,1.25)--node[above]{$\sigma$}(6,1.25);
  \end{tikzpicture}
  \caption{Deforming the continuation map; in the limit, solutions break into a configuration of two PSS solutions in the fiber product moduli space.}
  \label{fig:deforming-cont-map-fp}
\end{figure}

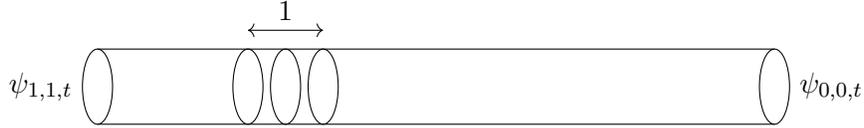
\begin{figure}[h]
  \centering
    \begin{tikzpicture}
    \draw (0,0) coordinate(X) --+ (9,0) +(0,1) --+ (9,1);
    \draw (X)+(0,0.5) circle(0.2 and 0.5) +(-0.2,0.5)node[left]{$\psi_{1,1,t}$} +(2,0.5) circle (0.2 and 0.5) +(2.5,0.5) circle (0.2 and 0.5)
    +(3,0.5) circle (0.2 and 0.5) +(9,0.5) circle (0.2 and 0.5) +(9.2,0.5) node[right]{$\psi_{0,0,t}$};    
    \draw[<->] (2,1.25)--node[above]{$1$}(3,1.25);
  \end{tikzpicture}
  \caption{Compressing the continuation data.}
  \label{fig:compression}
\end{figure}

\begin{lemma}\label{lemma:fancy-compactness}
  For generic $\kappa_{0,s,t},\kappa_{1,s,t}$,
  \begin{enumerate}[label=(\alph*)]
  \item\label{kka} The moduli space $\mathscr{M}(\kappa_{1,s,t},\kappa_{0,s,t})$ is cut transversally, its boundary is $\mathscr{M}^{-1}(\kappa_{1,s,t},\kappa_{0,s,t})$, which is also cut transversally, and its dimension at a solution $(\sigma,u)$ with asymptotics $\gamma_{\pm}$ is given by the formula:
    \begin{equation*}
      d=\mathrm{CZ}_{\mathfrak{s}}(\gamma_{+})-\mathrm{CZ}_{\mathfrak{s}}(\gamma_{-})+2\mathfrak{s}^{-1}(0)+1;
    \end{equation*}
    the associated collection of components is denoted $\mathscr{M}_{d}$.
  \item\label{kkb} The connected components of $\mathscr{M}_{1}$ are $1$-manifolds whose projection to the parameter space $[-1,\infty)$ is proper up to non-compact ends limiting to a configuration of a a Floer cylinders (at either end) joined to a solution in the component in $\mathscr{M}_{0}$.
  \item\label{kkc} For given $A\in \R$, the union of $\mathscr{M}_{0}(a)$ over all $a\le A$ is finite.
  \item\label{kkd} Any sequence $(\sigma_{n},u_{n})$ in the union of $\mathscr{M}_{1}(a)$, over all $a\le A$, has a subsequence which converges in the Floer theory sense to a configuration $(u_{+},u_{-})$ in $\mathscr{M}_{\mathrm{fp}}$ where $u_{+}\in \mathscr{M}_{+}(\kappa_{1,s,t})$ and $u_{-}\in \mathscr{M}_{-}(\kappa_{0,s,t})$ with $u_{+}(+\infty)=u_{-}(-\infty)$.
  \end{enumerate}
\end{lemma}
\begin{proof}
  Part \ref{kka} is similar to the other transversality statements encountered so far, e.g., Lemma \ref{lemma:transversality-continuation}, and the dimension formula follows from the general index formula for Cauchy-Riemann operators with non-degenerate asymptotics, as in, e.g., \cite{cant-thesis-2022}. One should pick $\kappa_{i,s,t}$ so that the $\sigma=-1$ fiber $\mathscr{M}^{-1}$ is cut transversally, and also so that the parametric moduli space is cut transversally; then the statement about $\bd \mathscr{M}=\mathscr{M}^{-1}$ follows from an implicit function theorem type argument.

  The crucial step in establishing \ref{kkb} is to exclude the bubbling of holomorphic spheres. For this, we observe that, if sphere bubbling happens, then we conclude the existence of a configuration of $u\in \mathscr{M}_{d}(\kappa_{1,s,t},\kappa_{0,s,t})$ attached to a simple holomorphic sphere $w\in \mathscr{M}^{*}(J)$. Because we asssume semipositivity and choose $J$ generically, we know that $\mathfrak{s}^{-1}(0)\cdot [w]\ge 0$, i.e., holomorphic spheres which bubble always decrease the dimension.

  Thus, let us also pick $\kappa_{i,s,t}$ generically so that the total evaluation: $$\mathscr{M}\times \R\times \R/\Z\to W$$ is transverse to the evaluation of the simple holomorphic spheres, which can be achieved by the same argument implying Lemma \ref{lemma:transversality-floer}. The total evaluation is defined on a manifold with $3-2\mathfrak{s}^{-1}(0)\cdot [w]$ dimensions or less, while $w$ lives in a manifold of dimension $2n-4+2\mathfrak{s}^{-1}(0)\cdot [w]$. The sum of these dimensions is $2n-1$, and hence generically the two evaluations do not intersect; thus bubbling does not happen for generic $\kappa_{i,s,t}$. Standard Floer theory implies the only other failure of compactness is the bubbling of a Floer cylinder at either end, and \ref{kkb} follows.

  For \ref{kkc}, take a sequence $(\sigma_{n},u_{n})\in \mathscr{M}_{0}$ with $\omega(u_{n})\le A$. First suppose that $\sigma_{n}$ remains bounded, and so converges after taking a subsequence. Then the PDE which $u_{n}$ solves converges, and so either $u_{n}$ converges to a solution for the limit PDE, after passing to a subsequence, or a Floer cylinder breaks, or a holomorphic sphere bubbles. The breaking of a Floer cylinder can be obstructed since $\mathscr{M}_{-1}=\emptyset$, because $\mathscr{M}$ is cut transversally. The bubbling of a holomorphic sphere can excluded by the same argument used in the proof of \ref{kkb}. Thus we conclude in this case that $(\sigma_{n},u_{n})$ has a convergent subsequence.

  Second, let us consider the case when $\sigma_{n}$ is unbounded, i.e., $\sigma_{n}\to\infty$ after passing to a subsequence. In this case, it will convenient to analyze the general case $(\sigma_{n},u_{n})\in \mathscr{M}_{d}$ rather than only $d=0$, as we will simultaneously be able to establish \ref{kkd}.

  The usual Floer compactness statements, similar to those in \cite{cant-chen-arXiv-2023}, imply $u_{n}$ breaks, giving rise to configuration:
  \begin{equation*}
    v_{1},\dots,v_{p},u_{+},w_{1},\dots,w_{q},u_{-},v_{1}',\dots,v_{r}',
  \end{equation*}
  where:
  \begin{enumerate}
  \item $v_{i}$ are non-stationary Floer cylinders,
  \item $w_{i}:\R\times \R/\Z\to W$ are simple holomorphic spheres,
  \item $u_{+}\in \mathscr{M}_{+,d_{1}}(\kappa_{1,s,t})$,
  \item $u_{-}\in \mathscr{M}_{-,d_{0}}(\kappa_{0,s,t})$,
  \end{enumerate}
  and so that the removable singularities satisfy the incidence:
  \begin{equation*}
    \text{$u_{+}(+\infty)=w_{1}(-\infty)$, $w_{i}(+\infty)=w_{i+1}(-\infty)$, $w_{q}(+\infty)=u_{-}(-\infty)$,}
  \end{equation*}
  and the asymptotic orbits of the non-degenerate systems satisfy a similar incidence. Moreover, we may also suppose that the total holomorphic map:
  \begin{equation*}
    w_{1}\sqcup w_{2}\sqcup \dots \sqcup w_{q}:\mathbb{C}P^{1}\sqcup \dots \sqcup \mathbb{C}P^{1}\to W
  \end{equation*}
  is simple. Otherwise two of the simple spheres are reparametrizations of each other, by the same argument given in \cite[Chapter 2]{mcduffsalamon}, and we can take a shortcut in the list, maintaining the above incidence conditions.

  Of course, there are may also be additional bubbling in the limit, and some of the limit spheres which form are multiply covered, but we always conclude an underlying configuration of the above form, by passing to underlying simple curves and ignoring bubbles. Since all simple holomorphic spheres have non-negative Chern number, and the Floer cylinders have positive index, we conclude that:
  \begin{equation}\label{eq:dim-ineq}
    d_{1}+d_{0}\le 2n+d-1-2\mathfrak{s}^{-1}(0)\cdot ([w_{1}]+\dots+[w_{q}]).
  \end{equation}
  with equality holding only if there are no Floer cylinders $v_{i},v_{j}'$.

  Consider the total evaluation map:
  \begin{equation*}
    \mathscr{M}_{+,d_{1}}(\kappa_{1,s,t})\times \mathscr{M}^{*}(\mathbb{C}P^{1}\sqcup \dots \sqcup \mathbb{C}P^{1},J)\times \mathscr{M}_{-,d_{0}}(\kappa_{0,s,t})\to (W\times W)^{q+1},
  \end{equation*}
  where we send $(u_{+},w_{1},\dots,w_{q},u_{-})$ to the evaluations at $+\infty,-\infty$, in the written order. Here $\mathscr{M}^{*}(\mathbb{C}P^{1}\sqcup \dots \sqcup \mathbb{C}P^{1},J)$ denotes the moduli space of simple maps $\mathbb{C}P^{1}\sqcup \dots \sqcup \mathbb{C}P^{1}\to W$.
  
  The existence of the above configuration implies that the diagonal $\Delta^{d+1}$ has a non-empty preimage. Moreover, for generic $J$ and generic $\kappa_{0,s,t},\kappa_{1,s,t}$, we can ensure that the above total evaluation is transverse to $\Delta^{q+1}$.

  The total dimension is of inverse image of $\Delta^{q+1}$ is at most:
  \begin{equation*}
    2n+d-1+2nq-2n(q+1)=d-1,
  \end{equation*}
  where we have used \eqref{eq:dim-ineq}. Thus the inverse image when $d=0$ is empty, concluding \ref{kkc}. Moreover, no Floer cylinders could break off in the $d=1$ case.

  In the case $d=1$, with no Floer cylinders breaking, we can still say something: the automorphism group of the Riemann surface $\mathbb{C}P^{1}\sqcup \dots \sqcup \mathbb{C}P^{1}$, fixing the points $+\infty,-\infty$ (one pair for each sphere) is $2q$ dimensional and acts freely on the inverse image of $\Delta^{q+1}$, and hence if $q\ge 1$ then the inverse image is empty when $d=1$. See Figure \ref{fig:fp-obstruction}.

  Finally, in the case $d=1$, we conclude that $u_{n}$ converges to $(u_{+},u_{-})$ up to the formation of bubbles. However, the bubbles can again be excluded by the same argument as in \ref{kkd}. This completes the proof.
\end{proof}

\begin{figure}[h]
  \centering
  \begin{tikzpicture}
    \begin{scope}[shift={(-5,0)}]
      \draw (0,0.5) circle (0.2 and 0.5);
      \node at (-0.2,0.5) [left] {$\gamma_{+}$};
      \draw (0,1)--(3,1)to[out=0,in=0,looseness=2]coordinate[pos=0.5](X)(3,0)--(0,0);
      \node at (1.5,0.5) {$u_{1}$};
      \node at (X)[draw,circle,fill,inner sep=1pt]{};
    \end{scope}
    \draw (3,0.5) circle (0.2 and 0.5);
    \node at (3+0.2,0.5) [right] {$\gamma_{-}$};
    \node at (1.5,0.5) {$u_{0}$};
    \draw (3,1)--(0,1)to[out=180,in=180,looseness=2]coordinate[pos=0.5](Y)(0,0)--(3,0);
    \node[draw,circle,fill,inner sep=1pt] at (Y){};

    \draw (X) arc (180:-180:.42 and 0.1) arc (180:-180:.42);
    
  \end{tikzpicture}
  \caption{A hypothetical configuration in the limit $\sigma\to\infty$ which can be excluded using general position arguments and the rotation action on the central holomorphic sphere by reparametrizations.}
  \label{fig:fp-obstruction}
\end{figure}
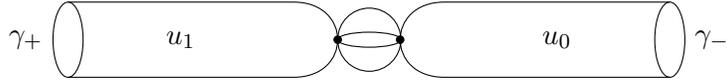

\subsubsection{Chain homotopy}
\label{sec:chain-homotopy}

Let $\kappa_{0,s,t},\kappa_{1,s,t}$ be generic. In particular, suppose they are generic enough that:
\begin{equation*}
  \mathfrak{c}_{\mathrm{fp},0}:\mathrm{CF}(\psi_{0,0,t})\to \mathrm{CF}(\psi_{1,1,t})
\end{equation*}
is well-defined and represents the chain homotopy class of the fiber product continuation map (here the parameter is $\tau=0$, so $\mathfrak{c}_{\mathrm{fp},0}$ counts PSS cylinders which are incident at their punctures). Then:
\begin{lemma}
  The chain homotopy class of $\mathrm{c}_{\mathrm{fp},0}$ equals the chain homotopy class of the continuation map $\mathrm{CF}(\psi_{0,0,t})\to \mathrm{CF}(\psi_{1,1,t})$ using the concatenation of the continuation data $\psi_{0,s,t}\#\psi_{1,s,t}$.
\end{lemma}
\begin{proof}
  We will be brief, since the argument is just standard Floer theory ideas. The one-dimensional components of the parametric moduli space $\mathscr{M}(\kappa_{1,s,t},\kappa_{0,s,t})$ has a boundary whose Floer theoretic count represents the regular continuation map, by construction. On the other hand, the analysis in Lemma \ref{lemma:fancy-compactness} shows that the one-dimensional components converge as $\sigma\to\infty$ to exactly the configurations whose Floer theoretic count gives $\mathfrak{c}_{\mathrm{fp},0}$. The rigid elements in $\mathscr{M}_{0}(a;\kappa_{1,s,t},\kappa_{0,s,t})$, have a well-defined Floer theoretic count which gives the chain homotopy term. This completes the proof.
\end{proof}

Combining this result with the result in \S\ref{sec:smooth-displ-fiber} completes the proof of Theorem \ref{theorem:smooth-displaceability}.\hfill$\square$

\section{Non-zero eternal classes and Lagrangians}
\label{sec:non-zero-eternal}

In this section we prove Theorem \ref{theorem:odd-euler-char-1} which states the PSS construction applied to a compact monotone Lagrangian $L$ with minimal Maslov number equal to 2 defines a non-zero class in $\mathrm{SH}_{e}$, provided there is another Lagrangian $L'$ of the same type (compact, monotone, minimal Maslov number at least two) so that $L'\cap L$ is transverse and $\#(L'\cap L)$ is odd.

\subsection{The Lagrangian element}
\label{sec:lagr-element}

We begin by constructing an element in the Floer cohomology $\mathrm{HF}(\psi_{t})$ using the Lagrangian $L$. The construction is inspired by the open-closed map in the context of Lagrangian quantum topology of \cite{biran-cornea-GT-2009}.

\subsubsection{Moduli space of half-infinite Floer cylinders}
\label{sec:moduli-space-half-infinite}

Given a generic contact-at-infinity system $\psi_{t}$, consider the moduli space $\mathscr{M}(\psi_{t},L)$ of solutions to:
\begin{equation*}
\left\{
  \begin{aligned}
    &u:(-\infty,0]\times \R/\Z\to W,\\
    &\bd_{s}u+J(u)(\bd_{t}u-X_{t}(u))=0,\\
    &u(s,0)\in L.
  \end{aligned}
\right.
\end{equation*}
As in \S\ref{sec:floer-differential-moduli}, $X_{t}$ is the generator of $\psi_{\beta(3t-1)}$. This moduli space decomposes into components of varying dimensions.

\begin{figure}[h]
  \centering
  \begin{tikzpicture}
    \draw (0,0) coordinate(X) --+ (5,0) +(0,1) --+ (5,1);
    \draw (X)+(0,0.5) circle(0.2 and 0.5) +(-0.2,0.5) node[left]{$\gamma$} +(5,0.5) circle (0.2 and 0.5) +(5.2,0.5) node[right]{$L$};
  \end{tikzpicture}
  \caption{The moduli space $\mathscr{M}(\psi_{t},L)$ of half-infinite Floer cylinders.}
\end{figure}
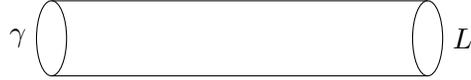

\subsubsection{Dimension formula}
\label{sec:dimension-formula}

Pick a generic section $\mathfrak{m}$ of $\det_{\C}(TW)^{\otimes2}$ so that $\mathfrak{m}|_{L}$ directs the canonical ray $\det_{\R}(TL)^{\otimes 2}$. In particular, $\mathfrak{m}$ is non-vanishing along $L$. We also require that along each orbit $\gamma$ of $X_{t}$ the section $\mathfrak{m}|_{\gamma}$ is non-vanishing.

Define the Conley-Zehnder index of $\gamma$ relative $\mathfrak{m}$ by the formula:\footnote{The \emph{winding number} between two non-vanishing $\mathfrak{y}_{0},\mathfrak{y}_{1}$ sections of a complex line bundle $E\to S^{1}$ is defined to be the signed count of zeros of a generic section of the pullback bundle $E\to \R\times S^{1}$ which agrees with $\mathfrak{y}_{0}$ for $s<0$ and with $\mathfrak{y}_{1}$ for $s>1$.}
\begin{equation*}
  \mathrm{CZ}_{\mathfrak{m}}(\gamma)=\mathrm{CZ}_{\mathfrak{s}}(\gamma)+\mathrm{wind}(\mathfrak{m}|_{\gamma},\mathfrak{s}\otimes\mathfrak{s}),
\end{equation*}
where $\mathfrak{s}$ is any non-vanishing section of $\det_{\C}(TW)|_{\gamma}$.

Introduce the Maslov class $\mu_{\mathfrak{m}}(L)$ as the (transverse) zero set $\mathfrak{m}^{-1}(0)$. Then the Fredholm index of the linearization of \S\ref{sec:moduli-space-half-infinite} at a solution $u$ whose negative asymptotic is $\gamma$ is equal to:
\begin{equation}\label{eq:dim-half-infinite}
  \mu_{\mathfrak{m}}(L)\cdot [u]-\mathrm{CZ}_{\mathfrak{m}}(\gamma).
\end{equation}
This follows from a mild variation of the formulas proved in \cite{cant-thesis-2022}.

\emph{Remark}. It is important to note that the local dimension near a Floer differential cylinder $u$ with asymptotics $\gamma_{-},\gamma_{+}$ (as in \S\ref{sec:floer-differential-moduli}) is equal to:
\begin{equation*}
  \text{index of Floer differential}=\mathrm{CZ}_{\mathfrak{m}}(\gamma_{+})-\mathrm{CZ}_{\mathfrak{m}}(\gamma_{-})+\mu_{\mathfrak{m}}(L)\cdot [u];
\end{equation*}
i.e., the usual dimension formula continues to hold with the generalized Conley-Zehnder indices provided $2c_{1}^{\mathfrak{s}}$ is replaced by $\mu_{\mathfrak{m}}(L)$. One proves this first when $\mathfrak{m}=\mathfrak{s}\otimes \mathfrak{s}$ and then shows that the quantity is invariant under changes in the section $\mathfrak{m}$.

\subsubsection{Transversality and generic compactness}
\label{sec:generic-compactness}

Recall that Theorem \ref{theorem:odd-euler-char-1} assumes that $L$ is monotone with minimal Maslov number at least two. As we show in this section, for generic $\psi_{t}$, the zero-dimensional component of the moduli space from \S\ref{sec:moduli-space-half-infinite}, with symplectic area $a$, is compact (and hence a finite set) and the one-dimensional component is compact up to the breaking of Floer cylinders at the left end.

To be precise about the genericity, we state the following lemma:
\begin{lemma}
  Suppose $\psi_{1}$ has non-degenerate fixed points. For generic compactly supported systems $\delta_{t}$, every solution to $\mathscr{M}(\psi_{t}\delta_{t},L)$ is cut transversally, and in particular, the local dimension in \eqref{eq:dim-half-infinite} is non-negative.
\end{lemma}
\begin{proof}
  The proof follows the same argument as Lemma \ref{lemma:transversality-floer}.
\end{proof}

Let $\mathscr{M}(a)=\mathscr{M}(a;\psi_{t}\delta_{t},L)$ be the component of consisting of solutions whose symplectic area is $a$. Then solutions in $\mathscr{M}(a)$ satisfy an a priori energy bound in terms of $a$. This implies $\mathscr{M}(a)$ is compact up to breaking of Floer cylinders at the negative end, or bubbling of holomorphic spheres or disks. However, any holomorphic sphere or disk has $\mu$ at least $2$ (by assumption), and hence one cannot have bubbling in the zero or one-dimensional components of $\mathscr{M}(a)$; otherwise the dimension in \eqref{eq:dim-half-infinite} would be negative for some solution. Consideration of Fredholm indices shows that:
\begin{enumerate}
\item Any sequence in the zero-dimensional component $\mathscr{M}_{0}(a)$ has a convergent subsequence; no Floer differential can bubble off since the index $-1$ components $\mathscr{M}_{-1}$ are empty.
\item Any sequence in the 1-dimensional component $\mathscr{M}_{1}(a)$ has a convergent subsequence or converges (in the Floer theory sense) to an Floer differential cylinder of index $1$ connected to an element of $\mathscr{M}_{0}$.
\end{enumerate}

\subsubsection{Definition of the Lagrangian element}
\label{sec:defin-natur-elem}

Let $a\in \R$ and let $x$ be a fixed point of $\psi_{1}$. Let $\mathscr{M}_{0}(x;a)=\mathscr{M}_{0}(x;a;\psi_{t},L)\subset \mathscr{M}(\psi_{t},L)$ be the component consisting of solutions $u$ whose left asymptotic $\gamma$ satisfies $\gamma(0)=x$, satisfy $\omega(u)=a$, and for which the local dimension in \eqref{eq:dim-half-infinite} is $0$.

As in \S\ref{sec:generic-compactness}, $\mathscr{M}_{0}(x;a)$ is a finite set of points (indeed, the union over all $a\le A$ is still finite). Define:
\begin{equation*}
  \mathrm{LE}(\psi_{t},L):=\sum_{a,x}\#\mathscr{M}_{0}(x;a) \tau^{a}x\in \mathrm{CF}(\psi_{t}).
\end{equation*}
Consideration of the one-dimensional components $\mathscr{M}_{1}(x;a)$ whose solutions satisfy the above, except with local dimension equal to $1$, proves $\mathrm{LE}(\psi_{t},L)$ is a cycle with respect to the differential from \S\ref{sec:floer-differential-1}.

\subsection{Naturality of the Lagrangian element}
\label{sec:natur-lagr-elem}

The goal in this section is to prove the assignment:
\begin{equation*}
  \psi_{t}\mapsto \mathrm{LE}(\psi_{t},L)\in \mathrm{HF}(\psi_{t}).
\end{equation*}
defines a natural transformation from the constant $\Z/2$-valued functor to the restriction of $\mathrm{HF}$ to the subcategory $\mathscr{C}^{\times}\subset \mathscr{C}$ of systems satisfying the genericity conditions of \S\ref{sec:lagr-element}. To do so, we will consider in \S\ref{sec:param-moduli-space} a 1-parametric moduli space which combines continuation maps with the half-infinite cylinders used to define $\mathrm{LE}(\psi_{t};L)$. In \S\ref{sec:defin-chain-homot} we explain how to use the parametric moduli space to prove continuation maps preserve the Lagrangian elements.

\subsubsection{The 1-parametric moduli space}
\label{sec:param-moduli-space}

Let $\psi_{0,t},\psi_{1,t}$ be systems which satisfy the requisite transversality for defining their Lagrangian elements (and therefore also the Floer complexes), and pick some continuation data $\psi_{s,t}$ between them.

We recall the set-up of \S\ref{sec:cont-cylind}; we let:
\begin{equation*}
  \xi_{s,t}=\psi_{\beta(1-s),\beta(3t-1)},
\end{equation*}
and let $Y_{s,t},X_{s,t}$ be its infinitesimal generators with respect to $s,t$.

Define $\mathscr{M}(\psi_{s,t},L)$ to be the moduli space of pairs $(\sigma,u)$ such that:
\begin{equation*}
  \left\{
    \begin{aligned}
      &u:(-\infty,\sigma]\times \R/\Z\to W,\\
      &\bd_{s}u-\rho(t)Y_{s,t}+J(u)(\bd_{t}u-X_{s,t})=0,\\
      &u(\sigma,t)\in L,
    \end{aligned}
  \right.
\end{equation*}
where $\rho(t)=\beta(3-3t)$. As always, we assume that $u$ has finite energy:
\begin{equation*}
  \textstyle\int \omega(\bd_{s}u-\rho(3t-3)Y_{s,t},\bd_{t}u-X_{s,t})dsdt.
\end{equation*}
Note that if $(\sigma,u)$ is a solution then $v(s,t)=u(s+\sigma,t)$ is defined on the domain $(-\infty,0]\times \R/\Z$. It is important to observe that the equation for $u$ agrees Floer's equation for $\psi_{1,t}$ on the region $s\le 0$.

As usual, it is convenient to introduce the component:
\begin{equation*}
  \mathscr{M}_{d}(a;\psi_{s,t},L)\subset \mathscr{M}(\psi_{s,t},L)
\end{equation*}
of solutions $(\sigma,u)$ so that $\omega(u)=a$ and:
\begin{equation*}
  1+\mu_{\mathfrak{m}}(L)\cdot [u]-\mathrm{CZ}(\gamma)=d,
\end{equation*}
where $\gamma$ is the asymptotic of $u$. Then:

\begin{lemma}\label{lemma:proper-natur-LE}
  For a generic perturbation $\delta_{s,t}$ so that $\delta_{s,t}=\id$ for $s=0,1$ and $t=0$, the moduli space $\mathscr{M}_{d}(a)=\mathscr{M}_{d}(a;\psi_{s,t}\delta_{s,t},L)$ is a $d$-dimensional manifold. Moreover, the space $\cup_{a\le A}\mathscr{M}_{0}(a)$ is a finite set and the map:
  \begin{equation*}
    (\sigma,u)\in \mathscr{M}_{1}(a)\to \sigma\in \R
  \end{equation*}
  is a proper map up to the breaking of index 1 Floer cylinders at the asymptotic end.
\end{lemma}
\begin{proof}
  The result is fairly standard Floer theory. We review the salient points below. We abuse notation and relabel $\psi_{s,t}=\psi_{s,t}\delta_{s,t}$. One shows:
  \begin{enumerate}
  \item\label{item:LECM1} any sequence $(\sigma_{n},u_{n})$ in $\mathscr{M}(\psi_{s,t},L)$ with $\omega(u_{n})\le A$ satisfies an a priori energy bound in terms of $A$;
  \item the energy bound implies convergence up to bubbling of holomorphic spheres and disks, and breaking of Floer cylinders at the negative end;
  \item assuming transversality holds (so $\mathscr{M}_{d}(\psi_{s,t},L)=\emptyset$ for $d<0$) one cannot have bubbling of spheres or disks along any sequence in $\mathscr{M}_{1}$ or $\mathscr{M}_{0}$; moreover, one cannot have breaking of Floer cylinders along any sequence in $\mathscr{M}_{0}$;
  \item a sequence $(\sigma_{n},u_{n})\in \mathscr{M}_{0}$ cannot have $\sigma_{n}\to\infty$ because otherwise one would conclude a solution of negative index to either the Floer differential equation for $\psi_{1,t}$, the continuation map equation for $\psi_{s,t}$, or the Lagrangian element equation for $\psi_{0,t}$;
  \item the part of $\mathscr{M}_{0}$ over the region $\sigma<0$ is empty, because on there the parametric equation is $\sigma$-independent and agrees with the Lagrangian element equation for $\psi_{1,t}$ (which has no solutions of negative index).
  \end{enumerate}
  These points and standard arguments imply the statement. Perhaps the most novel point is the energy estimate in \eqref{item:LECM1}. We recall the computation:
  \begin{equation*}
    \begin{aligned}
      &\omega(\bd_{s}u-\rho(t)Y_{s,t}(u),\bd_{t}u-X_{s,t}(u))\\
      &\hspace{1cm}=\omega(\bd_{s}u,\bd_{t}u)+\rho(t)\bd_{t}(K_{s,t}(u))-\bd_{s}(H_{s,t}(u))+R,
    \end{aligned}
  \end{equation*}
  where $K_{s,t},H_{s,t}$ are the normalized Hamiltonian generators for $Y_{s,t},X_{s,t}$, using that the curvature term vanishes:
  \begin{equation*}
    R=\rho(t)(\bd_{s}H_{s,t}-\bd_{t}K_{s,t}+\omega(Y_{s,t},X_{s,t}))=0.
  \end{equation*}
  Integrate this over the cylinder yields:
  \begin{equation*}
    \omega(u)+\int_{0}^{1}H_{-\infty,t}(\gamma(t))-H_{\sigma,t}(u(\sigma,t))dt+\int_{0}^{1}\int_{0}^{\sigma}-\rho'(t)K_{s,t}(u)dsdt.
  \end{equation*}
  The first two terms are uniformly bounded along any sequence (there are only finitely many possibilities for $\gamma$, since we assume $\psi_{1,t}$ is non-degenerate, and $u(\sigma,t)$ lies on the compact Lagrangian). On the other hand, the last term is bounded above along sequence, because $K_{s,t}=0$ for $s\ge 1$ and continuation data satisfy $K_{s,t}\le 0$ outside a compact set in $W$. This completes the proof of the energy bound. We leave the details of the other points to the reader.
\end{proof}

\subsubsection{Definition of the chain homotopy}
\label{sec:defin-chain-homot}

Let $\mathfrak{c}:\mathrm{HF}(\psi_{0,t})\to \mathrm{HF}(\psi_{1,t})$ be the continuation map associated to $\psi_{s,t}$.
\begin{lemma}\label{lemma:CH-lag}
  It holds that $\mathfrak{c}(\mathrm{LE}(\psi_{0,t},L))=\mathrm{LE}(\psi_{1,t},L)$.
\end{lemma}
\begin{proof}
  Let $\mathfrak{c}$ denote the chain level representative for generic perturbation of $\psi_{s,t}$. We will prove that:
  \begin{equation*}
    \mathfrak{c}(\mathrm{LE}(\psi_{0,t},L))=\mathrm{LE}(\psi_{1,t},L)+d\mathrm{K}
  \end{equation*}
  where $\mathrm{K}\in \mathrm{CF}(\psi_{1,t})$ is defined by:
  \begin{equation*}
    \mathrm{K}=\sum_{a,x}\#\mathscr{M}_{0}(x;a;\psi_{s,t},L)\tau^{a}x.
  \end{equation*}
  In other words, the rigid elements in the 1-parametric moduli space is interpreted as a chain in $\mathrm{CF}(\psi_{1,t})$. This sum is well-defined by the compactness results in \S\ref{sec:param-moduli-space}.

  Consider the inverse image of $[0,\sigma)$ under the projection map:
  \begin{equation*}
    \sigma:\mathscr{M}_{1}(x;a;\psi_{s,t},L)\to \R.
  \end{equation*}
  This forms a manifold $P$ with boundary equal to: $$\bd P=\sigma^{-1}(0)=\mathscr{M}_{0}(x;a;\psi_{1,t},L),$$
  where here we use the moduli space from \S\ref{sec:defin-natur-elem} (note that $\psi_{1,t}$ is not considered as continuation data).

  Let us further decompose $P$ as $P_{1}=\sigma^{-1}([0,\sigma_{0}])$ and $P_{2}=\sigma^{-1}([\sigma_{0},\infty))$. If we pick $\sigma_{0}$ large enough, then \S\ref{sec:param-moduli-space} implies that $\sigma:P_{2}\to [\sigma_{0},\infty)$ is proper.

  The non-compact ends of $P_{2}$ are therefore only the Floer theoretic breakings which happen as $\sigma\to\infty$. By the usual compactness/gluing analysis, the count of such breakings is encoded as the term appearing in $\mathfrak{c}(\mathrm{LE}(\psi_{0,t},L))$ with coefficient $\tau^{a}x$ (this coefficient is simply a number $0$ or $1$).

  Since the count of non-compact ends of a proper map over $[0,\infty)$ equals the count of boundary points, we conclude:
  \begin{equation*}
    \#\bd P_{2}=\ip{\mathfrak{c}(\mathrm{LE}(\psi_{0,t},L)),\tau^{a}x}.
  \end{equation*}
  We now claim:
  \begin{equation*}
    \#\bd P=\# P_{1}+\#\bd P_{2},
  \end{equation*}
  and hence:
  \begin{equation*}
    \ip{\mathrm{LE}(\psi_{1,t},L)-\mathfrak{c}(\mathrm{LE}(\psi_{1,t},L)),\tau^{a}x}=\#P_{1}.
  \end{equation*}
  On the other hand, $\tau:P_{1}\to [0,\sigma_{0}]$ is not proper. However, it is proper (and a submersion) near $0$ and $\sigma_{0}$, and so $\#\bd P_{1}$ is equal to the count of the non-compact ends which lie over $(0,\sigma_{0})$. By \S\ref{sec:param-moduli-space}, the count of non-compact ends are due to the breaking of Floer differential cylinders. Keeping track of the indices, and applying standard gluing results, one concludes that:
  \begin{equation*}
    \#\bd P_{1}=\ip{d\mathrm{K},\tau^{a}x}.
  \end{equation*}
  Thus:
  \begin{equation*}
    \ip{\mathrm{LE}(\psi_{1,t},L)-\mathfrak{c}(\mathrm{LE}(\psi_{1,t},L))-dK,\tau^{a}x}=0.
  \end{equation*}
  Since $a,x$ were arbitrary, we conclude the desired result.
\end{proof}

\subsubsection{Extension to the full category}
\label{sec:extens-full-categ}

We have only defined a natural transformation from $\Z/2$ to $\mathrm{HF}|_{\mathscr{C}^{\times}}$ where $\mathscr{C}^{\times}$ is a full subcategory of $\mathscr{C}$ (see \S\ref{sec:funct-struct-floer} for the notation). We can formally extend to the full category using the property that the natural isomorphism:
\begin{equation*}
  \mathrm{HF}(\psi_{t})\to \lim \mathrm{HF}(\varphi_{t})
\end{equation*}
where the limit is over the category of morphisms $\psi_{t}\to \varphi_{t}$ in $\mathscr{C}$ such that $\varphi_{t}$ lies in $\mathscr{C}^{\times}$. That the natural map is an isomorphism can be proved by the same argument as in \S\ref{sec:funct-struct-floer}.

Inverting the natural map gives induced elements $\mathrm{LE}(\psi_{t})$ for all systems. Standard abstract nonsense implies the extended elements $\mathrm{LE}(\psi_{t})$ are themselves natural.

\subsection{The Lagrangian element extends the PSS class of $L$}
\label{sec:extends-PSS}

In this section, we will explain that $\mathrm{LE}(\psi_{t},L)=\mathrm{PSS}(\psi_{t},L)$ whenever $\psi_{t}$ lies in the positive cone. The moduli space used is quite similar to the one used in \S\ref{sec:natur-lagr-elem}. Instead of using standard continuation data between two non-degenerate systems, one instead uses the continuation from the identity system as in \S\ref{sec:pss-positive-cone}.

\subsubsection{The 1-parametric moduli space}
\label{sec:1-parametric-moduli}

As in \S\ref{sec:pss-positive-cone}, let:
\begin{equation*}
  \kappa_{s,t}=\psi_{st}\delta_{s,t}\text{ and }\xi_{s,t}=\kappa_{\beta(1-s),\beta(3t-1)},
\end{equation*}
where $\delta_{s,t}$ is a perturbation term compactly supported in $W$ and so $\delta_{s,t}=\id$ for $s=0,1$ and $t=0$. Let $Y_{s,t},X_{s,t}$ be the generators of $\xi_{s,t}$, and consider the moduli space $\mathscr{M}(\kappa_{s,t},L)$ of pairs $(\sigma,u)$ satisfying:
\begin{equation*}
  \left\{
    \begin{aligned}
      &u:(-\infty,\sigma]\times \R/\Z\to W,\\
      &\bd_{s}u-\rho(t)Y_{s,t}+J(u)(\bd_{t}u-X_{s,t})=0,\\
      &u(\sigma,t)\in L,
    \end{aligned}
  \right.
\end{equation*}
where (as always) $u$ has finite energy. This is the same equation as \S\ref{sec:param-moduli-space}.

Let $\mathscr{M}_{d}(x;a;\kappa_{s,t},L)$ be the component where $\omega(u)=a$, the asymptotic satisfies $\gamma(0)=x$, and
\begin{equation*}
  1+\mu_{\mathfrak{m}}(L)\cdot [u]-\mathrm{CZ}_{\mathfrak{m}}(\gamma)=d.
\end{equation*}
The same arguments given in \S\ref{sec:natur-lagr-elem} prove that $\sigma:\mathscr{M}_{1}\to \R$ is proper except over some interval $[0,\sigma_{0}]$. Let:
\begin{equation*}
  P_{1}=\sigma^{-1}([0,\sigma_{0}])\text{ and }P_{2}=\sigma^{-1}([\sigma_{0},\infty)).
\end{equation*}
We conclude by the same argument as \S\ref{sec:defin-chain-homot} that:
\begin{equation*}
  \ip{\mathrm{LE}(\psi_{t},L),\tau^{a}x}=\#\bd P_{1}+\#\bd P_{2}.
\end{equation*}
Counting $\mathscr{M}_{0}(x;a;\kappa_{s,t},L)$ defines an element $\mathrm{K}$ so that:
\begin{equation*}
  \#\bd P_{1}=\ip{d\mathrm{K},\tau^{a}x}.
\end{equation*}
The following lemma will complete the proof:
\begin{lemma}\label{lemma:LE-PSS}
  For a generic almost complex structure $J$ and a generic $\delta_{s,t}$, both depending on $L$, it holds that: $$\# \bd P_{2}=\ip{\mathrm{PSS}(\psi_{t},L),\tau^{a}x}$$ where $\mathrm{PSS}(\psi_{t},L)$ is the PSS cycle defined in \S\ref{sec:pss-elements}. In particular, $\mathrm{PSS}(\psi_{t},L)$ and $\mathrm{LE}(\psi_{t},L)$ are the same element in $\mathrm{HF}(\psi_{t})$.
\end{lemma}

The genericity on the almost complex structure is used to achieve transversality for somewhere injective $J$-holomorphic disks with boundary on $L$, among other things. However, the equality $\mathrm{PSS}(\psi_{t},L)=\mathrm{LE}(\psi_{t},L)$ holds for all admissible almost complex structures $J$, as can be proved using similar continuation arguments as in, e.g., \cite{hofer-salamon-95}.

\subsubsection{Proof of Lemma \ref{lemma:LE-PSS}}
\label{sec:proof-lemma-le-pss}

First we will examine the possible failures in compactness along sequences $(\sigma_{n},u_{n})\in P_{2}$ as $\sigma_{n}\to\infty$. Standard compactness theory, as in, e.g., \cite{cant-chen-arXiv-2023}, imply that $u_{n}$ breaks/bubbles into:
\begin{enumerate}
\item some number of Floer differential cylinders $v_{1},\dots,v_{p}$,
\item a solution $u_{+}$ to $\mathscr{M}_{+}(\kappa_{s,t})$,
\item a sequence of holomorphic spheres $w_{1},\dots,w_{q}$,
\item a $J$-holomorphic disk $w_{-}$ with boundary on $L$,
\item some number of other holomorphic disks on $L$ or spheres.
\end{enumerate}
This limit objects satisfy the incidence condition that:
\begin{equation*}
  \text{$u_{+}(\infty)=w_{1}(-\infty)$, $w_{i}(+\infty)=w_{i+1}(-\infty)$, $w_{k}(+\infty)=w_{-}(-\infty)$.}
\end{equation*}
Moreover, the total intersection number of all the components with $\mu_{\mathfrak{m}}(L)$ equals $\mathrm{CZ}(\gamma)$, where $\gamma$ is the orbit starting at $x$.

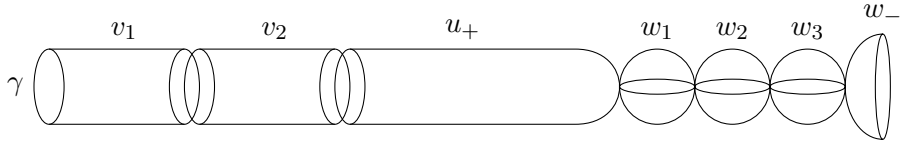
\begin{figure}[h]
  \centering
  \begin{tikzpicture}
    \begin{scope}[shift={(-4,0)}]
      \node at (-0.2,0.5) [left] {$\gamma$};
      \draw (0,0.5) circle (0.2 and 0.5);
      \draw (0,1)--+(1.8,0) (0,0)--+(1.8,0);
      \draw (1.8,0.5) circle (0.2 and 0.5);
    \end{scope}
    \begin{scope}[shift={(-2,0)}]
      \draw (0,0.5) circle (0.2 and 0.5);
      \draw (0,1)--+(1.8,0) (0,0)--+(1.8,0);
      \draw (1.8,0.5) circle (0.2 and 0.5);
    \end{scope}
    \begin{scope}
      \draw (0,0.5) circle (0.2 and 0.5);
      \draw (0,1)--(3,1)to[out=0,in=0,looseness=2]coordinate[pos=0.5](X)(3,0)--(0,0);
      \node at (X){};
    \end{scope}
    \draw (X) +(0.5,0) circle (0.5 and 0.1) circle (0.5 and 0.5);
    \draw (X) +(1.5,0) circle (0.5 and 0.1) circle (0.5 and 0.5);
    \draw (X) +(2.5,0) circle (0.5 and 0.1) circle (0.5 and 0.5);
    \draw (X) +(3.5,0) circle (0.1 and 0.7) +(3.5,0.7) arc (90:270:0.5 and 0.7);
    \path[every node/.style={above}] (-1,1)node{$v_{2}$}(-3,1)node{$v_{1}$}(1.5,1)node{$u_{+}$} (X)--+(0.5,0.5)node{$w_{1}$}--+(1.5,0.5)node{$w_{2}$}--+(2.5,0.5)node{$w_{3}$}--+(3.5,0.7)node{$w_{-}$};
  \end{tikzpicture}
  \caption{Limit of the sequence $u_{n}$ as $n\to\infty$. Not shown are any bubble trees which may have formed.}
\end{figure}

Similarly to Lemma \ref{lemma:fancy-compactness}, we claim that there exists an underlying simple nodal disk:
\begin{equation*}
  w:(\Sigma,\bd \Sigma)\to (W,L),
\end{equation*}
where $\Sigma=\mathbb{C}P^{1}\sqcup \dots \sqcup \mathbb{C}P^{1}\sqcup D(1)$, whose image contains the point $u_{+}(+\infty)$. Moreover, the intersection number of $w$ with $\mathfrak{m}^{-1}(0)$ is at most the total intersection of $w_{1},\dots,w_{q},w_{-}$ with $\mathfrak{m}^{-1}(0)$. The existence of this underlying simple curve essentially follows from the results of \cite{lazzarini-GAFA-2000,biran-cornea-GT-2009} which generalize the case of the closed case in, e.g., \cite[Chapter 2]{mcduffsalamon}. The case when there are no holomorphic spheres is literally proved in \cite{lazzarini-GAFA-2000,biran-cornea-GT-2009}.

To handle the case when there are holomorphic spheres, the key idea is:
\begin{lemma}\label{lemma:lazzarini}
  If a union of a simple sphere $w_{1}$ and a holomorphic disk $w_{-}$ with boundary on $L$ is \emph{not} simple, then every point in the image $w_{1}$ is contained in the image of a simple holomorphic disk with boundary on $L$.
\end{lemma}
\begin{proof}
  This follows from Lazzarini's decomposition theorems \cite{lazzarini-GAFA-2000}. See also \cite[Theorem 3.3.1]{biran-cornea-GT-2009}.
\end{proof}
To use this lemma, first replace $w_{1}\sqcup \dots\sqcup w_{q}$ by an underlying simple curve, as in Lemma \ref{lemma:fancy-compactness}. Then use Lemma \ref{lemma:lazzarini} to take shortcuts in the sequence if we the curve obtained by adding $w_{-}$ is not simple. Henceforth, let us therefore suppose that the coproduct:
\begin{equation*}
  w_{1}\sqcup \cdots \sqcup w_{q}\sqcup w_{-}
\end{equation*}
is a simple map. Note that we may need to replace the incidence condition by $w_{q}(+\infty)=w_{-}(i)$ where $i=1$ (or $i=0$) when passing to the underlying simple nodal disk.

Now consider the moduli space of tuples $(v_{1},\dots,v_{p},u_{+},w_{1},\dots,w_{q},w_{-})$ of the above type. We pick $J,\delta_{s,t}$ so that the total evaluation map:
\begin{equation*}
  (u_{+}(+\infty),w_{1}(-\infty),w_{1}(\infty),\dots,w_{q}(-\infty),w_{q}(\infty),w_{+}(i)),
\end{equation*}
is tranverse to the set $\Sigma$ of tuples $(y_{0},x_{1},y_{1},\dots,x_{k},y_{k},x_{k+1})$ satisfying the incidence $y_{i}=x_{i+1}$. The map is valued in either $W\times \dots \times W\times W$ or $W\times \dots \times W\times L$ depending on whether $i=0,1$.

The dimension of this moduli space (without the incidence condition) is:
\begin{equation*}
  2n(q+1)+\mu_{\mathfrak{m}}(L)\cdot([v_{1}]+\dots+[v_{p}]+[u_{+}]+[w_{1}]+\dots+[w_{-}])-\mathrm{CZ}_{\mathfrak{m}}(\gamma).
\end{equation*}
In particular, the inverse image of $\Sigma$ has dimension equal to:
\begin{equation*}
  \mu_{\mathfrak{m}}(L)\cdot [u]-\mathrm{CZ}_{\mathfrak{m}}(\gamma),
\end{equation*}
where $[u]$ is the class obtained by summing all of the components. If such a configuration occurs in the limit of $(\sigma_{n},u_{n})\in P_{2}$ as $\sigma_{n}\to\infty$, then the expected dimension of the inverse image of $\Sigma$ is therefore at most \emph{zero}. If there were any bubble trees which formed during the process, then the dimension would strictly negative, and hence the inverse image of $\Sigma$ would be empty.

The reparametrization group:
\begin{equation*}
  \R^{p}\times \mathrm{Aut}(\C^{\times})^{k}\times \mathrm{Aut}(D(1)\text{ fixing point }z=i)
\end{equation*}
acts on the inverse image of $\Sigma$. This implies the inverse image of $\Sigma$ is empty, unless $p=k=0$ and $w_{+}$ is constant. Thus we conclude that $u_{n}$ converges to $v_{-}$ and $v_{-}(0)$ lies on the Lagrangian.

Standard gluing results imply that any such PSS solution $u_{+}$ with area $a$, asymptotic $\gamma$ satisfying $\gamma(0)=x$, and such that $u_{+}(+\infty)$ lies on $L$ can be glued to form a non-compact end of solutions $(\sigma,u)$ in $\mathscr{M}_{1}(x;a;\kappa_{s,t},L)$ which converges to $v_{-}$. In this manner we conclude that the count of the non-compact ends of $P_{2}$ equals $\ip{\mathrm{PSS}(L),\tau^{a}x}$, as desired.\hfill$\square$

\subsection{Non-vanishing of the Lagrangian element}
\label{sec:non-vanish-lagr}

So far, everything in \S\ref{sec:lagr-element}, \S\ref{sec:natur-lagr-elem}, and \S\ref{sec:extends-PSS} has only used that $L$ is a compact monotone Lagrangian with minimal Maslov number at least two. In this section, we will use the existence of the other Lagrangian $L'$ to conclude that the colimit of $\mathrm{LE}(\psi_{t},L)$ over $\psi_{t}\in \mathscr{C}$ is non-vanishing in $\mathrm{SH}_{e}$.

\subsubsection{The Lagrangian evaluation map}
\label{sec:lagrangian-evaluation-map}

Given an input system $\psi_{t}$, consider the moduli space $\mathscr{M}(L',\psi_{t})$ of (finite energy) solutions $u$ to:
\begin{equation}\label{eq:moduli-space-EL}
  \left\{
    \begin{aligned}
      &u:[0,\infty)\times \R/\Z\to W,\\
      &\bd_{s}u+J(u)(\bd_{t}u-X_{t}(u))=0,\\
      &u(0,t)\in L',
    \end{aligned}
  \right.
\end{equation}
where $X_{t}$ is the generator of $\psi_{\beta(3t-1)}$ as usual. This moduli space satisfies the same tranversality/compactness properties as the moduli space used to define the Lagrangian element (indeed, it is simply the $s\mapsto-s$ reflected version, with $L$ replaced by $L'$).

\begin{figure}[h]
  \centering
  \begin{tikzpicture}[xscale=-1]
    \draw (0,0) coordinate(X) --+ (5,0) +(0,1) --+ (5,1);
    \draw (X)+(0,0.5) circle(0.2 and 0.5) +(-0.2,0.5) node[right]{$\gamma$} +(5,0.5) circle (0.2 and 0.5) +(5.2,0.5) node[left]{$L'$};
  \end{tikzpicture}
  \caption{The moduli space $\mathscr{M}(L',\psi_{t})$ used to define the Lagrangian evaluation map.}
\end{figure}
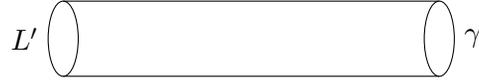

Let us pick $\mathfrak{m}$ similarly to \S\ref{sec:lagr-element} so that $\mathfrak{m}|_{\gamma}$ is non-vanishing for all orbits of $\psi_{t}$, and so that $\mathfrak{m}|_{L'}$ directs the canonical direction of $L'$.

Let $\mathscr{M}_{d}(x;a;L',\psi_{t})$ be the component of solutions $u$ so that:
\begin{enumerate}
\item $\gamma(0)=x$ where $\gamma$ is the asymptotic orbit of $u$,
\item $\omega(u)=a$,
\item $\mathrm{CZ}(\gamma)+\mu_{\mathfrak{m}}(L')\cdot [u]=d$.
\end{enumerate}

Similarly to \S\ref{sec:lagr-element}, for generic perturbations (replacing $\psi_{t}=\psi_{t}\delta_{t}$), the moduli space $\mathscr{M}_{d}$ is a $d$-dimensional manifold, and the zero dimensional component with $\omega(u)\le A$ is a finite set of points.

We can therefore define a map $\mathrm{CF}(\psi_{t})\to \Z/2$ by the formula:
\begin{equation*}
  \mathrm{EL}(L',\psi_{t})(\tau^{a}x)=\#\mathscr{M}_{0}(x;-a,L',\psi_{t}).
\end{equation*}
This map is well-defined, i.e., if it is applied to a semi-infinite sum $\sum c_{i}\tau^{a_{i}}x_{i}$ then:
\begin{equation*}
  \mathrm{EL}(L',\psi_{t})(\sum c_{i}\tau^{a_{i}}x_{i})=\sum_{c_{i}\ne 0}\#\mathscr{M}_{0}(x_{i};-a_{i};L',\psi_{t}).
\end{equation*}
has only finitely many non-zero terms appearing in the right-hand side. Indeed, the energy of a solution $u$ to \eqref{eq:moduli-space-EL} is equal to:
\begin{equation*}
  \omega(u)+(\text{bounded quantity depending on $H_{t}|_{L'}$, $H_{t}|_{\gamma}$}).
\end{equation*}
Thus, $\mathscr{M}_{0}(x_{i};-a_{i};L',\psi_{t})=\emptyset$ if $a_{i}$ is too large. On the other hand, $a_{i}$ cannot be too negative, because such sums are not allowed in $\mathrm{CF}(\psi_{t})$.

Arguments similar to those used to prove $\mathrm{LE}(\psi_{t},L)$ is a cycle prove that $\mathrm{EL}(L',\psi_{t})$ is a chain map. The rest of this section is devoted to proving:
\begin{lemma}\label{lemma:EL-LE-1}
  For generic system $\psi_{t}$, it holds that:
  \begin{equation*}
    \mathrm{EL}(L',\psi_{t})(\mathrm{LE}(\psi_{t},L))=\#(L\cap L')=1;
  \end{equation*}
  where the intersection number is computed mod 2.
\end{lemma}

Consequently, $\mathrm{LE}(\psi_{t},L)\ne 0$ in $\mathrm{HF}(\psi_{t})$, for every generic $\psi_{t}$ (in fact, this holds for all $\psi_{t}$ by a limiting process). It then follows easily that the projection of $\mathrm{LE}(\psi_{t},L)$ to $\mathrm{SH}$ is non-zero. Thus, to prove Theorem \ref{theorem:odd-euler-char-1} it remains only to prove Lemma \ref{lemma:EL-LE-1}.

\subsubsection{Moduli space of finite length Floer cylinders}
\label{sec:moduli-space-finite}

The proof of Lemma \ref{lemma:EL-LE-1} uses the moduli space $\mathscr{M}(L',\psi_{t},L)$ of pairs $(\sigma,u)$ such that:
\begin{equation*}
  \left\{
    \begin{aligned}
      &u:[0,\sigma]\times \R/\Z\to W,\\
      &\bd_{s}u+J(u)(\bd_{t}u-X_{t}(u))=0,\\
      &u(0,t)\in L'\text{ and }u(\sigma,t)\in L,
    \end{aligned}
  \right.
\end{equation*}
where $X_{t}$ is exactly as in \S\ref{sec:lagrangian-evaluation-map}. Here $\sigma$ should be thought of as varying over $(0,\infty)$. See \cite{brocic-cant-JFPTA-2024} for detailed analysis of a similar moduli space in the case when $W$ is a cotangent bundle.

Let $\mathfrak{m}$ be a section of $\det_{\C}(TW)^{\otimes 2}$ so that:
\begin{enumerate}
\item $\mathfrak{m}|_{\gamma}$ is non-vanishing along any orbit $\gamma$ of $\psi_{t}$,
\item $\mathfrak{m}|_{L}$ and $\mathfrak{m}|_{L'}$ are non-vanishing and never point opposite the canonical direction.
\end{enumerate}
The second condition implies that $\mathfrak{m}|_{L},\mathfrak{m}|_{L'}$ can be (separately) homotoped through non-vanishing sections to ones which point in the canonical direction. The second property can be achieved by working in local coordinates near the transverse intersections of $L$ and $L'$. Define:
\begin{equation*}
  \mu_{\mathfrak{m}}(L',L)=\mathfrak{m}^{-1}(0).
\end{equation*}
The Poincaré dual of this cycle represents the Maslov classes of $L'$ and $L$.

Introduce the component $\mathscr{M}_{d}(0;L',\psi_{t},L)$ of solutions $(\sigma,u)$ such that:
\begin{enumerate}
\item $\omega(u)=0$,
\item $1+\mu_{\mathfrak{m}}(L',L)\cdot [u]=d$.
\end{enumerate}
Similar (but simpler) arguments to the ones in \S\ref{sec:natur-lagr-elem} imply that:
\begin{lemma}\label{lemma:ELLE-proper}
  For generic $\psi_{t}$, the parametric moduli space $\mathscr{M}_{1}(0;L',\psi_{t},L)$ is a smooth $1$-manifold and the map $\sigma:\mathscr{M}_{1}(0;L',\psi_{t},L)\to(0,\infty)$ is proper.
\end{lemma}
\begin{proof}
  The transversality mostly follows Lemma \ref{lemma:transversality-floer}, however there is one subtle difference. In Lemma \ref{lemma:transversality-floer} we had to exclude the possibility that: $$u(s+s_{0},t)=u(s,t)$$ held for some $s_{0}$. Typically this is prevented by appealing to finiteness of energy. However, in the case of a finite length cylinder one cannot argue in the same way. Let us continue the argument assuming that $u$ is periodic in this fashion, and pick $s_{0}$ to be the minimal period. This decomposes the domain into subcylinders:
  \begin{equation*}
    \Sigma_{1}\cup \dots \cup \Sigma_{d+1}=[0,s_{0}]\times \R/\Z \cup \cdots \cup [ds_{0},\sigma]\times \R/\Z.
  \end{equation*}
  where $\abs{ds_{0}-\sigma}<s_{0}$; we adopt the convention that $\Sigma_{d+1}=\emptyset$ if $ds_{0}=\sigma$.

  First we claim that $\Sigma_{d+1}\ne \emptyset$. By the energy identity, and our assumption of monotonicity on the Lagrangian $L'$, the restriction of $u$ to $\Sigma_{1}$ has energy equal to $\omega(u)>0$ so $u|_{\Sigma_{1}}$ intersects the Maslov class $\mu_{\mathfrak{m}}(L',L)$ a positive number of times. In particular, the restriction of $u$ to $\Sigma_{1}\cup \dots\cup \Sigma_{d}$ has a strictly positive Maslov number and hence does not lie in $\mathscr{M}_{1}$.

  Thus $\Sigma_{d+1}\ne \emptyset$. The restriction of $u$ to $\Sigma_{d+1}$ is not periodic, and hence we can use \cite{floer_hofer_salamon_transversality} (see also Lemma \ref{lemma:transversality-floer}) to achieve generic transversality for the restriction of $u$ to $\Sigma_{d+1}$. However, by the above argument, the Maslov number of the restriction to $\Sigma_{d+1}$ must be strictly negative. In fact the Maslov number the restriction to $\Sigma_{d+1}$ is at most $-2$. Since we have achieved transversality for these ``simple'' cylinders, $u$ with such an index cannot exist. This contradiction proves a solution for generic data cannot be multiply covered in the above sense, and we achieve transversality for all solutions in $\mathscr{M}_{1}$.

  The rest of the argument follows the same lines as Lemma \ref{lemma:proper-natur-LE} and so we skip the proof. In fact, the argument is simpler because there is no asymptotic end on the domain to worry about.
\end{proof}

\subsubsection{Proof of Lemma \ref{lemma:EL-LE-1}}
\label{sec:proof-lemma-refl}

It follows from Lemma \ref{lemma:ELLE-proper} that the number of non-compact ends of $\mathscr{M}_{1}(0;L',\psi_{t},L)$ is even, and moreover divides into two kinds of ends:
\begin{enumerate}[label=(E\arabic*)]
\item\label{item:E1} ends containing sequences $(\sigma_{n},u_{n})$ with $\sigma_{n}\to 0$,
\item\label{item:E2} ends containing sequences $(\sigma_{n},u_{n})$ with $\sigma_{n}\to \infty$.
\end{enumerate}
Standard gluing arguments, quite similar to those in \cite{brocic-cant-JFPTA-2024} prove that the number of ends of type \ref{item:E1} is equal to $\#(L'\cap L)$. On the other hand, analysis of the breakings along ends of type \ref{item:E2} shows that $u_{n}$ must converge in the Floer theory sense to a configuration of $(u_{-},u_{+})$ such that:
\begin{enumerate}
\item $u_{-}\in \mathscr{M}_{0}(x,-a;L',\psi_{t})$,
\item $u_{+}\in \mathscr{M}_{0}(x,a;\psi_{t},L)$,
\end{enumerate}
for some $x$ in the fixed point set. There are only finitely many such possible breakings. The count of such hypothetical breakings is clearly equal to:
\begin{equation*}
  \mathrm{EL}(L',\psi_{t})(\mathrm{LE}(\psi_{t},L)),
\end{equation*}
and gluing analysis proves each configuration $(u_{-},u_{+})$ of the above type arises as a genuine non-compact end of $\mathscr{M}_{1}(L',\psi_{t},L)$ of type \ref{item:E2}. Since the number of ends of type \ref{item:E1} equals the number of ends of type \ref{item:E2} we conclude the desired result.\hfill$\square$

\bibliographystyle{alpha}
\bibliography{citations}
\end{document}